\documentclass[a4paper,11pt,reqno,makeidx]{amsart}
\usepackage{amsfonts,amssymb,amscd,amsmath,latexsym,amsbsy,enumerate}
\usepackage{stmaryrd,a4wide,mathrsfs,graphicx}
\usepackage[pagebackref,hypertexnames=false, colorlinks, citecolor=black, linkcolor=blue]{hyperref}

\newtheorem{thm}{Theorem}[section]
\newtheorem{cor}[thm]{Corollary}
\newtheorem{lemma}[thm]{Lemma}
\newtheorem{prop}[thm]{Proposition}
\newtheorem{defn}[thm]{Definition}
\theoremstyle{definition}
\newtheorem{remark}[thm]{Remark}

\numberwithin{equation}{section}

\def\hf{{\frac{1}{2}}}
\def\al{\alpha}
\def\be{\beta}
\def\ga{\gamma}
\def\de{\delta}
\def\ep{\varepsilon}
\def\la{\lambda}

\def\si{\sigma}
\def\vp{\varphi}

\def\De{\Delta}

\def\Ga{\Gamma}
\def\Z{\mathbb{Z}}
\def\R{\mathbb{R}}
\def\C{\mathbb{C}}
\def\N{\mathbb{N}}
\def\T{\mathbb{T}}
\def\cD{\mathcal{D}}
\def\cH{\mathcal{H}}

\def\cP{\mathcal{P}}

\def\cO{\mathcal{O}}
\def\cL{\mathcal{L}}
\def\cF{\mathcal{F}}
\def\Id{\text{\rm{Id}}}
\def\SU{\mathrm{SU}}

\def\Ker{\mathrm{Ker}}
\def\End{\mathrm{End}}

\def\free{\mathrm{free}}

\newcommand{\rphis}[5]{\,_{#1}\vp_{#2} \left( \genfrac{.}{.}{0pt}{}{#3}{#4}
\ ;#5 \right)}
\newcommand{\rFs}[5]{\,_{#1}F_{#2} \left( \genfrac{.}{.}{0pt}{}{#3}{#4}
\ ;#5 \right)}
\newcommand{\binomial}[2]{\left(\genfrac{.}{.}{0pt}{}{#1}{#2}\right)}
\newcommand{\qbinomial}[3]{\left[\genfrac{.}{.}{0pt}{}{#1}{#2}\right]_{#3}}

\makeindex

\begin{document}
\title[$q$-special functions]
{$q$-special functions, basic hypergeometric series\\ and operators}

\author{Erik Koelink}
\address{Radboud Universiteit, IMAPP, FNWI, PO Box 9010, 6500 GL Nijmegen,
the Netherlands}
\email{e.koelink@math.ru.nl}
\date{Final version, August 10, 2018. Lecture notes for June 2018 summer school OPSFA-S8,
Sousse, Tunisia.  \\
\hspace*{10pt} \texttt{http://www.essths.rnu.tn/OPSF-S8/acceuil.html}}

\begin{abstract}
In the lecture notes we start off with an introduction to the $q$-hypergeometric series, or basic hypergeometric series, 
and we derive some elementary summation and transformation results.
Then the $q$-hypergeometric difference equation is studied, and in particular we study solutions 
given in terms of power series at $0$ and at $\infty$. Factorisations of the corresponding operator
are considered in terms of a lowering operator, which is the $q$-derivative, and the 
related raising operator. 
Next we consider the $q$-hypergeometric operator in a special case, and we show that there
is a natural Hilbert space --a weighted sequence space-- on which this operator is symmetric.
Then the corresponding eigenfunctions are polynomials, which are the little $q$-Jacobi polynomials.
These polynomials form a family in the $q$-Askey scheme, and so many important properties are
well known. In particular, we show how the orthogonality relations and the three-term recurrence
for the little $q$-Jacobi polynomials can be obtained using only the factorisation of the 
corresponding operator. 
As a next step we consider the $q$-hypergeometric operator in general, which leads to 
the little $q$-Jacobi functions. We sketch the derivation of the corresponding 
orthogonality using the connection between various eigenfunctions. 
The link between the $q$-hypergeometric operators with different parameters is
studied in general using $q$-analogues of fractional derivatives, and this gives 
transmutation properties for this operator. 
In the final parts of these notes we consider partial extensions of this approach to 
orthogonal polynomials and special functions. The first extension is a brief introduction to 
the Askey-Wilson functions and the corresponding integral transform.
The second extension is concerned with a matrix-valued extension of the 
$q$-hypergeometric difference equation and its solutions.

\end{abstract}

\maketitle
\tableofcontents

%%%%%%%%%%%%%%%%%%%%%%%%%%%%%%%%%%%%%%%%%%%%%%%%%%%%%%%%%%%%%%%%%%%
%%%%%NEW SECTION%%%%%%%%%%%%%%%%%%%%%%%%%%%%%%%%%%%%%%%%%%%%%%%%%%%
%%%%%%%%%%%%%%%%%%%%%%%%%%%%%%%%%%%%%%%%%%%%%%%%%%%%%%%%%%%%%%%%%%%

% \section*{Preamble}

%%%%%%%%%%%%%%%%%%%%%%%%%%%%%%%%%%%%%%%%%%%%%%%%%%%%%%%%%%%%%%%%%%%
%%%%%NEW SECTION%%%%%%%%%%%%%%%%%%%%%%%%%%%%%%%%%%%%%%%%%%%%%%%%%%%
%%%%%%%%%%%%%%%%%%%%%%%%%%%%%%%%%%%%%%%%%%%%%%%%%%%%%%%%%%%%%%%%%%%

\section{Introduction}\label{sec:Introduction}

Basic hypergeometric series have been introduced a long time ago, and 
important contributions go back to 
Euler, Heine, Rogers, Ramanujan, etc. 
The importance and the history of the basic hypergeometric series is clearly indicated in 
Askey's foreword to the book on basic hypergeometric series by Gasper and Rahman \cite{GaspR}.
Since the work of Askey, Andrews, Ismail, and coworkers many new results on 
classes of special functions and orthogonal polynomials in terms of
basic hypergeometric series have been obtained. The relation 
to representation theory of quantum groups and related structures in e.g.
mathematical physics and combinatorics has given the topic a new boost in the recent
decades.

In these lecture notes we focus on the basic hypergeometric series of type 
${}_2\varphi_1$ by studying the corresponding $q$-difference operator
to which these series are eigenfunctions. 
The study of general $q$-difference operators go back to Birkhoff and 
Trjitzinsky in the 1930s. 
In \S \ref{sec:BHS} we first introduce the basic hypergeometric series,
and we derive some elementary summation and transformation formulas needed in the sequel.
However, we will not prove all the necessary transformation formulas, but refer
to Gasper and Rahman \cite{GaspR} when necessary. Section \ref{sec:BHS} 
is based on the book \cite{GaspR} by Gasper and Rahman, which is the 
basic reference for basic hypergeometric series. 
In \S \ref{sec:BHS-qdiff} we then discuss the corresponding 
$q$-difference operator in more detail, by studying the solutions obtained 
by Frobenius's method. We also look at the decomposition of the operator
using the standard $q$-difference operator.
Next, in \S \ref{sec:BHS-qdiff-pol} we consider a special case of the 
$q$-difference operator, namely the one which can be related to 
polynomial eigenfunctions for functions supported on $q^\N$. 
This essentially leads to the little $q$-Jacobi polynomials, for 
which we derive the natural orthogonality measure, the corresponding
orthogonality relations, and the three-term recurrence relation 
by using the shift operators. These shift operators are the 
operators in factorisations of the difference operator. 

In \S \ref{sec:BHS-qdiff-nonpol} we study a more general case.
This leads to general orthogonality for ${}_2\varphi_1$-series, which 
we derive by calculating the spectral measure of the corresponding
measure. In \S \ref{sec:transmutation} we study the 
transmutation properties of the basic hypergeometric $q$-difference
operator. 
In \S \ref{sec:AW-level} we then lift this to the level of Askey-Wilson 
polynomials and the Askey-Wilson functions.  

In \S \ref{sec:MVextensions} we make a first start in order to lift the 
results on little $q$-Jacobi polynomials of \S \ref{sec:BHS-qdiff-pol} 
and little $q$-Jacobi functions of \S \ref{sec:BHS-qdiff-nonpol} 
to the matrix-valued extensions. 

There are many related results available in the literature, and 
we indicate several developments in the notes to each section. 
In particular, it is not clear if the results of \S \ref{sec:BHS-qdiff-pol} 
and \S \ref{sec:BHS-qdiff-nonpol} can be extended to the level of the 
Askey-Wilson functions  as in \S \ref{sec:AW-level} or the matrix-valued analogues of 
\S \ref{sec:MVextensions}.

By $\N$ we denote the natural numbers starting at $0$. 
The standing assumption on $q$ is $0<q<1$. 

\medskip
\textbf{Acknowledgement.} I am much indebted to the organisers,
Hamza Chaggara, Frej Chouchene, Imed Lamiri, Neila Ben Romdhane, Mohamed Gaied, of the 
summer school for their work on the summer school and their kind hospitality. 
I also thank the participants for their interest and discussions. 
The main sections of the lecture notes are based on previous papers \cite{KoelR-RMJM02}, 
\cite{KoelS-IMRN2001},
joint with Hjalmar Rosengren and Jasper Stokman, respectively. 
The lecture notes do not contain any new results, except that the description of the 
solutions of the matrix-valued $q$-hypergeometric series in \S \ref{sec:MVextensions}
have not appeared before. These solutions have been determined by Nikki Jaspers in 
her BSc-thesis \cite{Jasp} under supervision of Pablo Rom\'an and the author.

\section{Basic hypergeometric series}\label{sec:BHS}

The basic hypergeometric series are analogues of the much better known 
hypergeometric series and hypergeometric functions. 
The hypergeometric series ${}_2F_1(a,b;c;z)$ as well as the
analogous Thomae series ${}_{r+1}F_r$ and the more general
hypergeometric ${}_rF_s$-series are discussed in detail
in e.g. \cite{AndrAR}, \cite{Bail}, \cite{Isma}, \cite{KoekLS}, \cite{KoekS}, \cite{Rain}, \cite{Slat}, \cite{Temm}
and many other standard textbooks. 
Recall  the notation for that standard hypergeometric function 
\begin{equation}\label{eq:def2F1Pochhammer}
% \begin{split} &\qquad\qquad 
{}_2F_1(a,b;c;z)\, = \, 
\rFs{2}{1}{a,b}{c}{z} \, =\, \sum_{n=0}^\infty \frac{(a)_n(b)_n}{(c)_n\, n!} z^n, 
% \\ &(a)_0 =1,\; (a)_n = a(a+1) \cdots (a+n-1) = \frac{\Gamma
% (a+n)}{\Gamma (a)},\quad n=1,2,\ldots \ .
% \end{split}
\end{equation}
for this series (and for its sum when it converges) assuming
$c\not= 0,-1,-2,\cdots$. This
is the \emph{(ordinary) hypergeometric
series}\index{hypergeometric series} or
the \emph{Gauss hypergeometric series}. The series converges absolutely for $|z|<1$, and
for $|z| = 1$ when $\Re(c-a-b) >0$, see Exercise \ref{ex:Raabe2F1}. 
Many important functions, such as the logarithm, arcsin, exponential,
classical orthogonal polynomials can be expressed in terms of
Gauss hypergeometric series.
 $(a)_n$ denotes the \emph{shifted
factorial}\index{shifted factorial} or \emph{Pochhammer symbol}\index{Pochhammer symbol} or 
\emph{raising factorial}\index{raising factorial} defined by
\begin{equation}\label{eq:1.2.8}
(a)_0 =1,\qquad  (a)_n = a(a+1) \cdots (a+n-1) = \frac{\Gamma
(a+n)}{\Gamma (a)},\quad n=1,2,\ldots \ .
\end{equation}
More generally, one can define hypergeometric series with more parameters. 

Around the mid 19th-century Heine  
introduced the series
\begin{equation}\label{eq:1.2.12}
1 + \frac{(1-q^a)(1-q^b)}{(1-q)(1-q^c)}z
+\frac{(1-q^a)(1-q^{a+1})(1-q^b)(1-q^{b+1})}
{(1-q)(1-q^2)(1-q^c)(1-q^{c+1})}\;z^2 + \cdots ,
\end{equation}
where it is assumed that $q \ne 1$, $c\not= 0, -1, -2, \ldots$ and the principal value of each power of $q$ is taken.
This series converges absolutely for $|z|<1$ when $|q| <1$ and it
tends termwise to Gauss' series as
$q\rightarrow 1$, because
\begin{equation}\label{eq:1.2.13}
\lim_{q\to 1} \frac{1-q^a}{1-q} = a.
\end{equation}
The ratio $(1-q^a)/(1-q)$ considered in \eqref{eq:1.2.13} is called a $q$-\emph{number}\index{q@$q$-number} (or \emph{basic number}) and it is denoted by $[a]_q$.
One should realise that other notations for $q$-numbers, such as $\frac{q^a-q^{-a}}{q-q^{-1}}$, are also in use.  
It is also called a $q$-analogue, $q$-deformation, $q$-extension, or a $q$-generalization of the complex number $a$. In terms of $q$-numbers the $q$-\emph{number factorial} $[n]_q!$ is defined for a nonnegative integer $n$ by $[n]_q! = \prod^n_{k=1} [k]_q$,
and the corresponding $q$-\emph{number shifted factorial} is defined by
$[a]_{q;n} = \prod^{n-1}_{k=0} [a+k]_q$.
Clearly,
$[a]_{q;n} = (1-q)^{-n} (q^a;q)_n$, with the notation \eqref{eq:1.2.15}  and  $\lim_{q\to 1} [a]_{q;n} = 
(a)_n$. 
The series in \eqref{eq:1.2.12} is usually called \emph{Heine's series} or,
in
view of the base $q$, the \emph{basic hypergeometric series} or
$q$-\emph{hypergeometric series}, or simply a $q$-\emph{series}.

%%%%%%%%%%%%%%%%%%%%%%%%%%%%%%%%%%%%%%%%%%%%%%%%%%%%%%%%%%%%%%%%%%%
%%%%%NEW SECTION%%%%%%%%%%%%%%%%%%%%%%%%%%%%%%%%%%%%%%%%%%%%%%%%%%%
%%%%%%%%%%%%%%%%%%%%%%%%%%%%%%%%%%%%%%%%%%%%%%%%%%%%%%%%%%%%%%%%%%%

\subsection{Notation for basic hypergeometric series}\label{ssec:BHS-notation}

Analogous to Gauss's notation for the hypergeometric function, Heine used the notation 
$\vp(a,b,c,q,z)$ for his series.
However, since one would like to also be able to consider the
case when $q$ to the
power $a,b,$ or $c$ is replaced by zero, it is now customary to
define the \emph{basic hypergeometric series}\index{basic hypergeometric series}\index{q@$q$-hypergeometric series}
\begin{equation}\label{eq:1.2.14}
\begin{split}
\vp (a,b;c;q,z) \equiv \; {}_2\vp_1 (a,b;c;q,z)
\equiv 
\rphis{2}{1}{a,b}{c}{q,z}
%\\[3pt]
=\; \sum_{n=0}^{\infty} \frac{(a;q
)_n(b;q)_n}{(q;q)_n(c;q)_n}\;z^n, 
\end{split}
\end{equation}
 where
\begin{equation}\label{eq:1.2.15}
(a;q)_n = \begin{cases} 1,& n=0,\\
 (1-a)(1-aq)\cdots(1-aq^{n-1}),& n=1,2,\ldots ,
 \end{cases}
 \end{equation}
is the $q$-\emph{shifted factorial}\index{q@$q$-shifted factorial} and for general $a$ and $b$ it is assumed
that
$c\not= q^{-m}$ for $m = 0, 1, \ldots\ $.  Some
other notations that have been used in the literature for the
product $(a;q)_n$ are $(a)_{q,n}, [a]_n$ (not to be confused 
with $[a]_q$).

Unless stated otherwise, when dealing with nonterminating basic
hypergeometric series we
shall assume that $|q| <1$ and that the parameters and
variables are such that the series converges absolutely.  Note
that if $|q|>1$, then we can perform an inversion with
respect to the base by setting $p= q^{-1}$ and using the identity
\begin{equation}\label{eq:1.2.24}
(a;q)_n = (a^{-1};p)_n (-a)^np^{-{\binomial{n}{2}}}
\end{equation}
to convert the series \eqref{eq:1.2.22} to a
similar series in base $p$ with $|p| <1$, see \eqref{eq:ex1.4i}. 
The inverted
series will have a finite radius of convergence if the original
series does.
% The case of $|q|=1$ is not taken into consideration, but 
% we refer to various papers \cite{Bult}, \cite{vdB:Rai:Sto},
% \cite{Lubinsky1999}, \cite{Ruijsenaars1997}
% \cite{RuijsenaarsAW}, \cite{Stokman2005} and references given there.

More generally, we call the series $\sum_{n=0}^\infty u_n$ 
a (unilateral) \emph{hypergeometric series}\index{hypergeometric series} if the quotient $u_{n+1}/u_n$ is
a rational function of $n$. Similarly, a series $\sum_{n=0}^\infty v_n$ 
a \emph{basic hypergeometric series}\index{basic hypergeometric series} (with base $q$) if the quotient $v_{n+1}/v_n$ is
a rational function of $q^n$ for a fixed base $q$. The most general form
of the quotient is 
\begin{equation}\label{eq:1.2.26}
\frac{v_{n+1}}{v_n} = \frac{(1-a_1q^n)(1-a_2q^n)\cdots
(1-a_rq^n)}{(1-q^{n+1}) (1-b_1q^n)\cdots
(1-b_sq^n)}\;\left(-q^n\right)^{1+s-r}z.
\end{equation}
normalising $v_0 =1$.
Generalising Heine's series, we  define an $_r\vp_s$ 
\emph{basic hypergeometric series}\index{basic hypergeometric series} by
\begin{equation}\label{eq:1.2.22}
\begin{split}
{}_r\vp_s (a_1, a_2, \ldots, a_r; b_1,\ldots,b_s;q,z)
&\equiv\; 
\rphis{r}{s}{a_1, a_2,\ldots , a_r}{b_1,\ldots ,b_s}{q,z} \\
&= \sum_{n=0}^{\infty} \frac{(a_1;q)_n (a_2;q)_n\cdots
(a_r;q)_n}{(q;q)_n (b_1 ;q)_n \cdots (b_s ;q)_n} \left[(-1)^n
q^{\binomial{n}{2}}\right]^{1+s-r}z^n
\end{split}
\end{equation}
with $\binomial{n}{2} = n(n-1)/2$, where $q\ne 0$ when $r> s+1$.

\begin{remark}\label{rmk:convergencerphis} If $0 < |q| <1$, the $_r\vp_s$ series converges absolutely
for all $z$ if $r\leq s$ and for $|z| <1$ if $r = s+1$.
This series also converges absolutely if $|q| >1$ and
$|z| < |b_1b_2\cdots b_sq|/|a_1a_2\cdots a_r|$. It diverges for
$z\not= 0$ if $0 < |q|
<1$ and $r>s+1$, and if $|q| >1$ and $|z| > |b_1b_2\cdots b_sq|/|
a_1 a_2 \cdots a_r|$,
unless it terminates.  
\end{remark}

Since products of $q$-shifted factorials occur so often, to
simplify them we shall frequently
use the more compact notations
\begin{equation}\label{eq:1.2.41} 
(a_1,a_2,\ldots, a_m;q)_n = (a_1;q)_n (a_2;q)_n\cdots (a_m;q)_n, 
\qquad n\in\mathbf{N}. %\cup\{\infty\}.
\end{equation}

As is customary, the
$_r\vp_s$ notation is also used
for the sums of these series inside the circle of convergence and
for their analytic continuations (called \emph{basic hypergeometric functions})
outside the circle of convergence. To switch from base $q$ to base $q^{-1}$
we note 
\begin{equation}\label{eq:ex1.4i}
\rphis{r}{s}{a_1, \ldots, a_r}{b_1, \ldots, b_s}{q,z}
= \sum^\infty_{n=0} \frac{(a^{-1}_1, \ldots, a^{-1}_r; q^{-1})_n}{(q^{-
1}, b^{-1}_1, \ldots, b_s^{-1};q^{-1})_n} \left(\frac{a_1 \cdots
a_rz}{b_1 \cdots b_s q}\right)^n
\end{equation}
assuming the upper and lower parameters are non-zero.

It is important to note that in case one of the upper parameters is of the form $q^{-n}$ for
$n\in\textbf{N}$ the series in \eqref{eq:1.2.22} terminates.  
From now on, unless stated otherwise, whenever $q^{-j},
q^{-k}, q^{-m}, q^{-n}$
appear as numerator parameters in basic series it will be assumed
that $j$, $k$, $m$, $n$, respectively, are nonnegative integers.
For terminating series it is sometimes useful to switch the
order of summation, which is given by
\begin{equation}\label{eq:ex1.4ii}
\begin{split}
\rphis{r+1}{s}{a_1, \ldots, a_r, q^{-n}}{b_1, \ldots,
b_s}{q,z}
= &\frac{(a_1, \ldots, a_r ;q)_n}{(b_1\ldots, b_s;q)_n} \left(\frac{z}{q}\right)^n \left((-1)^n q^{\binomial{n}{2}} \right)^{s-r-1}\\[4pt]
&\times 
\sum^n_{k=0} \frac{(q^{1-n}/b_1, \ldots, q^{1-n}/b_s, q^{-
n};q)_k}{(q, q^{1-n}/a_1, \ldots, q^{1-n}/a_r;q)_k}
\left(\frac{b_1\cdots b_s}{a_1\cdots a_r} \frac{q^{n+1}}{z}\right)^k
\end{split}
\end{equation}
for non-zero parameters.

Observe that the series \eqref{eq:1.2.22} has the property that if we
replace $z$ by $z/a_r$ and let $a_r\rightarrow \infty$, then
the resulting series is again of the form \eqref{eq:1.2.22} with $r$
replaced by $r-1$. Because this is not the case for the
$_r\vp_s$ series defined without the factors
$\left[(-1)^nq^{\binomial{n}{2}}\right]^{1+s-r}$ in the books of Bailey \cite{Bail}  and Slater 
\cite{Slat}
and we wish to be able to handle such limit cases, we have
chosen to define the series ${}_r\vp_s$ as in \eqref{eq:1.2.22}.  There is no loss
in generality since the Bailey and Slater series can be obtained
from the $r=s+1$ case of \eqref{eq:1.2.22} by choosing $s$ sufficiently
large and setting some of the parameters equal to zero.

For negative subscripts, the
$q$-\emph{shifted factorials}\index{q@$q$-shifted factorial} as defined in \eqref{eq:1.2.15} are defined by
\begin{equation}\label{eq:1.2.28}
(a;q)_{-n} = \frac{1}{(1-aq^{-1})(1-aq^{-2})\cdots (1-aq^{-n})} = 
\frac{1}{(aq^{-n};q)_n} = \frac{(-q/a)^nq^{\binomial{n}{2}}}{(q/a;q)_n},
\end{equation}
where $n = 0,1, \ldots\ $. We also define
\begin{equation}\label{eq:1.2.29}
(a;q)_\infty = \prod^\infty_{k=0} (1-aq^k)
\end{equation}
for $|q| <1$.  Since the infinite product in
\eqref{eq:1.2.29} diverges when $a\not= 0$ and $|q| > 1$, whenever
$(a;q)_\infty$ appears in a formula, we shall assume that $|q| < 1$. 
In particular, for $|q|<1$ and $z$ an integer
\begin{equation}\label{eq:AI.6}
(a;q)_z = \frac{(a;q)_\infty}{(aq^z;q)_\infty},
\end{equation}
which is a notation that we also employ for complex $z$, where we take 
standard branch cut for the complex power. 

% In Section \ref{ssec:identitiesinvolvingqshiftedfac} we give a list
% of useful and easily verified identities for $q$-shifted factorials.

% Since products of $q$-shifted factorials occur so often, to
% simplify them we shall frequently
% use the more compact notations
% \begin{equation}\label{eq:1.2.41} 
% (a_1,a_2,\ldots, a_m;q)_n = (a_1;q)_n (a_2;q)_n\cdots (a_m;q)_n, 
% \qquad n\in\mathbf{N}. %\cup\{\infty\}.
% \end{equation}

The
basic hypergeometric series 
\begin{equation*}
\rphis{r+1}{r}{a_1, a_2,\ldots ,a_{r+1}}{b_1,\ldots\
,b_r}{q,z}
\end{equation*}
is called $k$-\emph{balanced}\index{basic hypergeometric series 
!k@$k$-balanced}\index{k@$k$-balanced basic hypergeometric series}
if $b_1b_2\cdots b_r
=q^k a_1a_2\cdots a_{r+1}$ and $z = q$,
and a 1-balanced basic hypergeometric series is called \emph{balanced} (or \emph{Saalsch\"utzian}).  
\index{basic hypergeometric series 
!balanced}\index{balanced basic hypergeometric series}
\index{basic hypergeometric series 
!Saalsch\"utzian}\index{Saalsch\"utzian basic hypergeometric series}
The
basic hypergeometric series ${}_{r+1}\vp_r$ 
is \emph{well-poised}\index{basic hypergeometric series 
!well-poised}\index{well-poised basic hypergeometric series}
if the parameters satisfy the
relations
\begin{equation*}
qa_1 = a_2b_1 = a_3b_2 = \cdots = a_{r+1}b_r;
\end{equation*}
\emph{very-well-poised}\index{basic hypergeometric series 
!very-well-poised}\index{very-well-poised basic hypergeometric series} if, in addition, 
$a_2 = qa_1^\hf, a_3 = -qa_1^\hf$.
% ; 
% a \emph{nearly-poised series
% of the first kind} if
% \begin{equation*}
% qa_1 \not= a_2b_1 = a_3b_2 = \cdots = a_{r+1}b_r,
% \end{equation*}
% and a \emph{nearly-poised series of the second kind} if
% \begin{equation*}
% qa_1 = a_2 b_1 = a_3b_2 = \cdots = a_rb_{r-1} \not=
% a_{r+1}b_r.
% \end{equation*}
% Note also
% that a terminating nearly-poised series of the second kind can be
% expressed as a multiple of a nearly-poised series of the first
% kind by simply reversing the series.

For very-well-poised series the following notation is in use:
\index{W@${}_{r+1}W_r$ very-well-poised series}
\index{very-well-poised series}
\begin{equation}\label{eq:defWseries} 
\begin{split}
{}_{r+1}W_r \left(a_1; a_4, a_5, \ldots,
a_{r+1};q,z\right) =& 
\rphis{r+1}{r}{a_1, qa_1^\hf, -qa_1^\hf,
a_4,\ldots,a_{r+1}}{a_1^\hf, -a_1^\hf,
qa_1/a_4,\ldots,qa_1/a_{r+1}}{q,z} \\
=& \sum_{k=0}^\infty \frac{1-a_1q^{2k}}{1-a_1} 
\frac{(a_1, a_4,\cdots, a_{r+1};q)_k}{(q, qa_1/a_4, \cdots, qa_{r+1}/a_1} z^k.
\end{split}
\end{equation}
% The series in \eqref{eq:defWseries}) is \emph{very-well-poised balanced} if
% $(a_1\cdots a_{r+1})z\, =\, (\pm (a_1q)^\hf)^{r-3}$ (with either sign).

The $q$-binomial coefficient is defined as
\begin{equation}\label{eq:AI.39}
\qbinomial{n}{k}{q}\, = \, \qbinomial{n}{n-k}{q} \, = \, \frac{(q;q)_n}{(q;q)_k\, (q;q)_{n-k}}
\end{equation}
and satisfies the following recurrences
\begin{equation}\label{eq:AI.45}
\qbinomial{n+1}{k}{q} \,=\, q^k\, \qbinomial{n}{k}{q} + \qbinomial{n}{k-1}{q} \, =\, 
 \qbinomial{n}{k}{q} + q^{n+1-k}\, \qbinomial{n}{k-1}{q}.
\end{equation}
The generalized $q$-binomial coefficient is defined for complex $\alpha$, $\beta$ by
\begin{equation}\label{eq:AI.40}
\qbinomial{\alpha}{\beta}{q}\, = \, \frac{(q^{\beta+1}, q^{\alpha-\beta+1};q)_\infty}
{(q,q^{\alpha+1};q)_\infty} 
\end{equation}
and then \eqref{eq:AI.45} remains valid for complex $\alpha$.

We end Section \ref{ssec:BHS-notation} with some useful identities for 
$q$-shifted factorials;
\begin{gather*}
(a;q)_n = \frac{(a;q)_\infty}{(aq^n;q)_\infty} \\
(a^{-1}q^{1-n};q)_n =  (a;q)_n (-a^{-1})^n q^{-{\binomial{n}{2}}}, 
(a;q)_{n-k} = \frac{(a;q)_n}{(a^{-1}q^{1-n};q)_k}\;
(-qa^{-1})^kq^{\binomial{k}{2}-nk}, \\
(a;q)_{n+k} = (a;q)_n (aq^n;q)_k, \\
(aq^n;q)_k = \frac{(a;q)_k (aq^k;q)_n}{(a;q)_n}, \\
(aq^k;q)_{n-k} = \frac{(a;q)_n}{(a;q)_k},\\
(aq^{2k};q)_{n-k} = \frac{(a;q)_n(aq^n;q)_k}{(a;q)_{2k}}, \\
(q^{-n};q)_k = \frac{(q;q)_n}{(q;q)
_{n-k}}(-1)^kq^{\binomial{k}{2}-nk}, \\
(aq^{-n};q)_k = \frac{(a;q)_k (qa^{-1};q)_n}{(a^{-1}q^{1-k};q)_n} q^{-nk}, \\
(a;q)_{2n} = (a;q^2)_n (aq;q^2)_n, \\
(a;q)_{3n} = (a;q^3)_n (aq;q^3)_n (aq^2;q^3)_n, \\
(a^2;q^2)_n = (a;q)_n (-a;q)_n, \\
(a^3;q^3)_n = (a;q)_n (\omega a;q)_n (\omega^2 a;q)_n, \qquad \omega = e^{2\pi i /3} \\
\end{gather*}
and similar expressions for $(a;q)_{kn}$ and $(a^k;q^k)_n$, $k=4,5,\cdots$.

\subsection{Some summation and transformation formulae}\label{ssec:BHS-summationtransformation}

There are many summation and transformation results for basic hypergeometric series available, 
and we only give a few basic results. We give precise references in case we need
more advanced summation or transformation formulae. 

The most fundamental result is the $q$-binomial theorem, stating\index{q-binomial@$q$-binomial theorem} 
\begin{equation}\label{eq:qbinomialthm}
\rphis{1}{0}{a}{-}{q,z} =\frac{(az;q)_\infty}{(z;q)_\infty}, \qquad |z|<1.
\end{equation}
Its terminating version reads 
\begin{equation}\label{eq:qbinomial-term}
\rphis{1}{0}{q^{-n}}{-}{q,z} = (q^{-n}z;q)_n, \qquad n\in \N.
\end{equation}
The proof is sketched in Exercise \ref{ex:qbinomthm}. We discuss a few consequences of 
the $q$-binomial theorem.

First, we write 
\begin{gather*}
\frac{(az;q)_\infty}{(z;q)_\infty} \frac{(z;q)_\infty}{(bz;q)_\infty} = 
\frac{(az;q)_\infty}{(bz;q)_\infty} \quad \Longrightarrow\quad
\rphis{1}{0}{a}{-}{q,z} \rphis{1}{0}{1/b}{-}{q,bz} =  \rphis{1}{0}{a/b}{-}{q,bz}
\end{gather*}
and this is a product of analytic functions, so that the coefficients of
the power series have to be equal. This gives
\begin{equation}\label{eq:qChuVDM}
\sum_{k+p=n} \frac{(a;q)_k}{(q;q)_k} \frac{(1/b;q)_p}{(q;q)_p} b^p = 
\frac{(a/b;q)_n}{(q;q)_n} b^n \quad \Longrightarrow \quad
\rphis{2}{1}{q^{-n},a}{c}{q,q} = \frac{(c/a;q)_n}{(c;q)_n}a^n
\end{equation}
after relabeling. This is the $q$-Chu-Vandermonde formula.\index{q-Chu@$q$-Chu-Vandermonde sum}

Another application of the $q$-binomial formula \eqref{eq:qbinomialthm} is Heine's 
transformation formula. 
Heine showed\index{Heine's transformation for ${}_2\vp_1$}
\begin{equation}\label{eq:1.4.1}
_2\varphi_1 (a,b;c;q,z) = \frac{(b,az;q)_\infty}{(c,z;q)_\infty}\;_2\varphi_1 (c/b, z; az;q,b),
\end{equation}
where $|z| <1$ and $|b| <1$. By iterating the result
\begin{equation}\label{eq:1.4.4}
 \begin{split}
_2\varphi_1 (a,b;c;q,z) &= \frac{(c/b, bz;q)_\infty}{
(c,z;q)_\infty}\;_2\varphi_1(abz/c,b;bz;q,c/b) \\[3pt]
&= \frac{(abz/c;q)_\infty}{(z;q)_\infty}\;_2\varphi_1
(c/a,c/b;c;q,abz/c).
\end{split}
\end{equation}
with appropriate conditions on the parameters for the last two series
to be convergent.  Heine's formula \eqref{eq:1.4.1} can directly be obtained 
from the $q$-binomial theorem \eqref{eq:qbinomialthm};
\begin{equation}
\begin{split}
_2\varphi_1 (a,b;c;q,z) \, =&\,  \frac{(b;q)_\infty}{(c;q)_\infty} \sum_{n=0}^{\infty}
 \frac{(a;q)_n(cq^n;q)_\infty}{(q;q)_n (bq^n;q)_\infty}z^n\\[3pt]
 =&\, \frac{(b;q)_\infty}{(c;q)_\infty} \sum_{n=0}^{\infty}
\frac{(a;q)_n}{(q;q)_n} z^n \sum_{m=0}^{\infty} \frac{(c/b;q)_m}{(q;q)_m}
(bq^n)^m\\[3pt]
=&\, \frac{(b;q)_\infty}{(c;q)_\infty} \sum_{m=0}^{\infty}
\frac{(c/b;q)_m}{(q;q)_m} b^m \sum_{n=0}^{\infty}
\frac{(a;q)_n}{(q;q)_n} (zq^m)^n\\[3pt]
=&\, \frac{(b;q)_\infty}{(c;q)_\infty} \sum_{m=0}^{\infty} \frac{(c/b;q)_m}{(q;q)_m} b^m
\frac{(azq^m;q)_\infty}{(zq^m;q)_\infty}\\[3pt]
=&\, \frac{(b,az;q)_\infty}{(c,z;q)_\infty} \rphis{2}{1}{c/b, z}{az}{q, b},
\end{split}
\end{equation}
which gives \eqref{eq:1.4.1}. The implied convergence
of the series above is assumed to hold.
Limit cases of Heine's transformation formulas \eqref{eq:1.4.1}, \eqref{eq:1.4.4} are 
\begin{equation}\label{eq:1.4.1limita}
\begin{split}
(c;q)_\infty \rphis{1}{1}{a}{c}{q,z} &\, =\,  (a,z;q)_\infty \rphis{2}{1}{c/a, 0}{z}{q, a} \\ 
&\, = \, (c/a;q)_\infty \rphis{2}{1}{az/c, a}{0}{q, \frac{c}{a}} \\ 
&\, = \, (az/c, c;q)_\infty  \rphis{2}{1}{c/a, 0}{c}{q, \frac{az}{c}} \\
&\, = \, (z;q)_\infty \rphis{1}{1}{az/c}{z}{q, c},
\end{split}
\end{equation}
so that in particular $(c;q)_\infty \, {}_1\varphi_1(0;c;q,z) =  (z;q)_\infty \, {}_1\varphi_1(0;z;q,c)$
is symmetric in $c$ and $z$. 
This symmetry is observed by Koornwinder and Swarttouw \cite{KoorS-TAMS} in their study
of the $q$-Hankel transform for the ${}_1\varphi_1$-$q$-Bessel functions. 
Taking a limit in \eqref{eq:1.4.1limita} we obtain 
\begin{equation}\label{eq:1.4.1limitb}
(c;q)_\infty \rphis{0}{1}{-}{c}{q,z} 
\, = \, (z/c, c;q)_\infty  \rphis{2}{1}{0, 0}{c}{q, \frac{z}{c}} 
\, = \,  \rphis{1}{1}{z/c}{0}{q, c} .
\end{equation}

Consider the $q$-integral on an interval $[0,a]$, defined by 
\begin{equation}\label{eq:qintegral0a}
\int_0^a f(t) \, d_q(t) = (1-q)a \sum_{k=0}^\infty q^k f(aq^k)
\end{equation}
whenever the function $f$ is such that the series in \eqref{eq:qintegral0a} converges. 
Note that we can view \eqref{eq:qintegral0a} as a Riemann sum for $\int_0^a f(t)\, dt$ on a
non-equidistant partition of the interval $[0,a]$. Using the notation \eqref{eq:qintegral0a} 
we can rewrite \eqref{eq:1.4.1} as 
\begin{gather*}
\rphis{2}{1}{a,b}{c}{q,z} = \frac{(b,c/b, z;q)_\infty}{(q,c,z;q)_\infty} 
\sum_{k=0}^\infty \frac{(q^{k+1},azq^k;q)_\infty}{ (q^kc/b, zq^k;q)_\infty} b^k \\
= \frac{(b,c/b, z;q)_\infty}{(q,c,z;q)_\infty}  \frac{1}{1-q} \int_0^1
\frac{(qt,azt;q)_\infty}{ (tc/b, tz;q)_\infty} t^{-1+\log_q b}\, d_qt
\end{gather*}
which can be considered as a $q$-analogue of Euler's integral representation 
\[
\rFs{2}{1}{a,b}{c}{z} =\frac{\Ga(c)}{\Ga(b)\Ga(c-b)}\int_0^1 t^{b-1} (1-t)^{c-b-1} (1-tz)^{-a}\, dt,
\qquad \Re c> \Re b>0.
\]
for the hypergeometric series, see e.g. \cite{AndrA}, \cite{Temm}. 

Another integral representation is the Watson integral 
representation.\index{Watson's integral representation} 
In Watson's
formula we assume $0<q<1$, and then 
\begin{equation}\label{eq:Watsonint2phi1}
\rphis{2}{1}{a,b}{c}{q,z} = 
\frac{-1}{2\pi} \frac{(a,b;q)_\infty}{(q,c;q)_\infty} 
\int_{-i\infty}^{i\infty} \frac{(q^{1+s}, cq^s;q)_\infty}{(aq^s, bq^s;q)_\infty}
\frac{\pi (-z)^s}{\sin(\pi s)} \, ds
\end{equation}
for $|z|<1$ and $|\arg(-z)|<\pi$. The contour runs from $-i\infty$ to $i\infty$ via the imaginary 
axis with indentations such that the poles of $1/\sin(\pi s)$ lie to the right of
the contour and the poles of $1/(aq^s, bq^s;q)_\infty$ lie to the left of the contour.
Then Watson's formula \eqref{eq:Watsonint2phi1} follows by a residue calculation and 
estimates on the behaviour of the integrand, see \cite[\S 4.2]{GaspR}. By then 
flipping the contour, and evaluating the integral using the residues at the poles
of $1/(aq^s, bq^s;q)_\infty$ and performing the right estimates shows the 
connection formula, see \cite[\S 4.3]{GaspR};
\begin{equation}\label{eq:4.3.2}
\begin{split}
\rphis{2}{1}{a,b}{c}{q,z}\, 
= &\, \frac{(b,c/a;q)_\infty (az,q/az;q)_\infty}
{(c,b/a;q)_\infty(z,q/z;q)_\infty}
\; \rphis{2}{1}{a,aq/c}{aq/b}{q,cq/abz}
 \\ &\, + \, \frac{(a,c/b;q)_\infty}{(c,a/b;q)_\infty} \frac{(bz,q/bz;q)_\infty}
{(z,q/z;q)_\infty}\;
\rphis{2}{1}{b,bq/c}{bq/a}{q,cq/abz}.
\end{split}
\end{equation}
which gives the analytic continuation
to the region $|\text{arg} (-z)| < \pi$,  with $c$ and
$a/b$ not integer powers of $q$, and $a,b, z \not= 0$. 
Note that the coefficients in \eqref{eq:4.3.2} are related to theta 
functions,\index{theta function}. Indeed, because of Jacobi's triple product identity,
see \cite[\S 1.6]{GaspR},
\begin{equation}\label{eq:thetafunction}
\theta(z) = (z,q/z;q)_\infty \qquad  \Longrightarrow 
\qquad  \theta(q^kz) = (-z)^{_k} q^{-\frac12 k(k-1)} \theta(z)
\end{equation}
$\theta$ is a renormalized Jacobi theta function.

\subsection{Exercises}

\begin{enumerate}[1.]
\item\label{ex:Raabe2F1} Use Raabe's test to show that ${}_2F_1(a,b;c,z)$, $|z|=1$, converges absolutely for 
$\Re(c-a-b)>0$.
\item Prove the statements on convergence of the basic hypergeometric series as in Remark \ref{rmk:convergencerphis}.
\item Prove \eqref{eq:ex1.4ii}. 
\item Prove \eqref{eq:AI.45}.
\item Prove the useful identities for $q$-shifted factorials. 
\item\label{ex:qbinomthm} Askey's proof of the $q$-binomial theorem \eqref{eq:qbinomialthm} goes as follows. 
Denote the ${}_1\vp_0$-series by $h_a(z)$ and show that
\begin{gather*}
h_a(z) \, - \, h_{aq}(z)\, = \, -az\, h_{aq}(z), \qquad 
h_a(z)\, - \, h_a(qz)\, = \, (1-a)z\, h_{aq}(z)\quad \Longrightarrow \\
 h_a(z) \, = \, \frac{1-az}{1-z}\, h_a(qz).
\end{gather*}
Iterate and use the analyticity and the value at $z=0$ to finish the proof.
\end{enumerate}

\subsection*{Notes}
The basic reference for basic hypergeometric series is the standard book \cite{GaspR} by 
Gasper and Rahman, or the first edition of \cite{GaspR}. 
The book by Gasper and Rahman contains a wealth of information on basic hypergeometric series. 
There are older books containing chapters on basic hypergeometric series,
e.g. Bailey \cite{Bail}, Slater \cite{Slat}, as well as the Heine's book --the second edition
of \emph{Handbuch der Kugelfunktionen}
of 1878, see 
references in \cite{GaspR}. 
More modern books on special functions having chapters
on basic hypergeometric series are e.g. \cite{AndrAR}, \cite{Isma}.
Another useful reference is the lecture notes by Ismail \cite{Isma-LN}.

%%%%%%%%%%%%%%%%%%%%%%%%%%%%%%%%%%%%%%%%%%%%%%%%%%%%%%%%%%%%%%%%%%%
%%%%%NEW SECTION%%%%%%%%%%%%%%%%%%%%%%%%%%%%%%%%%%%%%%%%%%%%%%%%%%%
%%%%%%%%%%%%%%%%%%%%%%%%%%%%%%%%%%%%%%%%%%%%%%%%%%%%%%%%%%%%%%%%%%%

\section{Basic hypergeometric $q$-difference equation}\label{sec:BHS-qdiff}

An important aspect of the hypergeometric series ${}_2F_1$ 
is that it can be used to describe the solutions to 
the \emph{hypergeometric differential equation}\index{hypergeometric differential equation}
\begin{equation}\label{eq:hypergeometricdiffeqtn}
z(1-z) \frac{d^2f}{dz^2}(z) + (c-(a+b+1)z) \frac{df}{dz}(z) - ab f(z) =0, 
\end{equation}
see e.g. \cite{AndrAR}, \cite{Isma}, \cite{Rain}, \cite{Temm}. 
In particular, 
\[
u_1(z) = \rFs{2}{1}{a,b}{c}{z}, \ \ c\not= 0,-1,-2,\cdots
\]
solves the hypergeometric differential equation \eqref{eq:hypergeometricdiffeqtn} as can be
checked directly by plugging the power series expansion. Other solutions expressible in 
terms of hypergeometric series are e.g. 
\begin{gather*}
u_2(z) =  z^{1-c} \rFs{2}{1}{a-c+1,b-c+1}{2-c}{z},\ \  c\not=2,3, \cdots \\
u_3(z) =  z^{-a} \rFs{2}{1}{a,a-c+1}{a-b+1}{\frac{1}{z}}, \ \ a-b\not= -1, -2, \cdots \\
u_4(z) =  z^{-b} \rFs{2}{1}{b,b-c+1}{b-a+1}{\frac{1}{z}}, \ \ b-a\not= -1, -2, \cdots.
\end{gather*}
The differential equation \eqref{eq:hypergeometricdiffeqtn} is a Fuchsian 
differential equation with three regular singular points at $0$, $1$ and $\infty$.
So one usually also considers the similar solutions in terms of power series around $z=1$,
but these solutions do not have appropriate $q$-analogues. 

So in general we have two linearly independent solutions in terms of power series around $0$ and two linearly independent solutions in 
terms of power series around $\infty$. Since the solution space of \eqref{eq:hypergeometricdiffeqtn} is $2$-dimensional, there
are all kinds of relations between these solutions. 
One of the classical relations between hypergeometric series is given by 
\begin{equation}\label{eq:rel2F1at0andinfty}
\begin{split}
\rFs{2}{1}{a,b}{c}{z} = \frac{\Ga(c)\Ga(b-a)}{\Ga(b)\Ga(c-a)} 
&(-z)^{-a} \rFs{2}{1}{a,a-c+1}{a-b+1}{\frac{1}{z}} \\
+ &\frac{\Ga(c)\Ga(a-b)}{\Ga(a)\Ga(c-b)} 
(-z)^{-b} \rFs{2}{1}{b,b-c+1}{b-a+1}{\frac{1}{z}}
\end{split}
\end{equation}
for $|\arg(-z)|<\pi$. 

The Jacobi polynomials are special cases of the hypergeometric series 
${}_2F_1$; explicitly
\begin{equation}\label{eq:defJacobipols}
P_n^{(\al,\be)}(x) = \frac{(\al+1)_n}{n!} \rFs{2}{1}{-n, n+\al+\be+1}{\al+1}{\frac12(1-x)}.
\end{equation}
So in particular, the Jacobi polynomials are eigenfunctions to a second-order differential
operator. This differential operator can then be studied on the weighted $L^2$ spaces 
with respect to the beta-weight $(1-x)^\al(1-x)^\be$ on $[-1,1]$. The 
differential operator is 
a self-adjoint operator on a suitable domain with compact resolvent.
The orthogonality of the Jacobi polynomials is related to the orthogonality of the 
eigenvectors of the corresponding differential operator. 

The Jacobi functions are 
\begin{equation}\label{eq:defJacobifunctions}
\phi^{(\al,\be)}_\la (t) = \rFs{2}{1}{\frac12(\al+\be+1+i\la),\frac12(\al+\be+1-i\la)}
{\al+1}{-\sinh^2t}
\end{equation}
and these are eigenfunctions of a related second order differential operator, after a
change of variables. The corresponding Jacobi function transform arises 
from the spectral decomposition of the differential operator, see \cite{Koor-Jacobi} 
for more information as well as the link to representation theory of non-compact 
symmetric spaces of rank one.

%%%%%%%%%%%%%%%%%%%%%%%%%%%%%%%%%%%%%%%%%%%%%%%%%%%%%%%%%%%%%%%%%%%
%%%%%NEW SECTION%%%%%%%%%%%%%%%%%%%%%%%%%%%%%%%%%%%%%%%%%%%%%%%%%%%
%%%%%%%%%%%%%%%%%%%%%%%%%%%%%%%%%%%%%%%%%%%%%%%%%%%%%%%%%%%%%%%%%%%

\subsection{Basic hypergeometric $q$-difference equation}\label{ssec:BHS-qdiff}

For fixed $q\not=1$, the $q$-derivative operator $D_q$ is defined by
\begin{equation}\label{eq:1.3.21}
D_qf \, (x) \, = \, \frac{f(x) - f(qx)}{(1-q)x},  \quad x\not=0,
\end{equation}
and $D_qf \, (0)=f'(0)$ assuming the derivative exists. 
Then $D_qf(x)$ tends to $f'(x)$ as $q\to 1$ for differentiable $f$. 
We can iterate; $D_q^nf \, = \, D_q (D_q^{n-1}f)$, $n=1,2,\cdots$. 
The $q$-difference operator $D_q$ 
applied to the ${}_2\vp_1$-series:
\begin{equation}\label{eq:ex1.12ii}
D^n_q \, \rphis{2}{1}{a, b}{c}{q, z} 
\, = \, \frac{(a,b;q)_n}{(c;q)_n(1-q)^n}\, 
\rphis{2}{1}{aq^n, bq^n}{cq^n}{q, z}
\end{equation}
which can be checked directly.
% and 
% \begin{equation}\label{eq:ex1.12iii}
% D^n_q \left( \frac{(z;q)_\infty}{(abz/c;q)_\infty} \, \rphis{2}{1}{a, b}{c}{q, z}\right) \, =\, 
%  \frac{(c/a, c/b;q)_n}{(c;q)_n (1-q)^n} \left(\frac{ab}{c}\right)^n
% \frac{(zq^n;q)_\infty}{(abz/c;q)_\infty} \,  \rphis{2}{1}{a, b}{cq^n}{q, zq^n} 
% \end{equation}
Moreover, $u(z) = \; _2\vp_1 (a,b;c;q,z)$ satisfies
(for $|z| < 1$ and in the formal power series sense) the second
order $q$-difference equation
\begin{equation}\label{eq:ex1.13}
z(c-abqz) D^2_q u + \left[ \frac{1-c}{1-q} + \frac{(1-a)(1-b) -
(1-abq)}{1-
q}z\right] D_q u\\[3pt]
- \frac{(1-a)(1-b)}{(1-q)^2} u = 0,
\end{equation}
which is a $q$-analogue of the hypergeometric differential equation 
\eqref{eq:hypergeometricdiffeqtn}. Indeed, replacing 
$a$, $b$, $c$ with $q^a$, $q^b$, $q^c$ and taking formal limits, shows that 
\eqref{eq:ex1.13} tends to \eqref{eq:hypergeometricdiffeqtn} as $q\to 1$. 
% For a study of this difference operator analogous to the study of
% the hypergeometric differential operator as a Fuchsian system and
% its Galois groups, see
% \cite{Roques2008} and references given there.  
Explicitly, \eqref{eq:ex1.13} is
\begin{equation}\label{eq:ex1.13a}
(c-abz)\, u(qz) \,+\, \left(-(c+q)+ (a+b)z\right)\, u(z)\,  +\, 
(q-z)\, u(z/q) \, =  0 
\end{equation}
for $a,b,c$ non-zero complex numbers. We consider 
\eqref{eq:ex1.13a} as the 
\emph{basic hypergeometric $q$-difference equation}.\index{basic hypergeometric $q$-difference equation}
Note that if $u$ is a solution to \eqref{eq:ex1.13a}, and $C$ is a $q$-periodic function, \index{q@$q$-periodic function}
i.e. $C(qz)=C(z)$, then $Cu$ is also a solution to \eqref{eq:ex1.13a}.

\begin{prop}\label{prop:solofBqDEat0atinfty}
The functions 
\begin{equation}
\begin{split}
u_1(z) \, &= \, \rphis{2}{1}{a, b}{c}{q, z}, \quad c\not= q^{-n}, \ n=0,1,2\cdots \\
u_2(z) \, &= \, z^{1-\log_q(c)}\, \rphis{2}{1}{qa/c, qb/c}{q^2/c}{q, z}, 
\quad c\not= q^{n+2}, \ n=0,1,2\cdots, 
\end{split}
\end{equation}
and the functions
\begin{equation}
\begin{split}
u_3(z) \, &= \, z^{-\log_q(a)}\, \rphis{2}{1}{a, qa/c}{qa/b}{q, \frac{qc}{abz}}, \quad a\not= bq^{-n-1}, \ n=0,1,2\cdots, \\
u_4(z) \, &= \, z^{-\log_q(b)}\, \rphis{2}{1}{b, qb/c}{qb/a}{q, \frac{qc}{abz}}, \quad b\not= aq^{-n-1}, \ n=0,1,2\cdots.
\end{split}
\end{equation}
are solutions of the basic hypergeometric $q$-difference equation \eqref{eq:ex1.13a}. 
\end{prop}

Since the map $z\mapsto qz$ has two fixed points on the Riemann sphere, namely $z=0$ and $z=\infty$, 
it is natural to consider power series expansion solutions of \eqref{eq:ex1.13a} at $z=0$ and $z=\infty$ 
using the Frobenius method. 

\begin{proof} We make the Ansatz 
\[
u(z) = \sum_{n=0}^\infty a_n z^{n+\mu}, \qquad a_n\in \C, \ a_0\not= 0,\ \mu\in \C
\]
Plugging such a solution into \eqref{eq:ex1.13a} and collecting the coefficients of $z^{n+\mu}$, we require 
\begin{gather*}
0 = a_0 z^\mu \bigl( cq^\mu -(c+q) + q^{1-\mu}\bigr) + \\ 
\sum_{n=1}^\infty z^{\mu+n} \Bigl( a_n \bigl( cq^{\mu+n} -(c+q) + q^{1-\mu-n}\bigr) + a_{n-1} \bigl( -abq^{\mu+n-1} 
+ a+b - q^{1-\mu-n}\bigr) \Bigr)
\end{gather*}
So the coefficient of $z^\mu$ has to be zero, and this gives the 
\emph{indicial equation}\index{indicial equation}
\[
0 =  cq^\mu -(c+q) + q^{1-\mu} = (q^\mu-1)(c-q^{1-\mu}).
\]
So we find $q^\mu=1$ or $q^{\mu} = q/c$. 

In case $q^\mu=1$ we find the recurrence relation 
\begin{gather*}
a_n \bigl( cq^{n} -(c+q) + q^{1-n}\bigr) =  a_{n-1} \bigl( abq^{n-1} 
-  (a+b) + q^{1-n}\bigr) \quad \Longrightarrow \\
a_n  =  a_{n-1}  \frac{(aq^{n-1}-1)(b-q^{1-n})}{(cq^n-q)(1-q^{-n})}  
= a_{n-1} \frac{(1-aq^{n-1})(1-bq^{n-1})}{(1-cq^n)(1-q^{n+1})} = a_0 \frac{(a,b;q)_n}{(c,q;q)_n}, 
\end{gather*}
so we find the solution $u_1$ for $\mu=0$. If we take more generally $\mu = \frac{2\pi i}{\log q}k$, $k\in \Z$, then 
we multiply $u_1$ by the $q$-periodic function $z\mapsto z^{\frac{2\pi i}{\log q}k}$. 

In case $q^{\mu} = q/c$
we find the recurrence relation 
\begin{gather*}
a_n \bigl( q^{1+n} -(c+q) + cq^{-n}\bigr) =  a_{n-1} \bigl( \frac{ab}{c}q^{n} 
-  (a+b) + cq^{-n}\bigr) \quad \Longrightarrow \\
a_n  =  a_{n-1}  \frac{(q^n\frac{a}{c}-1)(b-cq^{-n})}{(q^n-1)(q-cq^{-n})}  
= a_{n-1} \frac{(1-q^na/c)(1-bq^n/c)}{(1-q^n)(1-q^{n+1}/c)}  = a_0 \frac{(qa/c,qb/c;q)_n}{(q,q^2/c;q)_n}, 
\end{gather*}
so we find the solution $u_2$ for $\mu = 1-\log_q(c)$, and again if we add an integer multiple of 
$\frac{2\pi i}{\log q}$, we multiply by a $q$-periodic function. 

Similarly we obtain the solutions $u_3$, $u_4$ by replacing the Ansatz by 
$u(z) = \sum_{n=0}^\infty a_n z^{-n-\mu}$, $a_n\in \C$, $a_0\not= 0$, $\mu\in \C$. 
We leave this as Exercise \ref{eq:prop:solofBqDEat0atinfty}.
\end{proof}

\begin{remark}\label{rmk:linindepsolBHDE}
Note that for generic parameters the solutions $u_1$ and $u_2$, respectively $u_3$ and $u_4$,
are linearly independent (over $q$-periodic functions). So we expect that these solutions
satisfy relations amongst each other. In particular, \eqref{eq:4.3.2} gives
\begin{gather*} 
u_1(z) = C_3(z) u_3(z) + C_4(z) u_4(z) \\
C_3(z) = C_3(z;a,b;c) =\frac{(b,c/a;q)_\infty (az,q/az;q)_\infty}
{(c,b/a;q)_\infty(z,q/z;q)_\infty} z^{\log_q(a)}, 
\qquad C_4(z) = C_3(z;b,a;c)
\end{gather*}
Note that indeed, $C_3(qz)=C_3(z)$ using \eqref{eq:thetafunction} 
and so $C_3$ and $C_4$ are $q$-periodic functions. 
\end{remark}

We rewrite the basic hypergeometric equation \eqref{eq:ex1.13a} as 
\begin{equation}\label{eq:BHqDiff-rewrite}
(c-abz)\, \frac{u(qz)-u(z)}{z} + (1-z/q) \, \frac{u(z/q)-u(z)}{z/q} 
= (1-a)(1-b) u(z).
\end{equation}
We consider the left hand side as an operator acting on functions $u$, so 
we put\index{L@$L=L^{a,b,c}$} 
\begin{equation}\label{eq:Labc-def}
\bigl(L u\bigr)(z) = \bigl(L^{a,b,c} u\bigr)(z)
= (c-abz)\, \frac{u(qz)-u(z)}{z} + (1-z/q) \, \frac{u(z/q)-u(z)}{z/q},  
\end{equation}
so that upon using the normalized $q$-difference operator, 
cf. \eqref{eq:1.3.21},\index{D@$\tilde{D}_q$}  
\[
\bigl( \tilde{D}_qu\bigr) (z) = \frac{u(qz)-u(z)}{z}
\]
we can rewrite \eqref{eq:BHqDiff-rewrite} as 
\begin{equation}\label{eq:BHqDiff-rewrite2}
\bigl(L u\bigr)(z) = 
(c-abz)\, \bigl( \tilde{D}_qu\bigr) (z) -  (1-z/q) \, \bigl( \tilde{D}_qu\bigr) (z/q) 
= (1-a)(1-b) u(z),
\end{equation}
so that the operator $L$ in 
left hand side of \eqref{eq:BHqDiff-rewrite}, \eqref{eq:BHqDiff-rewrite2} 
can be written as the composition $L=S\circ \tilde{D}_q$, where 
\begin{equation}\label{eq:defSabc}
S = S^{a,b,c}, \qquad \bigl( Sf\bigr)(z) = (c-abz)f(z) - (1-z/q)f(z/q). 
\end{equation}
Note that when acting on polynomials, $\tilde{D}_q$ maps polynomials of degree $n$ to 
polynomials of degree $n-1$ and $S$ maps polynomials of degree $n$ to 
polynomials of degree $n+1$. So we see that $L= S\circ\tilde{D}_q$ gives a
factorisation in terms of a lowering operator and a raising operator. 
It is then common to consider the reversed composition and consider it as the 
Darboux transform of $L$.\index{Darboux transform}
Lemma \ref{lem:Darboux-L} shows that the Darboux transform is again of the same 
class. 

\begin{lemma}\label{lem:Darboux-L} We have 
\[
\bigl( \tilde{D}_q \circ S^{a,b,c} f\bigr)(z) 
= \frac{1}{q} \bigl(L^{aq,bq,cq} f\bigr)(z) + (1-q)(ab-q^{-1}) f(z), 
\]
and if $L^{a,b,c}u = (1-a)(1-b)u$, then $f=\tilde{D}_qu$ satisfies 
$L^{aq,bq,cq}f = (1-aq)(1-bq)f$. 
\end{lemma}

\begin{proof} A straightforward calculation gives 
\begin{gather*}
\bigl( \tilde{D}_q \circ S^{a,b,c} f\bigr)(z) = 
\frac{1}{z}\bigl( (c-abqz) f(qz) -(1-z)f(z) - (c-abz)f(z) + (1-z/q)f(z/q)\bigr) \\
= (c-abqz) \frac{f(qz)-f(z)}{z} + (1-z/q) \frac{f(z/q)-f(z)}{z/q}\frac{1}{q} 
+ ab(1-q)f(z) + (1-q^{-1})f(z)  \\
= \frac{1}{q} \left( (cq-abq^2z) \frac{f(qz)-f(z)}{z} + (1-z/q) \frac{f(z/q)-f(z)}{z/q}
\right) + (1-q)(ab-q^{-1}) f(z)
\end{gather*}
and the term in brackets is precisely the operator $L^{aq,bq,cq}$ acting on $f$ by 
\eqref{eq:Labc-def}. 

Apply the first result to $f=\tilde{D}_q u$, so that 
\begin{gather*}
\frac{1}{q} \bigl(L^{aq,bq,cq} f\bigr)(z) + (1-q)(ab-q^{-1}) f(z) 
= \bigl( \tilde{D}_q \circ S^{a,b,c} \circ \tilde{D}_q u\bigr)(z) \\ 
= \bigl( \tilde{D}_q \circ L^{a,b,c} u\bigr)(z) = 
(1-a)(1-b) \bigl(\tilde{D}_q u\bigr)(z) = (1-a)(1-b) f(z)
\end{gather*}
so that 
\begin{gather*}
\bigl(L^{aq,bq,cq} f\bigr)(z) = \bigl( q(1-a)(1-b) - q(1-q)(ab-q^{-1})\bigr) f(z) \\
= \bigl(  q -aq-bq +abq  - abq +abq^2 +1 -q)\bigr) f(z) 
= (1-aq)(1-bq) f(z).
\end{gather*}
This proves the second statement, and it is in line with \eqref{eq:ex1.12ii}.
\end{proof}

%%%%%%%%%%%%%%%%%%%%%%%%%%%%%%%%%%%%%%%%%%%%%%%%%%%%%%%%%%%%%%%%%%%

\subsection{Exercises}
\begin{enumerate}[1.]
\item Prove \eqref{eq:ex1.12ii}.
\item\label{eq:prop:solofBqDEat0atinfty} Show the second part of 
Proposition \ref{prop:solofBqDEat0atinfty}, see also Theorem \ref{thm:MVBHSatinfty}. 
\end{enumerate}

\subsection*{Notes}
The solutions follow unpublished notes by Koornwinder, and the factorisation and 
the Darboux transform seems to be well known. 

%%%%%%%%%%%%%%%%%%%%%%%%%%%%%%%%%%%%%%%%%%%%%%%%%%%%%%%%%%%%%%%%%%%
%%%%%NEW SECTION%%%%%%%%%%%%%%%%%%%%%%%%%%%%%%%%%%%%%%%%%%%%%%%%%%%
%%%%%%%%%%%%%%%%%%%%%%%%%%%%%%%%%%%%%%%%%%%%%%%%%%%%%%%%%%%%%%%%%%%

\section{Basic hypergeometric $q$-difference equation: polynomial case}\label{sec:BHS-qdiff-pol}

We first consider the basic hypergeometric $q$-difference operator on a space
which we can identify with a sequence space on $\N$. 
This is closely connected to the little $q$-Jacobi polynomials, and we derive 
its orthogonality and recurrence properties from properties of this operator. 

%%%%%%%%%%%%%%%%%%%%%%%%%%%%%%%%%%%%%%%%%%%%%%%%%
%%%%%%%%%%%%%%%%%%%%%%%%%%%%%%%%%%%%%%%%%%%%%%%%%

\subsection{The difference equation in a special case}\label{ssec:littleqJacobipols}
Replace $z=z_0q^{k+1}$ in \eqref{eq:ex1.13a} and put $u_k = u(z_0q^{k+1})$ then
we get 
\begin{gather*}
(c-abz_0 q^{k+1})\, u_{k+1} \,+\, \left(-(c+q)+ (a+b)z_0q^{k+1}\right)\, u_k\,  +\, 
(q-z_0q^{k+1})\, u_{k-1} \, =  0 \quad \Longrightarrow \\
(cq^{-k-1}-abz_0)\, u_{k+1} \,-\, (c+q)q^{-k-1}\, u_k\,  +\, 
(q^{-k}-z_0)\, u_{k-1} \, =  - (a+b)z_0\, u_k  
\end{gather*}
Note that in the special case $z_0=1$ the coefficient of $u_{-1}$ is zero.
So we consider the operator $L$ on sequences $u=(u_k)_{k\in\N}$ by 
\begin{equation}\label{eq:defDOlittleqJacobi}
\begin{split}
(L u)_k &= (cq^{-k-1}-ab)\, u_{k+1} \,-\, \bigl((cq^{-k-1}-ab) + 
(q^{-k}-1)\bigr) \, u_k\,  +\, 
(q^{-k}-1)\, u_{k-1} \\
&= (cq^{-k-1}-ab)\, (u_{k+1}-u_k)   +\, 
(q^{-k}-1)\, (u_{k-1}-u_k), \qquad k\geq 1, \\
(L u)_0 &=  (cq^{-1}-ab)\, (u_{1}-u_0)
\end{split}
\end{equation}
Note that putting $\varphi(q^k)=u_k$, we can view $L$ as 
\begin{gather*}
(L\varphi)(x) = (cq^{-1}-abx)\, \frac{\varphi(xq)-\varphi(x)}{x} + (1-x) \frac{\varphi(x/q)-\varphi(x)}{x}, 
\quad x=q^k, k\geq 1 \\
(L\varphi)(1) = (cq^{-1}-ab)\, \bigr(\varphi(q)-\varphi(1)\bigr)
\end{gather*}
so $L\varphi$ is expressible in terms of $D_q\varphi$ and $D_{q^{-1}}\varphi$. 
Considering $L$ as an operator acting on functions by
\begin{equation}\label{eq:lqJacpol-defL}
(L\varphi)(x) = (cq^{-1}-abx)\, \frac{\varphi(xq)-\varphi(x)}{x} + (1-x) \frac{\varphi(x/q)-\varphi(x)}{x}, 
\end{equation}
we see that $L$ preserves the polynomials and moreover that $L$ also preserves the degree.
Note that $L$ of \eqref{eq:lqJacpol-defL} is slightly different from the $L$ in 
\eqref{eq:Labc-def} since there is a $q$-shift in the argument. 

Letting $\C_N[x]$ be the polynomials of degree less than or equal to $N$, we see that 
$L\colon \C_N[x]\to \C_N[x]$ and moreover that $L$ is lower-triangular with respect to the 
basis $\{x^k \mid 0\leq k\leq N\}$. Then 
\[
Lx^n = \bigr(-ab (q^n-1) + (1-q^{-n})\bigr) x^n + \text{l.o.t.} =
(1-abq^{n}) (1-q^{-n}) x^n + \text{l.o.t.} 
\]
where $\text{l.o.t.}$ means `lower order terms'. 

\begin{lemma}\label{lem:lqJac-poleigfunctions}
For each degree $n$, the operator $L$ as in \eqref{eq:lqJacpol-defL} has a polynomial
eigenfunction of degree $n$ with eigenvalue $(1-abq^{n}) (1-q^{-n})$. 
\end{lemma}

\begin{proof} A lower triangular operator has its eigenvalues on the diagonal. Since
the eigenvalues are different for different $n$, all eigenvalues have algebraic and 
geometric multiplicity equal to $1$. 

From the basic hypergeometric difference equation and Proposition \ref{prop:solofBqDEat0atinfty}
we see that the polynomial eigenfunction is
\[
\rphis{2}{1}{q^{-n}, abq^n}{c}{q,qx}.
\]
\end{proof}

Consider the Hilbert space $\ell^2(\N,w)$\index{l@$\ell^2(\N,w)$} of weighted $\ell^2$-sequences with 
inner product 
\[
\langle u, v\rangle = \sum_{k=0}^\infty u_k \overline{v_k} w_k 
\]
for some positive sequence $w=(w_k)_{k\in \N}$.

\begin{prop}\label{prop:Lsaonell2w}
Take $0<q<1$, $ab, c\in\R$, then with 
\[
w_k = c^k \frac{(abq/c;q)_k}{(q;q)_k}, \qquad c>0, abq<c, 
\]
the operator $L$ with $D(L) = \{ u=(u_k)_{k\in\N} \mid u_k\not=0 \text{ for at most finitely many } k\in\N\}$
is symmetric;
\[
\langle Lu, v\rangle =  \langle u, Lv\rangle, \qquad \forall \, u, v\in D(L).
\]
\end{prop}

\begin{proof} We consider for finite sequences $(u)_k$, $(v)_k$ the difference of the inner products.
Note that in particular all sums are finite, so that absolute convergence of all series involved is automatic.

We do the calculation slightly more general by not making any assumptions on the sequences $(u)_k$, $(v)_k$, 
but by chopping off the inner product. Denote 
$\langle u, v\rangle_N = \sum_{k=0}^N u_k \overline{v_k} w_k$ 
then obviously $\lim_{N\to\infty} \langle u, v\rangle_N = \langle u, v\rangle$ for any two sequences 
$(u)_k$, $(v)_k \in \ell^2(\N,w)$. Now 
\begin{gather*}
\langle Lu, v\rangle_N 
-  \langle u, Lv\rangle_N = (cq^{-1}-ab)(u_1-u_0)\overline{v_0}w_0 -
u_0 \overline{(cq^{-1}-ab)(v_1-v_0)}w_0 + 
\\
\sum_{k=1}^N \Bigl( (cq^{-k-1}-ab)\, (u_{k+1}-u_k)   +\, 
(q^{-k}-1)\, (u_{k-1}-u_k)\Bigr) \overline{v_k} w_k - \\
\sum_{k=1}^N  u_k \overline{\Bigl( (cq^{-k-1}-ab)\, (v_{k+1}-v_k)   +\, 
(q^{-k}-1)\, (v_{k-1}-v_k)\Bigr)} w_k = \\
(cq^{-1}-ab)w_0\bigl( u_1\overline{v_0} - u_0\overline{v_1}\bigr) + 
\sum_{k=1}^N \Bigl( (cq^{-k-1}-ab)\, u_{k+1}  +\, 
(q^{-k}-1)\, u_{k-1}\Bigr) \overline{v_k} w_k - \\
\sum_{k=1}^N  u_k \overline{\Bigl( (cq^{-k-1}-ab)\, v_{k+1} +
(q^{-k}-1)\, v_{k-1}\Bigr)} w_k
\end{gather*}
since the coefficients are real. Relabeling gives 
\begin{gather*}
\langle Lu, v\rangle_N 
-  \langle u, Lv\rangle_N = (cq^{-1}-ab)w_0\bigl( u_1\overline{v_0} - u_0\overline{v_1}\bigr) + \\ 
\sum_{k=2}^{N+1} u_k (cq^{-k}-ab)\overline{v_{k-1}} \frac{w_{k-1}}{w_k} w_k -
\sum_{k=1}^N  u_k 
(q^{-k}-1)\, \overline{v_{k-1}} w_k \\
+ \sum_{k=0}^{N-1}  
u_{k} (q^{-1-k}-1)\,  \overline{v_{k+1}} \frac{w_{k+1}}{w_k} w_k - 
\sum_{k=1}^N  u_k(cq^{-k-1}-ab)\, \overline{v_{k+1}} w_k.
\end{gather*}
Since we want $L$ to be symmetric, most of the terms have to cancel for all sequences $(u)_k$, $(v)_k$.
So we need to impose 
\begin{equation}\label{eq:littleqJac-weight}
\begin{split}
(cq^{-k}-ab) \frac{w_{k-1}}{w_k} =  
(q^{-k}-1),\qquad 
(q^{-1-k}-1) \frac{w_{k+1}}{w_k} = 
(cq^{-k-1}-ab)
\end{split}
\end{equation}
which give the same recurrence relation for $w_k$;
\begin{equation*}
% \begin{split}
w_k = \frac{(cq^{-k}-ab)}{(q^{-k}-1)} w_{k-1} = 
c \frac{(1-abq^k/c)}{(1-q^k)} w_{k-1} = c^k \frac{(abq/c;q)_k}{(q;q)_k} w_0.
% \end{split}
\end{equation*}
Since we need $w_k>0$, we require $c>0$ and $abq/c<1$. 

Taking this value for $w_k$ we see that most of the terms in the sum cancels, and we 
get 
\begin{gather*}
\langle Lu, v\rangle_N 
-  \langle u, Lv\rangle_N = 
(cq^{-1}-ab)w_0\bigl( u_1\overline{v_0} - u_0\overline{v_1}\bigr) + \\
 u_{N+1} (q^{-N-1}-1) \overline{v_N} w_{N+1} - u_1(q^{-1}-1) \overline{v_0} w_1 \\
 + u_0(cq^{-1}-ab)\overline{v_1}w_0 
-  u_N(cq^{-N-1}-ab)\, \overline{v_{N+1}} w_{N} \\
= u_{N+1} (q^{-N-1}-1) \overline{v_N} w_{N+1} - u_N(cq^{-N-1}-ab)\, \overline{v_{N+1}} w_{N}.
% = w_N \bigl(  (q^{-N-1}-1) u_{N+1} \overline{v_N}- (cq^{-N-1}-ab) u_N\, \overline{v_{N+1}}\bigr)
\end{gather*}
Using \eqref{eq:littleqJac-weight} one obtains
\begin{equation}\label{eq:LuvN-uLvN-lqJac}
\langle Lu, v\rangle_N 
-  \langle u, Lv\rangle_N = 
\sqrt{(q^{-N-1}-1)(cq^{-N-1}-ab)} \sqrt{w_Nw_{N+1}} \bigl( 
u_{N+1} \overline{v_N} - u_N \, \overline{v_{N+1}}\bigr) 
\end{equation}
In particular, for finitely supported sequences 
$(u)_k$, $(v)_k$ we have $\langle Lu, v\rangle =  \langle u, Lv\rangle$ by taking $N\gg0$, so that
$L$ is symmetric with respect to the domain of the finitely supported sequences in $\ell^2(\N,w)$.
\end{proof}

\begin{remark}\label{rmk:spectralLlittleqJacobi}
We do not study the spectral analysis of $L$, but 
see \S \ref{ssec:AlSalamChihara-littleqJacobi}. In the analysis the form 
in \eqref{eq:LuvN-uLvN-lqJac} is closely related to the study of self-adjoint 
extensions of $L$. 
\end{remark}

Let us now assume that $u$ and $v$ correspond to the polynomial eigenvalues of 
$L$ as in Lemma \ref{lem:lqJac-poleigfunctions}. So this means 
$u_k = A + Bq^k + \cO(q^{2k})$, $v_k = C + Dq^k + \cO(q^{2k})$ as $k\to \infty$. For 
these values we see that \eqref{eq:LuvN-uLvN-lqJac} gives
\begin{gather*}
\langle Lu, v\rangle_N 
-  \langle u, Lv\rangle_N = \\ 
\sqrt{(q^{-N-1}-1)(cq^{-N-1}-ab)} \sqrt{w_Nw_{N+1}} \Bigl( 
\bigl(A + Bq^{N+1} + \cO(q^{2N})\bigr)
\bigl( \bar C + \bar Dq^N + \cO(q^{2N}) \bigr) - \\
\bigl(A + Bq^{N} + \cO(q^{2N})\bigr)
\bigl( \bar C + \bar Dq^{N+1} + \cO(q^{2N}) \bigr)
\Bigr) = \\
\sqrt{(q^{-N-1}-1)(cq^{-N-1}-ab)} \sqrt{w_Nw_{N+1}} \Bigl( (B\bar C q +A\bar D) q^N
- (B\bar C + A \bar D q) q^N + \cO(q^{2N}) \Bigr) = \\
\sqrt{(q^{-1}-q^N)(cq^{-1}-abq^N)} \sqrt{w_Nw_{N+1}} 
\Bigl( (B\bar C - A\bar D) (q-1) + \cO(q^{N}) \Bigr).
\end{gather*}
In particular, since we want to 
$\lim_{N\to\infty} \langle Lu, v\rangle_N 
-  \langle u, Lv\rangle_N = 0$ for such sequences $(u_k)$ and $(u_k)$ and 
since $w_N = \cO(c^N)$, we need $c<1$.
Assuming $0<c<1$ and $abq<c$ we see that any polynomial $p$ gives a 
sequence $(p(q^k))_k \in \ell^2(\N,w)$. 

%%%%%%%%%%%%%%%%%%%%%%%%%%%%%%%%%%%%%%%%%%%%%%%%%
%%%%%%%%%%%%%%%%%%%%%%%%%%%%%%%%%%%%%%%%%%%%%%%%%
\subsection{The little $q$-Jacobi polynomials}

In order to get to the little $q$-Jacobi polynomials we relabel. We (re-)define the 
polynomials from Lemma \ref{lem:lqJac-poleigfunctions} as the little $q$-Jacobi polynomials;
\index{little $q$-Jacobi polynomials}
\begin{equation}\label{eq:deflqJacpols}
p_n(x;\al,\be;q) = \rphis{2}{1}{q^{-n},\al\be q^{n+1}}{\al q}{q, qx}. 
\end{equation}
So we have specialized $(a,b,c)$ to $(q^{-n}, \al\be q^{n+1}, \al q)$ and the 
conditions $0<c<1$, $abq<c$ translate into $0<\al<q^{-1}$, $\be<q^{-1}$.  
By \eqref{eq:ex1.12ii} we have 
\begin{equation}\label{eq:DqonlittleqJac}
\begin{split}
\bigl( {D}_qp_n(\cdot;\al,\be;q)\bigr)(x) &= \frac{(1-q^{-n})(1-\al\be q^{n+1})}{(1-c)(1-q)} 
\rphis{2}{1}{q^{1-n},\al\be q^{n+2}}{\al q^2}{q, qx} \\
&=
\frac{(1-q^{-n})(1-\al\be q^{n+1})}{(1-c)(1-q)} p_{n-1}(x;\al q,\be q;q)
\end{split}
\end{equation}
Let us also rename the weight of Proposition \ref{prop:Lsaonell2w} to the new labeling.
We get 
\[
w(\al,\be;q)_k = (\al q)^k \frac{(\be q;q)_k}{(q;q)_k} 
\]
and we denote the corresponding Hilbert space $\ell^2(\N,w)$ by $\ell^2(\N;\al,\be;q)$.
A polynomial sequences in $\ell^2(\N;\al,\be;q)$ is a sequence of the form 
$u_k= p(q^k)$ for some polynomial $p$. Note that these sequences are indeed in 
$\ell^2(\N;\al,\be;q)$, and we denote them by $\cP$. 

\begin{lemma}\label{lem:adjointDq-lqJacpol}
$\tilde{D}_q$ is an unbounded map from $\ell^2(\N;\al,\be;q)$ to 
$\ell^2(\N;\al q,\be q;q)$. As its domain we take the polynomial sequences $\cP$.
Then we have 
\[
\langle \tilde{D}_q p, r \rangle_{\ell^2(\N;\al q,\be q;q)} 
= \langle p, S^{\al,\be} r \rangle_{\ell^2(\N;\al,\be;q)}, \qquad \forall \, p,r\in\cP
\]
where 
\[
S^{\al,\be} \colon \cP \to \cP, \qquad 
\bigl( S^{\al,\be} r \bigr)_k =  
(\al q)^{-1} \frac{(1-q^k)}{(1-\be q)} r(q^{k-1}) -  
\frac{(1- \be q^{k+1})}{(1-\be q)}r(q^k)
\]
and $-\al q (1-\be q)S^{\al,\be}$ 
corresponds to $S^{a,b,c}$ as in \eqref{eq:defSabc} with $a=q^{-n}$, $b=\al\be q^{n+1}$,
$c=\al q$ and $z=q^{k+1}$.
\end{lemma}

\begin{proof} Note that 
\begin{gather*}
\langle \tilde{D}_q p, r \rangle_{\ell^2(\N;\al q,\be q;q)} 
= \sum_{k=0}^\infty \frac{p(x^{k+1})-p(q^k)}{q^k} \overline{r(q^k)} 
(\al q^2)^k \frac{(\be q^2;q)_k}{(q;q)_k} \\
= \sum_{k=1}^\infty p(x^{k}) \overline{r(q^{k-1})} 
(\al q)^{k-1} \frac{(\be q^2;q)_{k-1}}{(q;q)_{k-1}} -  \sum_{k=0}^\infty p(x^{k}) \overline{r(q^k)} 
(\al q)^k \frac{(\be q^2;q)_k}{(q;q)_k} \\ 
= - p(1)\overline{r(1)} + 
\sum_{k=1}^\infty p(x^{k}) \left( \overline{r(q^{k-1})} 
(\al q)^{k-1} \frac{(\be q^2;q)_{k-1}}{(q;q)_{k-1}} - \overline{r(q^k)} 
(\al q)^k \frac{(\be q^2;q)_k}{(q;q)_k}\right)  \\ 
= - p(1)\overline{r(1)} + 
\sum_{k=1}^\infty p(x^{k}) \left( \overline{r(q^{k-1})} 
(\al q)^{-1} \frac{(1-q^k)}{(1-\be q)} - \overline{r(q^k)} 
\frac{(1- \be q^{k+1})}{(1-\be q)}\right)  (\al q)^k \frac{(\be q;q)_k}{(q;q)_k}.
\end{gather*}
In this derivation we use that all sums converge absolutely, 
so that we can split the series and rearrange them. This calculation gives, using that 
$\al,\be, q\in \R$, 
\[
\bigl( S^{\al,\be} r \bigr)_k =  
(\al q)^{-1} \frac{(1-q^k)}{(1-\be q)} r(q^{k-1}) -  
\frac{(1- \be q^{k+1})}{(1-\be q)}r(q^k)
\]
which is again in $\cP$ since $r\in \cP$. Note that the formula is also valid for $k=0$. 

We leave the fact that $\tilde{D}_q$ is unbounded to the reader. 
\end{proof}

Note that $S^{\al,\be}$ raises the degree of the polynomial by $1$. 
We have already observed how $\tilde{D}_q$ acts on a little $q$-Jacobi polynomial, but 
we also want to find out for $S^{\al,\be}$.

\begin{prop}\label{prop:adjointDq-lqJacpol2}
The little $q$-Jacobi polynomials are orthogonal in $\ell^2(\N;\al,\be;q)$.
Moreover, 
\begin{gather*}
\bigl( \tilde{D}_qp_n(\cdot;\al,\be;q)\bigr)(x) = 
\frac{(1-q^{-n})(1-\al\be q^{n+1})}{(1-c)} p_{n-1}(x;\al q,\be q;q),  \\
\bigl( S^{\al,\be} p_{n-1}(\cdot;\al q,\be q;q)\bigr)(x) 
= \frac{1}{\al q} \frac{1-\al q}{1-\be q} p_n(x;\al,\be;q)
\end{gather*}
\end{prop}

\begin{proof}
Note that the second order difference operator has the little $q$-Jacobi polynomial
as eigenfunction, see Lemma \ref{lem:lqJac-poleigfunctions}. In the relabeling the 
operator $L$ is given 
by 
\[
\bigl(L^{\al,\be} f\bigr)(x) = \bigl(Lf\bigr)(x) = \al(1- \be q x) \frac{f(qx)-f(x)}{x} 
+ (1-x) \frac{f(x/q)-f(x)}{x}
\]
see \eqref{eq:lqJacpol-defL} with $(a,b,c,x)$ replaced by $(q^{-n}, \al\be q^{n+1}, \al q, qx)$.
Since the decomposition of $L^{\al,\be}$ of \eqref{eq:defSabc} corresponds, up to a scalar,  
with the operator $S^{\al,\be}$ we see that 
$\bigl( S^{\al,\be} p_{n-1}(\cdot;\al q,\be q;q)\bigr)(x)$  has to be 
a multiple of $p_n(x;\al,\be;q)$. By considering the evaluation at $0$, and 
using $p_n(0;\al,\be;q)=1$, we see that the multiple is 
\[
\frac{(1/\al q-1)}{1-\be q} = \frac{1}{\al q} \frac{1-\al q}{1-\be q}
\]
since we can write for a polynomial $r$
\[
\bigl( S^{\al,\be} r\bigr) (x) = \frac{1}{\al q} \frac{(1-x)}{(1-\be q)}r(x/q) - 
\frac{(1-\be q x)}{(1-\be q)}r(x).
\]

In order to show the orthogonality, we consider the following inner product for $k\leq n$, using
the raising and lowering operators $S^{\al,\be}$ and $\tilde{D}_q$;
\begin{gather*}
\langle x^k, p_n(\cdot;\al,\be;q) \rangle_{\ell^2(\N;\al,\be;q)} = \al q \frac{(1-\be q)}{(1-\al q)}
\langle x^k, S^{\al,\be} p_{n-1}(\cdot;\al q,\be q;q) \rangle_{\ell^2(\N;\al q,\be q;q)} \\
= (q^k-1) \al q \frac{(1-\be q)}{(1-\al q)} 
\langle x^{k-1}, p_{n-1}(\cdot;\al q,\be q;q) \rangle_{\ell^2(\N;\al q,\be q;q)} 
\end{gather*}
since $\tilde{D}_q x^k = (q^k-1) x^{k-1}$. In particular, we get $0$ for the inner product in 
case $k=0$. By iterating the procedure we get 
\begin{gather*}
\langle x^k, p_n(\cdot;\al,\be;q) \rangle_{\ell^2(\N;\al,\be;q)} = \al q \frac{(1-\be q)}{(1-\al q)}
\langle x^k, S^{\al,\be} p_{n-1}(\cdot;\al q,\be q;q) \rangle_{\ell^2(\N;\al q,\be q;q)} \\
= (-1)^p (q^{k-p+1};q)_p
% (q^k-1)(q^{k-1}-1)\cdots (q^{k-p+1}-1) 
\al^p q^{\frac12 p(p+1)} 
\frac{(\be q;q)_p}{(\al q;q)_p} 
\langle x^{k-p}, p_{n-p}(\cdot;\al q^p,\be q^p;q) \rangle_{\ell^2(\N;\al q^p,\be q^p;q)} 
\end{gather*}
and for $k<n$ this gives zero for $p=k+1\leq n$. Hence, the little $q$-Jacobi polynomial
$p_n(\cdot;\al,\be;q)$ of degree $n$ 
is orthogonal to all monomials of degree $<n$. So the  the little $q$-Jacobi polynomials are 
orthogonal. 
\end{proof}

Note that as an immediate corollary to the proof, we also obtain
\begin{equation}\label{eq:squarednormlqJacopols}
\begin{split}
&\langle p_n(\cdot;\al,\be;q), p_n(\cdot;\al,\be;q) \rangle_{\ell^2(\N;\al,\be;q)} =
\text{lc}\bigl(p_n(\cdot;\al,\be;q)\bigr)
\langle x^n , p_n(\cdot;\al,\be;q) \rangle_{\ell^2(\N;\al,\be;q)} \\
= &
\text{lc}\bigl(p_n(\cdot;\al,\be;q)\bigr) 
(-1)^n (q^;q)_n
\al^n q^{\frac12 n(n+1)} 
\frac{(\be q;q)_n}{(\al q;q)_n} 
\langle 1, 1 \rangle_{\ell^2(\N;\al q^n,\be q^n;q)},
\end{split}
\end{equation}
where $\text{lc}(p)$ denotes the leading coefficient of the polynomial $p$. 
So the shift operators can be used to find the squared norm of the orthogonal polynomials in terms
of the squared norm of the constant function $1$. 

The leading coefficient can be calculated directly from the definition \eqref{eq:deflqJacpols};
\begin{gather*}
\text{lc}\bigl(p_n(\cdot;\al,\be;q)\bigr) = 
\frac{ (q^{-n}, \al \be q^{n+1};q)_n}{(q,\al q;q)_n}q^n = 
(-1)^n q^{-\frac12 n(n-1)} \frac{ (\al \be q^{n+1};q)_n}{(\al q;q)_n}. 
\end{gather*}
The squared norm of the constant function $1$ follows from the $q$-binomial sum 
\eqref{eq:qbinomialthm}; 
\begin{gather*}
\langle 1, 1 \rangle_{\ell^2(\N;\al,\be;q)} = 
\sum_{k=0}^\infty  (\al q)^k \frac{(\be q;q)_k}{(q;q)_k} =
\rphis{1}{0}{\be q}{-}{q, \al q} = \frac{(\al \be q^2;q)_\infty}{(\al q;q)_\infty}.
\end{gather*}

This gives the following orthogonality relations for the little $q$-Jacobi polynomials from the 
analysis of the hypergeometric $q$-difference operator on a very specific set of points. 

\begin{thm}\label{thm:orthlittleqJac} Let $0<\al<q^{-1}$, $\be<q^{-1}$ and consider the 
little $q$-Jacobi polynomials 
\[
p_n(x;\al,\be;q) = \rphis{2}{1}{q^{-n},\al\be q^{n+1}}{\al q}{q, qx} 
\]
and the inner product space 
\[
\langle f, g \rangle_{\ell^2(\N;\al,\be;q)} = \sum_{k=0}^\infty f(q^k)\overline{g(q^k)} 
(\al q)^k \frac{(\be q;q)_k}{(q;q)_k}
\]
then the little $q$-Jacobi polynomials satisfy 
\begin{gather*}
\langle p_n(\cdot;\al,\be;q) , p_m(\cdot;\al,\be;q) \rangle_{\ell^2(\N;\al,\be;q)} = 
\de_{m,n} h_n(\al,\be;q) \\ 
h_n(\al,\be;q) = (\al q)^n \frac{(q, \be q;q)_n}{(\al q, \al\be q;q)_n} 
\frac{1-\al\be q}{1-\al\be q^{2n+1}}
\frac{(\al\be q^2;q)_\infty}{(\al q;q)_\infty}.
\end{gather*}
\end{thm}

\begin{proof} This is a combination of the results in this section, and we are left with 
calculating the squared norm. Now 
\begin{gather*}
h_n(\al,\be;q) = \text{lc}\bigl(p_n(\cdot;\al,\be;q)\bigr) 
(-1)^n (q^;q)_n
\al^n q^{\frac12 n(n+1)} 
\frac{(\be q;q)_n}{(\al q;q)_n} 
\langle 1, 1 \rangle_{\ell^2(\N;\al q^n,\be q^n;q)} \\
= (-1)^n q^{-\frac12 n(n-1)} \frac{ (\al \be q^{n+1};q)_n}{(\al q;q)_n} 
(-1)^n (q^;q)_n
\al^n q^{\frac12 n(n+1)} 
\frac{(\be q;q)_n}{(\al q;q)_n} 
\frac{(\al \be q^{2n+2};q)_\infty}{(\al q^{n+1};q)_\infty} \\
=  \frac{(q, \be q;q)_n}{(\al q, \al\be q;q)_n} 
\frac{(\al q)^n}{1-\al\be q^{2n+1}}
\frac{(\al\be q;q)_\infty}{(\al q;q)_\infty} 
\end{gather*}
which is equal to the stated value. 
\end{proof}

%%%%%%%%%%%%%%%%%%%%%%%%%%%%%%%%%%%%%%%%%%%%%%%%%
%%%%%%%%%%%%%%%%%%%%%%%%%%%%%%%%%%%%%%%%%%%%%%%%%
\subsection{The three-term recurrence relation for the little $q$-Jacobi polynomials}

The shift operators $\tilde{D}_q$ and $S^{\al,\be}$ play an essential role in the 
derivation of the orthogonality relations for the little $q$-Jacobi polynomials in 
Theorem \ref{thm:orthlittleqJac}. As a final application we show how one can 
obtain the coefficients in the three-term recurrence relation for the 
monic little $q$-Jacobi polynomials. Let $\tilde{p}_n(x;\al,\be;q)$ be the monic
little $q$-Jacobi polynomials. Because of the monicity we have a three-term recurrence
of the form
\begin{equation*}
x\, \tilde{p}_n(x;\al,\be;q) = \tilde{p}_{n+1}(x;\al,\be;q) + b_n 
\tilde{p}_n(x;\al,\be;q) + c_n \tilde{p}_{n-1}(x;\al,\be;q).
\end{equation*}
The value of $c_n$ can then be calculated from the knowledge we already have obtained;
\begin{gather*}
c_n \langle \tilde{p}_{n-1}(x;\al,\be;q), \tilde{p}_{n-1}(x;\al,\be;q)\rangle_{\ell^2(\N;\al,\be;q)}
= \langle x\, \tilde{p}_n(x;\al,\be;q), \tilde{p}_{n-1}(x;\al,\be;q)\rangle_{\ell^2(\N;\al,\be;q)} \\
= \langle \tilde{p}_n(x;\al,\be;q), x \tilde{p}_{n-1}(x;\al,\be;q)\rangle_{\ell^2(\N;\al,\be;q)}
= \langle \tilde{p}_n(x;\al,\be;q), \tilde{p}_{n}(x;\al,\be;q)\rangle_{\ell^2(\N;\al,\be;q)}
\end{gather*}
using that multiplication by $x$ is self-adjoint. So we find
\begin{gather*}
 c_n = \frac{\langle \tilde{p}_n(x;\al,\be;q), \tilde{p}_{n}(x;\al,\be;q)\rangle_{\ell^2(\N;\al,\be;q)}}
 {\langle \tilde{p}_{n-1}(x;\al,\be;q), \tilde{p}_{n-1}(x;\al,\be;q)\rangle_{\ell^2(\N;\al,\be;q)}} 
 = \left( \frac{\text{lc}\bigl(p_{n-1}(\cdot;\al,\be;q)\bigr)}{\text{lc}\bigl(p_{n}(\cdot;\al,\be;q)\bigr)}
 \right)^2 \frac{h_n(\al,\be;q)}{h_{n-1}(\al,\be;q)}\\
 = \left(  -q^{n-1} \frac{(1-\al\be q^n)(1-\al q^n)}{(1-\al\be q^{2n-1})(1-\al\be q^{2n})}\right)^2
\al q \frac{(1-q^n)(1-\be q^n)(1-\al\be q^{2n-1})}{(1-\al q^n)(1-\al\be q^n)(1-\al\be q^{2n+1})} \\
= \al q^{2n-1} \frac{(1-q^n)(1-\al q^n)(1-\be q^n)(1-\al\be q^n)}{(1-\al\be q^{2n-1})
(1-\al\be q^{2n})^2(1-\al\be q^{2n+1})}.
 \end{gather*}
Writing $\tilde{p}_n(x;\al,\be;q) = x^n + r_n(\al,\be)x^{n-1} + \text{l.o.t.}$, it follows that 
upon comparing coefficients of $x^{n}$ in the three-term recurrence relation that 
\begin{equation*}
r_n(\al,\be) = r_{n+1}(\al,\be) + b_n \qquad \Longrightarrow \qquad b_n = 
r_n(\al,\be) - r_{n+1}(\al,\be).
\end{equation*}
We could read it off from the explicit expression \eqref{eq:deflqJacpols}, but we use the 
shift operators to find the values. Indeed, since, by Proposition \ref{prop:adjointDq-lqJacpol2},
\begin{gather*}
\bigl( \tilde{D}_q \tilde{p}_n(\cdot;\al,\be;q)\bigr)(x) 
= (q^n-1) \, \tilde{p}_{n-1}(x;\al q,\be q;q) \quad \Longrightarrow \\
(q^{n-1}-1) r_n(\al,\be) = (q^n-1) r_{n-1}(\al q,\be q)  
\quad \Longrightarrow \quad  r_n(\al,\be) = \frac{(1-q^n)}{(1-q^{n-1})} r_{n-1}(\al q,\be q)  
\quad \Longrightarrow \\
r_n(\al,\be) = \frac{(1-q^n)}{(1-q^{n-p})} r_{n-p}(\al q^p,\be q^p) =
\frac{(1-q^n)}{(1-q)} r_{1}(\al q^{n-1},\be q^{n-1}).
\end{gather*}
In order to determine $r_1(\al,\be)$ we use the shift operator as well. By Proposition 
\ref{prop:adjointDq-lqJacpol2} we have that $p_1(x;\al,\be;q)$ is a multiple of 
\begin{gather*}
\bigr( S^{\al,\be}1\bigr)(x) = \frac{1}{\al q} \frac{(1-x)}{(1-\be q)} - 
\frac{(1-\be q x)}{(1-\be q)} = \frac{\be q -1/\al q}{(1-\be q)} x  
+ \frac{1/\al q -1}{(1-\be q)} = 
\frac{\be q -1/\al q}{(1-\be q)} \left( x + \frac{1/\al q -1}{\be q -1/\al q}\right) \\
\Longrightarrow \qquad r_1(\al,\be) = \frac{1/\al q -1}{\be q -1/\al q}
= - \frac{1-\al q}{1-\al\be q^2}
\end{gather*}

\begin{prop}\label{prop:3trmoniclqJacpol}
The monic little $q$-Jacobi polynomials $\tilde{p}_n(x;\al,\be;q)$ satisfy the 
three-term recurrence relation 
\begin{gather*}
x\, \tilde{p}_n(x;\al,\be;q) = \tilde{p}_{n+1}(x;\al,\be;q) + b_n 
\tilde{p}_n(x;\al,\be;q) + c_n \tilde{p}_{n-1}(x;\al,\be;q), \\
b_n = \frac{q^n\left( (1+\al) -\al(1+\be)(1+q)q^n + \al\be q(1+\al)q^{2n}\right)}
{(1-\al\be q^{2n})(1-\al\be q^{2n+2})}
\\
c_n= \al q^{2n-1} \frac{(1-q^n)(1-\al q^n)(1-\be q^n)(1-\al\be q^n)}{(1-\al\be q^{2n-1})
(1-\al\be q^{2n})^2(1-\al\be q^{2n+1})}.
\end{gather*}
\end{prop}

The value for $b_n$ is seemingly different from the classical value given in e.g. 
\cite[(3.12.4)]{KoekS}. We leave this to Exercise \ref{ex:bnrightvalue}.

\begin{proof} We have already established the value for $c_n$. It remains to finish the 
calculation of $b_n$. This is done as follows;
\begin{gather*}
b_n = r_n(\al,\be) - r_{n+1}(\al,\be) = 
\frac{(1-q^n)}{(1-q)} r_{1}(\al q^{n-1},\be q^{n-1}) - 
\frac{(1-q^{n+1})}{(1-q)} r_{1}(\al q^{n},\be q^{n}) \\
= - \frac{(1-q^n)(1-\al q^n)}{(1-q)(1-\al\be q^{2n})} + 
\frac{(1-q^{n+1})(1-\al q^{n+1})}{(1-q)(1-\al\be q^{2n+2})}
\end{gather*}
By working this out we get 
\begin{gather*}
b_n 
= \frac{q^n}{(1-\al\be q^{2n})(1-\al\be q^{2n+2})}
\left( (1+\al) -\al(1+\be)(1+q)q^n + \al\be q(1+\al)q^{2n}\right)
\end{gather*}
\end{proof}

Note that the value 
\[
r_n(\al,\be) = - \frac{(1-q^n)(1-\al q^n)}{(1-q)(1-\al\be q^{2n})}
\]
also corresponds with \eqref{eq:deflqJacpols} taking into account division
by the leading coefficient. 

% Note that we have several formal eigenvectors to $L$. Indeed, define the sequence $(u_k(\la))_k$ by 
% \[
% u(\la)_k = u_k(\la) = \rphis{2}{1}{a\la, b/\la}{c}{q, q^{k+1}}, \quad c\not= q^{-n}, \ n=0,1,2\cdots 
% \]
% then by Proposition \ref{prop:solofBqDEat0atinfty} 
% \begin{equation}
% \bigl( Lu\bigr)_k = (ab-a\la-b/\la+1) u_k \quad \Longrightarrow \quad Lu = (1-a\la)(1-b/\la) u
% \end{equation}
% since $L$ only depends on the product of $a$ and $b$. Similarly, we 

\begin{remark} 
Let us view the operator $L=L^{\al,\be}$ as an operator acting on polynomials as
well as the operator $M$ which is acting by multiplication. They can be viewed 
as generators (up to an affine scaling) of a limit of the Zhedanov algebra, also 
known as the Askey-Wilson algebra. 
We refer to \cite{KoorM} for the precise formulation and related references. 
Moreover, the Zhedanov algebra as well as its degenerations in \cite{KoorM} 
have relations that can be interpreted as non-homogeneous Serre relations 
in quantum algebras, and this type of relations hold for generators of 
quantum symmetric pairs as studied by Gail Letzter and coworkers, 
see e.g. \cite[\S 5.3]{Kolb}. It is not clear what the connection entails. 
\end{remark}

%%%%%%%%%%%%%%%%%%%%%%%%%%%%%%%%%%%%%%%%%%%%%%%%%
%%%%%%%%%%%%%%%%%%%%%%%%%%%%%%%%%%%%%%%%%%%%%%%%%
\subsection{Relation to Al-Salam--Chihara polynomials}\label{ssec:AlSalamChihara-littleqJacobi}
We have circumvented the precise analytic study of the basic $q$-difference equation for 
the little $q$-Jacobi polynomials or the basic hypergeometric series. 
The reason is that this analytic study is somewhat complicated since the 
self-adjoint extension of the symmetric operator 
as in Lemma \ref{lem:adjointDq-lqJacpol} depend in general on parameters.
Indeed, $(L, D((L))$ as in Lemma \ref{lem:adjointDq-lqJacpol} is not 
essentially self-adjoint in general. We explain this in this section by relating
to a non-determinate moment problem. 

We can also relate the eigenvalue equation to orthogonal polynomials.
Indeed, rewriting $Lu=\la u$ gives 
\begin{gather*}
\la u_k(\la) = (cq^{-k-1}-ab)\, (u_{k+1}(\la)-u_k(\la))   +\, 
(q^{-k}-1)\, (u_{k-1}(\la)-u_k(\la))
\end{gather*}
which we can consider as a three-term recurrence for orthogonal polynomials 
with initial values $u_0(\la)=0$ (and $u_{-1}(\la)=0$).
In order to determine these polynomials, we first look at the monic version.
So we put $u_k(\la) = \al_k r_k(\la)$ with 
\[
 \frac{\al_{k+1}}{\al_k} (cq^{-k-1}-ab) = 1
\]
so that the recurrence relation becomes
\begin{gather*}
\la r_k(\la) = \, r_{k+1}(\la)- \bigl((cq^{-k-1}-ab)+ (q^{-k}-1)\bigr) r_k(\la)  +\, 
(q^{-k}-1)(cq^{-k}-ab)\, r_{k-1}(\la)
\end{gather*}
and putting $2\mu=\al(\la-ab-1)$, $p_k(\mu) = \al^{-k} r_k(\al(2\mu-ab-1))$ we get
\begin{gather*}
2\mu p_k(\mu) = p_{k+1}(\mu)- \frac{1}{\al} (cq^{-k-1}+ q^{-k}) p_k(\mu)  +\, 
\frac{ab}{\al^2}(1-q^{-k})(1-cq^{-k}/ab)\, p_{k-1}(\mu)
\end{gather*}
and finally taking $\al = - \sqrt{ab}$ we find 
\begin{equation}\label{eq:relAlSalChih}
2\mu p_k(\mu) = p_{k+1}(\mu) + q^{-k}(c/q\sqrt{ab}+ 1/\sqrt{ab}) p_k(\mu)  +\, 
(1-q^{-k})(1-cq^{-k}/ab)\, p_{k-1}(\mu).
\end{equation}
Now \eqref{eq:relAlSalChih} can be matched to \cite[\S 3.8]{KoekS}, so that $p_k(\mu)$ can be identified with the 
Al-Salam--Chihara polynomials $Q_k(\mu; c/q\sqrt{ab}, 1/\sqrt{ab}|q^{-1})$ 
in base $q^{-1}>1$. 

Theorem \ref{thm:AlSalChih-indeterminate} gives a characterization of the (in-)determinacy of 
the Al-Salam--Chihara polynomials in case the base is bigger than $1$, and it is 
due to Askey and Ismail \cite[Thm.~3.2, p.~36]{AskeI}. 

\begin{thm}\label{thm:AlSalChih-indeterminate} Consider the sequence of monic polynomials
\[
xv_n(x) = v_{n+1}(x) + A q^{-n} + (1-q^{-n}) (C-Bq^{-n}) v_{n-1}(x)
\]
with $0<q<1$, $B\geq 0$, $B> C$ and initial conditions $v_{-1}(x)=0$, $v_0(x)=1$.  
Then the corresponding moment problem is indeterminate if and only if 
\[
A^2>4B, \quad \text{and} \quad q\geq |\be^2 B|
\]
where $(1-At+Bt^2) = (1-t/\al)(1-t/\be)$ with $|\al|\geq |\be|$.  
\end{thm}

The conditions $0<q<1$, $B\geq 0$, $B> C$ in Theorem \ref{thm:AlSalChih-indeterminate}
ensure that the conditions of Favard's theorem, see e.g. \cite{Chih}, 
\cite{Deif}, \cite{DunfS}, 
\cite{Isma}, \cite{Koel-Laredo}, are met. 
So there is an orthogonality measure for which the polynomials $v_n(x)$ are orthogonal. 
In the determinate case this measure is uniquely determined by the polynomials, whereas
in the indeterminate case there are infinitely many orthogonality measures for these
polynomials. In the indeterminate case this means that the operator $L$ with domain
$D(L)$ the finite linear combinations as in Proposition \ref{prop:Lsaonell2w} is 
not essentially self-adjoint, see e.g. \cite{DunfS}, \cite{Koel-Laredo}, \cite{Simo}. 

The proof of Theorem \ref{thm:AlSalChih-indeterminate} follows by observing that 
a moment problem is indeterminate if and only if $\sum_{n=0}^\infty |p_n(i)|^2<\infty$ 
for the corresponding orthonormal polynomials $p_n(x)$, see e.g. \cite{Akhi}. 
Askey and Ismail then determine the asymptotic behaviour of the Al-Salam--Chihara 
polynomials by applying Darboux's method, see e.g. \cite{Olve}, to the 
generating function for the Al-Salam--Chihara 
polynomials.

Comparing Theorem \ref{thm:AlSalChih-indeterminate} with \eqref{eq:relAlSalChih} we see that 
we can apply Theorem \ref{thm:AlSalChih-indeterminate} with $(A,B,C)=(c/q\sqrt{ab} + 1/\sqrt{ab}, c/abq,1)$ 
and the same base $q$.
So the requirement for Favard's theorem translates to $c/abq>1$ or $c>abq$, which we 
now assume. 
Then  the first condition $A^2>4B$ translates to 
\[
\left( \frac{c}{q\sqrt{ab}} + \frac{1}{\sqrt{ab}}\right)^2 > \frac{4c}{abq} 
\ \Longleftrightarrow \ \left( \frac{c}{q\sqrt{ab}} - \frac{1}{\sqrt{ab}}\right)^2 >0,
\]
which is always true unless $c=q$. For the second condition we factorise
\[
1-At+Bt^2= 1- (c/q\sqrt{ab} + 1/\sqrt{ab})t + \frac{c}{abq}t^2 = 
(1-\frac{ct}{q\sqrt{ab}})(1-\frac{t}{\sqrt{ab}})
\]
so that $\{\al,\be\} =\{\frac{q\sqrt{ab}}{c}, \sqrt{ab}\}$. So 
$\be= \sqrt{ab}$ if $c\leq q$, and then $q\geq |\be^2B|$ is equivalent to $c\leq q^2$.
Next $\be=\frac{q\sqrt{ab}}{c}$ if $c\geq q$ and then 
$q\geq |\be^2B|$ is equivalent to $q\geq 1$. 
We conclude that $L$ is not essentially self-adjoint if $0<c<q$.

%%%%%%%%%%%%%%%%%%%%%%%%%%%%%%%%%%%%%%%%%%%%%%%%%%%%%%%%%%%%%%%%%%%

\subsection{Exercises}
\begin{enumerate}[1.]
\item Show that $\tilde{D}_q$ as in Lemma \ref{lem:adjointDq-lqJacpol} is unbounded.
\item\label{ex:bnrightvalue} Look up the standard value for $b_n$ as in Proposition \ref{prop:3trmoniclqJacpol}
and establish the equality with the value as given in Proposition \ref{prop:3trmoniclqJacpol}. 
\item Show that the three-term recurrence relation for the litle $q$-Jacobi polynomials as
in Proposition \ref{prop:3trmoniclqJacpol} gives a relation for the Al-Salam--Chihara polynomials
which is related to the $q$-difference operator for the  Al-Salam--Chihara polynomials. 
\end{enumerate}

\subsection*{Notes} The little $q$-Jacobi polynomials were introduced by 
Andrews and Askey \cite{AndrA} in 1977. The link to the quantum $SU(2)$-group
as matrix elements of unitary representations by Vaksman \& Soibelman,
Koornwinder and  Masuda, Mimachi, Nakagami, Noumi, and Ueno at the end of the 
1980s has led to many results on (subclasses of) little $q$-Jacobi polynomials, see 
the references in the lecture notes \cite{Koor-CQGSF} by Koornwinder.
The usage of the shift operators to obtain the explicit results is 
a technique that can be useful in other applications, such as multivariable
setting or in the matrix-valued case, see also \cite{Koor-CQGSF} for this 
approach for little and big $q$-Jacobi polynomials. 
The duality between little $q$-Jacobi polynomials and Al-Salam--Chihara polynomials
is observed by Rosengren \cite{Rose} and it is also observed by Groenevelt 
\cite{Groe-Ramanujan}. 
This duality --but in a dual way-- also plays an important role in the 
study of the quantum analogue of the Laplace-Beltrami operator on 
bounded quantum symmetric domain, see Vaksman \cite{Vaks}. 
For the corresponding Zhedanov algebra, the duality is described in 
\cite{KoorM}. 
The duality can also be extended to big $q$-Jacobi polynomials and 
continuous dual $q^{-1}$-Hahn polynomials, see \cite{KoelS-CA}. 
In general, this duality is related to explicit solutions of explicit indeterminate
moment problems, and several examples are known. 
A vector-valued analogue of \cite{KoelS-CA} is given by Groenevelt \cite{Groe-CA2009}.

%%%%%%%%%%%%%%%%%%%%%%%%%%%%%%%%%%%%%%%%%%%%%%%%%%%%%%%%%%%%%%%%%%%
%%%%%NEW SECTION%%%%%%%%%%%%%%%%%%%%%%%%%%%%%%%%%%%%%%%%%%%%%%%%%%%
%%%%%%%%%%%%%%%%%%%%%%%%%%%%%%%%%%%%%%%%%%%%%%%%%%%%%%%%%%%%%%%%%%%

\section{Basic hypergeometric $q$-difference equation: non-polynomial case}\label{sec:BHS-qdiff-nonpol}

We now consider the basic hypergeometric $q$-difference equation in 
a more general version. In the general version we cannot restrict naturally
to a simple domain. We have to take all the  general $q$-line $zq^\Z$ into
account. 

%%%%%%%%%%%%%%%%%%%%%%%%%%%%%%%%%%%%%%%%%%%%%%%%%%%%%%%%%%%%%%%%%%%%
%%%%%%%%%%%%%%%%%%%%%%%%%%%%%%%%%%%%%%%%%%%%%%%%%%%%%%%%%%%%%%%%%%%%
\subsection{Doubly infinite Jacobi operators}\label{ssec:doublyinfiniteJacobioperators}
In this section we briefly review the spectral analysis of 
a doubly infinite Jacobi operator, i.e. a three-term recurrence
on the Hilbert space $\ell^2(\Z)$. 
This section requires some knowledge from functional analysis, 
in particular of symmetric, unbounded, and self-adjoint operators and 
the spectral theorem.

We consider an operator on
the Hilbert space $\ell^2(\Z)$ of the form
\begin{equation}\label{eq:defLnonpolcase}
L\, e_k = a_k \, e_{k+1} + b_k\, e_k +
a_{k-1}\, e_{k-1}, \qquad
a_k> 0, \ b_k\in\R,
\end{equation}
where $\{ e_k\}_{k\in\Z}$ is the standard orthonormal basis
of $\ell^2(\Z)$. If $a_i=0$ for some $i\in\Z$, then $L$
splits as the direct sum of two Jacobi operators, 
so that we are essentially back to two three-term recurrence 
operators related to two sets of orthogonal 
polynomials. Recall that a three-term recurrence operator on 
$\ell^2(\N)$ is a Jacobi operator.\index{Jacobi operator}
The spectral analysis is closely related to the orthogonality of 
the corresponding orthogonal polynomials, and is essentially a proof of 
Favard's theorem, see \cite{DunfS}, \cite{Koel-Laredo}, \cite{Simo}. 
So we will assume that $a_i\not=0$ for all $i\in \Z$. 
We call $L$ a Jacobi operator\index{Jacobi operator} on $\ell^2(\Z)$ or a
doubly infinite Jacobi operator.\index{doubly infinite Jacobi operator}

The domain $\cD(L)$
of $L$ is the dense subspace $\cD(\Z)$ of finite linear
combinations of the basis elements $e_k$, $k\in \Z$.  
This makes $L$ a densely defined symmetric operator.

We extend the action of $L$ to an arbitrary vector
$v=\sum_{k=-\infty}^\infty v_ke_k\in\ell^2(\Z)$  by
$$
L^\ast \, v = \sum_{k=-\infty}^\infty (a_k \, v_{k+1}
+ b_k\, v_k + a_{k-1}\, v_{k-1})\, e_k,
$$
which is not an element of $\ell^2(\Z)$ in general.
Define
$$
{\cD}^\ast = \{ v\in\ell^2(\Z)\mid L^\ast v\in \ell^2(\Z)\}.
$$

\begin{lemma}\label{lem:maxdomdiJacobi} 
$(L^\ast, \cD^\ast)$ is the adjoint of $(L,\cD(\Z))$.
\end{lemma}

The proof of Lemma \ref{lem:maxdomdiJacobi} requires a bit of Hilbert
space theory, and we leave it to the Exercise \ref{ex:lem:maxdomdiJacobi}. 

In particular, $L^\ast$ commutes with complex conjugation,
so its deficiency indices are equal. Here the deficiency indices $(n_+,n_-)$ are
defined as\index{deficiency index}
\begin{gather*}
n_+= \dim \ker (L^\ast - z) = \dim \ker (L^\ast - i), \qquad \Im z>0 \\
n_-= \dim \ker (L^\ast - z) = \dim \ker (L^\ast + i), \qquad \Im z<0
\end{gather*}
since the dimension is constant in the upper and lower half plane. 
The solution
space of $L^\ast v=z\, v$ is two-dimensional, since $v$
is completely determined by any initial data $(v_{n-1},v_n)$
for any fixed $n\in\Z$. So the deficiency indices are
equal to $(i,i)$ with $i\in \{ 0,1,2\}$.
From the general theory of self-adjoint operators, see \cite{DunfS}, 
we have that $(L, \cD(L))$ has self-adjoint extensions since the 
deficiency indices $n_-=n_+$. In case $n_-=n_+=0$, the operator 
$(L^\ast, \cD^\ast)$ is self-adjoint, and this case will be generally assumed
in this section. 

%%%%%%%%%%%%%%%%%%%%%%%%%%%%%%%%%%%%%%%%%%%%%%%%%%%%%%%%%%%%%%%%%%%%
\subsubsection{Relation to Jacobi operators}

To the operator $L$ we associate two Jacobi operators
$J^+$ and $J^-$ acting on $\ell^2(\N)$ with orthonormal basis
denoted by $\{f_k\}_{k\in\N}$ in order to avoid
confusion. Define
\begin{equation*}
\begin{split}
J^+\, f_k &= \begin{cases}
a_k \, f_{k+1} + b_k\, f_k + a_{k-1}\, f_{k-1},&
\text{for $k\geq 1$,} \\
a_0 \, f_1 + b_0\, f_0, & \text{for $k=0$,}\end{cases} \\
J^-\, f_k &= \begin{cases}
a_{-k-2}\, f_{k+1} + b_{-k-1}\, f_k + a_{-k-1}\, f_{k-1},&
\text{for $k\geq 1$,} \\
a_{-2} \, f_1 + b_{-1}\, f_0, & \text{for $k=0$,}\end{cases}
\end{split}
\end{equation*}
and extend by linearity to $\cD(\N)$, the space
of finite linear combinations of the basis vectors
$\{f_k\}_{k=0}^\infty$ of $\ell^2(\N)$. Then $J^{\pm}$ are densely
defined symmetric operators with deficiency indices $(0,0)$
or $(1,1)$ corresponding to whether the associated Hamburger
moment problems is determinate or indeterminate, see \cite{Akhi}, 
\cite{BuchC}, 
\cite{DunfS}, \cite{Koel-Laredo}, \cite{Simo}.
The following theorem, due to Masson and Repka \cite{MassR},
relates the deficiency indices of $L$ and $J^\pm$.

\begin{thm}[Masson and Repka]\label{thm:MassonRepka}
The deficiency indices of
$L$ are obtained by summing the deficiency indices of $J^+$
and the deficiency indices of $J^-$. 
\end{thm}

For the proof of Theorem \ref{thm:MassonRepka} we refer to \cite{MassR}, \cite{Koel-Laredo}.

We define the Wronskian $[u,v]_k = a_k(u_{k+1}v_k-u_kv_{k+1})$.
\index{Wronskian}
The Wronskian is also known as the Casorati determinant.
\index{Casorati determinant}

\begin{lemma}\label{lem:WronskianLonl2Z}
The Wronskian $[u,v]=[u,v]_k$ is independent
of $k$ for $L^\ast u=z\, u$, $L^\ast v=z\, v$.
Moreover, 
$[u,v]\not=0$ if and only if $u$ and $v$ are linearly
independent solutions. 
\end{lemma}

The proof is straightforward, see Exercise \ref{ex:lem:WronskianLonl2Z}.

Now using Lemma \ref{lem:WronskianLonl2Z} 
\begin{equation*}
\begin{split}
&\sum_{k=M}^N (L^\ast u)_k \bar v_k - u_k \overline{(L^\ast v)_k}
\\ &=
\sum_{k=M}^N (a_ku_{k+1}+b_ku_k+a_{k-1}u_{k-1})\bar v_k
-u_k(a_k\bar v_{k+1}+b_k\bar v_k+a_{k-1}\bar v_{k-1}) \\ &=
\sum_{k=M}^N [u,\bar v]_k - [u,\bar v]_{k-1} =
[u,\bar v]_N-[u,\bar v]_{M-1},
\end{split}
\end{equation*}
so that, cf. Proposition \ref{prop:Lsaonell2w}, 
$$
B(u,v) = \lim_{N\to\infty} [u,\bar v]_N
-\lim_{M\to-\infty}[u,\bar v]_M, \qquad u,v \in \cD^\ast.
$$
In particular, if $J^-$ and $J^+$ are essentially self-adjoint, 
$(L^\ast, \cD^\ast)$ is self-adjoint, and then 
\[
\lim_{M\to-\infty}[u,\bar v]_M=0 \quad  \text{and} \quad \lim_{N\to\infty}[u,\bar v]_N=0.
\]

%%%%%%%%%%%%%%%%%%%%%%%%%%%%%%%%%%%%%%%%%%%%%%%%%%%%%%%%%%%%%%%%%%%%
\subsubsection{The Green kernel and the resolvent operator}

From on we
assume that $J^-$ and $J^-$ have deficiency indices $(0,0)$, so that
$J^-$ and $J^+$ are essentially self-adjoint and by
Theorem \ref{thm:MassonRepka} the deficiency indices of $L$ are
$(0,0)$. 
We refer to e.g. \cite{Koel-Laredo}, \cite{MassR}, for the case that 
one of the operators has deficiency indices $(0,0)$ and the other on $(1,1)$.
This can also be analysed in this framework.  In case $L$ has deficiency indices $(2,2)$
the restriction of the domain of a self-adjoint extension
of $L$ to the Jacobi operator $J^\pm$ does not in general
correspond to a self-adjoint extension of $J^\pm$, cf.
\cite[Thm.~XII.4.31]{DunfS}, so that this is the most difficult situation. 
We restrict ourselves to  the case of essentially self-adjoint $L$ or equivalently
that $J^\pm$ have both deficiency indices $(0,0)$, i.e.
the adjoint of $L$ is self-adjoint. 

Let $z\in\C\backslash\R$, so that we know that $L^\ast -z\Id$ has an inverse in $B(\ell^2(\Z))$,
the bounded linear operators on $\ell^2(\Z)$. The inverse is
denoted by $R(z)$, and is called the \emph{resolvent operator}.\index{resolvent operator}
Introduce the spaces
\begin{equation}\label{eq:Keq409}
\begin{split}
S^-_z &= \{ \{f_k\}_{k=-\infty}^\infty \mid
L^\ast f = z\, f\text{\ and\ } \sum_{k=-\infty}^{-1}
|f_k|^2<\infty \}, \\
S^+_z &= \{ \{f_k\}_{k=-\infty}^\infty \mid
L^\ast f = z\, f\text{\ and\ } \sum_{k=0}^\infty
|f_k|^2<\infty \}.
\end{split}
\end{equation}
Since the solution of a three-term recurrence operator is completely determined by 
two starting values $v_0$, $v_1$, we find $\dim S^\pm_z \leq 2$. The
deficiency index $n_+$, respectively $n_-$, for $L^\ast$ is precisely $\dim(S^+_z\cap S^-_z)$ for
$\Im z>0$, respectively $\Im z<0$. 
From the general theory of orthogonal polynomials we know that $\dim(S^\pm_z)\geq 1$, and 
in case of deficiency indices $(0,0)$ of $J^\pm$ we actually have $\dim(S^\pm_z)= 1$. 
Consequently, in the case of a self-adjoint $(L^\ast, \cD^\ast)$ we 
have $\dim(S^\pm_z)= 1$ and $\dim(S^+_z\cap S^-_z)=0$.

Choose $\Phi_z\in S^-_z$, so that $\Phi_z$ is determined
up to a constant. We assume $\overline{(\Phi_z)_k} =
(\Phi_{\bar z})_k$, which we can do since $L^\ast$ commutes with 
complex conjugation. 
Let $\varphi_z\in S^+_z$, such that
$\overline{(\varphi_z)_k}=(\varphi_{\bar z})_k$. We may assume 
\begin{enumerate}[1.]
\item $[\varphi_z,\Phi_z]\not= 0$,
\item $\tilde\varphi_z$, defined by $(\tilde\varphi_z)_k=0$ for
$k<0$ and $(\tilde\varphi_z)_k=(\varphi_z)_k$ for $k\geq 0$, is
contained in the domain $\cD^\ast$ of the self-adjoint $L^\ast$.
\end{enumerate}

Let $(L^\ast, \cD^\ast)$ be
the self-adjoint extension of $L$, assuming, as before, that
$J^\pm$ have deficiency indices $(0,0)$.
Let $\varphi_z\in S^+_z$, $\Phi_z\in S^-_z$ as before. 
We define the Green kernel for $z\in\C\backslash\R$ by
$$
G_{k,l}(z) = \frac{1}{[\varphi_z,\Phi_z]}\begin{cases}
(\Phi_z)_k\, (\varphi_z)_l, & k\leq l, \\
(\Phi_z)_l\, (\varphi_z)_k, & k>l.
\end{cases}
$$
So $\{ G_{k,l}(z)\}_{k=-\infty}^\infty,
\{ G_{k,l}(z)\}_{l=-\infty}^\infty \in \ell^2(\Z)$ and
$\ell^2(\Z)\ni v\mapsto G(z)v$ given by
$$
(G(z)v)_k = \sum_{l=-\infty}^\infty v_l G(z)_{k,l} =
\langle v, \overline{G_{k,\cdot}(z)}\rangle
$$
is well-defined. For $v\in\cD(\Z)$ we have
$G(z)v\in \cD^\ast$.

\begin{prop}\label{prop:GreenkerLonl2Z} 
The resolvent of $(L^\ast, \cD^\ast)$ is given by $R(z)=G(z)$ for
$z\in\C\backslash\R$.
\end{prop}

For the proof of Proposition \ref{prop:GreenkerLonl2Z} we
refer to \cite{Koel-Laredo}, and we give here the basic
calculation. For $v\in \cD(\Z)$ 
\begin{equation*}
\begin{split}
[\varphi_z,\Phi_z]&\bigl( (L^\ast -z) G(z)v\bigr)_k =
\sum_{l=-\infty}^{k-1} v_l \bigl( a_k
(\varphi_z)_{k+1}+(b_k-z)(\varphi_z)_k +a_{k-1}(\varphi_z)_{k-1}\bigr)
(\Phi_z)_l \\
& + \sum_{l=k+1}^\infty v_l \bigl( a_k
(\Phi_z)_{k+1}+(b_k-z)(\Phi_z)_k +a_{k-1}(\Phi_z)_{k-1}\bigr)
(\varphi_z)_l  \\
& + v_k\bigl( a_k(\Phi_z)_k(\varphi_z)_{k+1} +
(b_k-z)(\Phi_z)_k(\varphi_z)_k +
a_{k-1}(\Phi_z)_{k-1}(\varphi_z)_k\bigr) \\
= & v_k a_k\bigl((\Phi_z)_k(\varphi_z)_{k+1} -(\Phi_z)_{k+1}
(\varphi_z)_k\bigr) = v_k [\varphi_z,\varphi_z]
\end{split}
\end{equation*}
and canceling the Wronskian gives the result. 

% \begin{proof} Note $\C\backslash\R\subset \rho(L^\prime)$,
% because $L^\prime$ is self-adjoint. Hence $(L^\prime-z)^{-1}$
% is a bounded operator mapping $\ell^2(\Z)$ onto $\cD^\prime$.
% So $v\mapsto ((L^\prime-z)^{-1}v)_k$ is a continuous map, hence,
% by the Riesz representation theorem,
% $((L^\prime-z)^{-1}v)_k = \langle v,
% \overline{G_{k,\cdot}(z)}\rangle$ for some $G_{k,\cdot}(z)\in\ell^2(\Z)$.
% So it suffices to check $(L^\prime -z )G(z)v=v$ for $v$ in the
% dense subspace $\cD(\Z)$.
% As in the proof of Proposition \ref{3370} we have
% \begin{equation*}
% \begin{split}
% [\varphi_z,\varphi_z]&\bigl( (L^\prime -z) G(z)v\bigr)_k =
% \sum_{l=-\infty}^{k-1} v_l \bigl( a_k
% (\varphi_z)_{k+1}+(b_k-z)(\varphi_z)_k +a_{k-1}(\varphi_z)_{k-1}\bigr)
% (\varphi_z)_l \\
% & + \sum_{l=k+1}^\infty v_l \bigl( a_k
% (\varphi_z)_{k+1}+(b_k-z)(\varphi_z)_k +a_{k-1}(\varphi_z)_{k-1}\bigr)
% (\varphi_z)_l  \\
% & + v_k\bigl( a_k(\varphi_z)_k(\varphi_z)_{k+1} +
% (b_k-z)(\varphi_z)_k(\varphi_z)_k +
% a_{k-1}(\varphi_z)_{k-1}(\varphi_z)_k\bigr) \\
% = & v_k a_k\bigl((\varphi_z)_k(\varphi_z)_{k+1} -(\varphi_z)_{k+1}
% (\varphi_z)_k\bigr) = v_k [\varphi_z,\varphi_z]
% \end{split}
% \end{equation*}
% and canceling the Wronskian gives the result.
% \end{proof}

With Proposition \ref{prop:GreenkerLonl2Z} we can calculate
\begin{equation}
\langle G(z)u, v\rangle = \sum_{k,l=-\infty}^\infty
G_{k,l}(z)u_l\bar v_k =
\frac{1}{[\varphi_z,\Phi_z]} \sum_{k\leq l}
(\Phi_z)_k(\varphi_z)_l\bigl( u_l\bar v_k+u_k\bar v_l\bigr) (1-\hf
\de_{k,l}),
\end{equation}

Now the spectral theorem, see \cite[\S XII.4]{DunfS}, \cite[Ch.~13]{Rudi}, 
can be stated as follows. In particular, one sees that 
the resolvent in terms of the Green kernel gives the spectral 
decomposition by the Stieltjes-Perron inversion formula. 

\begin{thm}[Spectral theorem]\label{thmspectralthm} 
Let $T\colon\cD(T)\to \cH$ be an unbounded self-adjoint linear map 
with dense domain $\cD(T)$ in the Hilbert space $\cH$,
then there exists a unique spectral measure $E$ such that
$T=\int_\R t \, dE(t)$, i.e. $\langle Tu,v\rangle =\int_\R t \,
dE_{u,v}(t)$ for $u\in\cD(T)$, $v\in {\mathcal H}$.
Moreover, $E$ is supported on the
spectrum $\si(T)$, which is contained in $\R$.
% For any bounded operator $S$ that satisfies $ST\subset TS$ we
% have $E(A)S=SE(A)$, $A\subset \R$
% a Borel set.
Moreover, the Stieltjes-Perron inversion formula is valid;
$$
E_{u,v}\bigl( (a,b)\bigr) = \lim_{\de\downarrow 0}
\lim_{\ep\downarrow 0} \frac{1}{2\pi i}
\int_{a+\de}^{b-\de}
\langle R(x+i\ep)u,v\rangle - \langle R(x-i\ep)u,v\rangle \, dx.
$$
\end{thm}

Recall that a spectral measure $E$ is a self-adjoint orthogonal projection-valued 
measure on the Borel sets of $\R$ such that $E(\R)=\Id$, $E(\emptyset)=0$,
$E(A\cap B) = E(A)E(B)$ for Borel sets $A$ and $B$ 
and that $\si$-finite additivity with respect to the 
strong operator topology holds, i.e. for all $x\in\cH$ and any sequence
$(A_i)_{i\in\N}$ of mutually disjoint Borel sets we have
\[
E\Bigl( \bigcup_{i\in \N} A_i\Bigr)x = \sum_{i\in\N} E(A_i)x.
\]
In particular $E_{x,y}(B)= \langle E(B)x,y\rangle$ for $x,y\in\cH$ and $B$ a Borel set
gives a complex Borel measure on $\R$, which is positive in case $x=y$.

%%%%%%%%%%%%%%%%%%%%%%%%%%%%%%%%%%%%%%%%%%%%%%%%%%%%%%%%%%%%%%%%%%%%
\subsection{The basic hypergeometric
difference equation}

This example is based on Appendix A in \cite{KoelS-PublRIMS}, which
was greatly motivated by Kakehi \cite{Kake} and unpublished notes
by Koornwinder. 
% The transform described in this section
% has been obtained from its quantum $SU(1,1)$ group theoretic
% interpretation, see \cite{Kake}, \cite{KoelSsu} for references.
On a formal level the result can be obtained as a limit case
of the orthogonality of the Askey-Wilson polynomials, see
\cite{KoelS-NATO} for a precise formulation. %??take up explicitly??
% The limit transition descibed in \ref{4250} is motivated
% from the fact that the Jacobi operators in this example and
% the previous example play the same role.

We take the coefficients as
$$
a_k = \hf \sqrt{
(1-\frac{q^{-k}}{r})(1-\frac{cq^{-k}}{d^2r})}, \qquad
b_k =\frac{q^{-k}(c+q)}{2dr},
$$
where we assume $0<q<1$, and $r<0$, $c>0$, $d\in\R$.
This assumption is made in order to get the expression under
the square root sign positive. There are more possible choices
in order to achieve this, see \cite[App.~A]{KoelS-PublRIMS}.
Note that $a_k$ and $b_k$ are bounded for $k<0$, so that
$J^-$ is self-adjoint. Hence, the deficiency indices of $L$ are
$(0,0)$ or $(1,1)$ by Theorem \ref{thm:MassonRepka}.

% {\label{4175}} In order to write down solutions of
% $Lu=z\, u$ we need the basic hypergeometric series.
% Define
% $$
% (a;q)_k = \prod_{i=0}^{k-1}(1-aq^i), \qquad k\in\N\cup\{\infty\},
% \qquad
% (a_1,\ldots,a_n;q)_k = \prod_{j=1}^n (a_j;q)_k,
% $$
% and the basic hypergeometric series
% $$
% {}_2\vp_1\left( {{a,b}\atop{c}};q,x\right) =
% \sum_{k=0}^\infty \frac{(a,b;q)_k}{(q,c;q)_k} x^k.
% $$
% The radius of convergence is $1$, and there exists a unique
% analytic continuation to $\C\backslash[1,\infty)$. See
% Gasper and Rahman \cite{GaspR} for all the necessary information
% on basic hypergeometric series.

\begin{lemma}\label{lem:lemma4.5.2} Put
\begin{equation*}
\begin{split}
w_k^2 &= d^{2k} \frac{(cq^{1-k}/d^2r;q)_\infty}
{(q^{1-k}/r;q)_\infty}, \\
f_k(\mu(y)) &=
\rphis{2}{1}{dy,d/y}{c}{q,rq^k}, 
\qquad c\not\in q^{-\N},\quad \mu(y)=\hf(y+y^{-1}), \\
g_k(\mu(y)) &= q^kc^{-k} \, 
\rphis{2}{1}{qdy/c, qd/cy}{q^2/c}{q, rq^k}
\quad \mu(y)=\hf(y+y^{-1}),\ c\notin q^{2+\Z} \\
F_k(y) &= (dy)^{-k} \, 
\rphis{2}{1}{dy,qdy/c}{qy^2}{q,\frac{q^{1-k}c}{d^2r}}
% {}_2\vp_1\left({{dy,qdy/c}\atop{qy^2}};q,\frac{q^{1-k}c}{d^2r}\right), 
\qquad y^2\not\in q^{-\N},
\end{split}
\end{equation*}
then, with $z=\mu(y)$, we have that
$u_k(z)= w_kf_k(\mu(y))$, $u_k(z)= w_kg_k(\mu(y))$,
$u_k(z)=w_k F_k(y)$ and
$u_k(z)=w_kF_k(y^{-1})$
define solutions to
$$
z\, u_k(z) = a_k \, u_{k+1}(z) + b_k \, u_k(z) +
a_{k-1}\, u_{k-1}(z).
$$
\end{lemma}

\begin{proof} Put $u_k(z)=w_kv_k(z)$, then $v_k(z)$ satisfies
$$
2z\, v_k(z) = (d-\frac{cq^{-k}}{dr})\, v_{k+1}(z)
+q^{-k}\frac{c+q}{dr}\, v_k(z)
+ (d^{-1}-\frac{q^{1-k}}{dr})\, v_{k-1}(z)
%\label{4181}
$$
and this is precisely the second order $q$-difference equation
that has the solutions given, see Proposition 
\ref{prop:solofBqDEat0atinfty} and Section \ref{sec:BHS-qdiff}. 
\end{proof}

The asymptotics of the solutions
of Lemma \ref{lem:lemma4.5.2} can be given as follows.
First observe that $w_{-k} ={\mathcal O}(d^{-k})$ as $k\to\infty$, and
using
$$
 w_k^2=c^k
\frac{(r q^k, d^2r/c, cq/d^2r;q)_\infty}
{(d^2rq^k/c, r, q/ r;q)_\infty} \Rightarrow
w_k={\mathcal O}(c^{\hf k}), \ k\to\infty.
$$
Now $f_k(\mu(y))= {\mathcal O}(1)$ as $k\to\infty$, and
$g_k(\mu(y)) = {\mathcal O}((q/c)^k)$ as $k\to \infty$.
Similarly, $F_{-k}(y)={\mathcal O}((dy)^k)$ as $k\to\infty$.

\begin{prop}\label{prop:4.5.3} The operator $L$ is essentially
self-adjoint for $0<c\leq q^2$, and $L$ has deficiency
indices $(1,1)$ for $q^2<c<1$. Moreover,
for $z\in\C\backslash\R$ the one-dimensional
space $S^-_z$ is spanned by $wF(y)$ with $\mu(y)=z$ and $|y|<1$.
For $0<c\leq q^2$ the one-dimensional
space $S^+_z$ is spanned by $wf(z)$, and for $q^2<c<1$  
the two-dimensional
space $S^+_z$ is spanned by $wf(z)$ and $wg(z)$.
\end{prop}

The proof of Proposition \ref{prop:4.5.3} relies on criteria
establishing the defect indices of Jacobi operators.
We refer to \cite[p.~79]{Koel-Laredo} for the application 
of these criteria leading to Proposition \ref{prop:4.5.3}. 
Since we restrict ourselves to the self-adjoint setting, we
assume from now on that $0<c\leq q^2$. 

% \begin{proof} In \ref{4170} we have already observed that the
% deficiency indices of $L$ are $(0,0)$ or $(1,1)$. Now
% $2a_k=q^{-k}\sqrt{c/d^2r^2}-\hf(r+d^2r/c)+{\mathcal O}(q^k)$,
% $k\to\infty$, shows that
% the boundedness condition of Proposition \ref{3450}(ii)
% is satisfied if the
% coefficient of $q^{-k}$ in $a_k + a_{k-1}\pm b_k$ is non-positive.
% Since $c>0$, $dr\in\R$, this is
% the case when $(1+q)\sqrt{c}\leq c+q$.
% For $0<c\leq q^2$ the inequality holds, so that by
% Proposition \ref{3450}(ii) also $J^+$, and hence
% $L$ by Theorem \ref{435},
% is essentially self-adjoint.
% 
% {}From the asymptotic behaviour we see that $wf(z)$ and $wg(z)$
% are linearly independent solutions of the
% recurrence in Lemma \ref{4180} for $c\not=q$, and moreover that
% they both belong to $S^+_z$ for $q^2<c<1$.
% The other statements follow easily from the asymptotics described
% above. In case $c=q$ we easily see that $h_k=\lim_{c\to q}
% w_k\frac{f_k-g_k}{c-q}$ yields a second linearly independent
% solution of the
% recurrence in Lemma \ref{4180}, and $h_k={\mathcal O}(kq^{k/2})$,
% so that the argument goes through in this case.
% \end{proof}

The Wronskian
$$
[wF(y), wF(y^{-1})] = \lim_{k\to-\infty}
a_k w_{k+1}w_k\bigl(
F_{k+1}(y)F_k(y^{-1})-F_k(y)F_{k+1}(y^{-1})\bigr)
=\hf (y^{-1}-y)
$$
using $a_k\to \hf$ as $k\to-\infty$ and the asymptotics
of $F_k$ and $w_k$ as $k\to-\infty$. 
Note that the Wronskian is non-zero for $y\not= \pm 1$
or $z\not=\pm1$.
Since $wF(y)$ and $wF(y^{-1})$ are linearly independent
solutions to the
recurrence fo $L^\ast f=zf$  
for $z\in\C\backslash\R$, we see that we can express
$f_k(\mu(y))$ in terms of $F_k(y)$ and $F_k(y^{-1})$.
These solutions are related by the expansion
\begin{equation}\label{eq:Keq4191}
\begin{split}
f_k(\mu(y))
&= c(y) F_k(y) + c(y^{-1})F_k(y^{-1}), \\
c(y) &= \frac{(c/dy,d/y,dry,q/dry;q)_\infty}
{(y^{-2},c,r,q/r;q)_\infty},
\end{split}
\end{equation}
for $c\not\in q^{-\N}$,
$y^2\not\in q^\Z$, which is a reformulation 
of \eqref{eq:4.3.2}.
This shows that we have
\begin{equation}\label{eq:Wronskianexpl}
[wf(\mu(y)),wF(y)] =\hf c(y^{-1})(y-y^{-1}).
\end{equation}

Since we assume that $0<c\leq q^2$, $L$ is essentially self-adjoint, or $L^\ast$ 
is self-adjoint. Then for
$z\in\C\backslash\R$ we have
$\varphi_z = wf(z)$ and $\Phi_z=wF(y)$, where $z=\mu(y)$ and
$|y|<1$. In particular, it follows that $\varphi_{x\pm i\ep}\to
\varphi_x$ as $\ep\downarrow 0$.
For the asymptotic solution $\Phi_z$ we have
to be more careful in computing the limit of $z$ to the real axis.
For $x\in \R$ satisfying $|x|>1$ we have
$\varphi_{x\pm i\ep} \rightarrow wF_y$ as
$\ep\downarrow 0$, where
$y\in (-1,1)\backslash \{0\}$ is such that $\mu(y)=x$.
If $x\in [-1,1]$, then we put $x=\cos\chi=\mu(e^{i\chi})$
with $\chi\in [0,\pi]$, and then
$\Phi_{x-i\ep}\rightarrow wF_{e^{i\chi}}$
and $\Phi_{x+i\ep}\rightarrow wF_{e^{-i\chi}}$
as $\ep\downarrow 0$.

We calculate the integrand
in the Stieltjes-Perron inversion formula of 
Theorem \ref{thmspectralthm} 
using Lemma \ref{lem:lemma4.5.2} and 
Proposition \ref{prop:4.5.3} in the
case $|x|<1$, where $x=\cos\chi=\mu(e^{i\chi})$.
For $u,v\in {\mathcal D}(\Z)$ we have 
\begin{equation*}
\begin{split}
&\lim_{\ep\downarrow 0}\langle G(x+i\ep)u,v\rangle -
\langle G(x-i\ep)u,v\rangle = \\ & \lim_{\ep\downarrow 0}
\sum_{k\leq l}
\Bigl( \frac{(\Phi_{x+i\ep})_k(\varphi_{x+i\ep})_l}
{[\varphi_{x+i\ep},\Phi_{x+i\ep}]} -
\frac{(\Phi_{x-i\ep})_k(\varphi_{x-i\ep})_l}
{[\varphi_{x-i\ep},\Phi_{x-i\ep}]} \Bigr)
\bigl( u_l\bar v_k+u_k\bar v_l\bigr) (1-\hf
\de_{k,l})
= \\ &2\sum_{k\leq l}
\Bigl( \frac{w_kF_k(e^{-i\chi})w_lf_l(\cos\chi)}
{c(e^{i\chi})(e^{-i\chi}-e^{i\chi})} -
\frac{w_kF_k(e^{i\chi})w_lf_l(\cos\chi)}
{c(e^{-i\chi})(e^{i\chi}-e^{-i\chi})}\Bigr)
\bigl( u_l\bar v_k+u_k\bar v_l\bigr) (1-\hf
\de_{k,l})
= \\ &2\sum_{k\leq l}
\Bigl( w_kw_lf_l(\cos\chi)
\frac{c(e^{-i\chi})F_k(e^{-i\chi})+c(e^{i\chi})F_k(e^{i\chi})}
{c(e^{i\chi})c(e^{-i\chi})(e^{-i\chi}-e^{i\chi})}
\bigl( u_l\bar v_k+u_k\bar v_l\bigr) (1-\hf
\de_{k,l}) = \\
&2\sum_{k\leq l}
\Bigl( \frac{w_kw_lf_l(\cos\chi)f_k(\cos\chi)}
{c(e^{i\chi})c(e^{-i\chi})(e^{-i\chi}-e^{i\chi})}
\bigl( u_l\bar v_k+u_k\bar v_l\bigr) (1-\hf
\de_{k,l}) = \\
&\frac{2}{c(e^{i\chi})c(e^{-i\chi})(e^{-i\chi}-e^{i\chi})}
\sum_{l=-\infty}^\infty w_lf_l(\cos\chi)u_l
\sum_{k=-\infty}^\infty w_kf_k(\cos\chi)\bar v_k
\end{split}
\end{equation*}
using the expansion \eqref{eq:Keq4191} and the Wronskian in
\eqref{eq:Wronskianexpl}. Now integrate over the interval
$(a,b)$ with $-1\leq a<b\leq 1$ and replacing $x$ by $\cos\chi$, so
that $\frac{1}{2\pi i}dx = (e^{i\chi}-e^{-i\chi})d\chi/4\pi$.
We obtain, with $a=\cos \chi_a$, $b=\cos\chi_b$, and
$0\leq \chi_b<\chi_a\leq\pi$,
\begin{equation*}
\begin{split}
E_{u,v}\bigl( (a,b)\bigr) &= \frac{1}{2\pi}
\int_{\chi_b}^{\chi_a} \bigl({\mathcal F}u\bigr)(\cos\chi)
\overline{\bigl({\mathcal F}v\bigr)(\cos\chi)}
\frac{d\chi}{|c(e^{i\chi})|^2}, \\
\bigl({\mathcal F}u\bigr)(x) &= \langle u, \varphi_x\rangle =
\sum_{l=-\infty}^\infty w_lf_l(\cos\chi)u_l.
\end{split}
\end{equation*}
This shows that $[-1,1]$ is contained in the continuous
spectrum of $L$.

For $|x|>1$ we can calculate as above the integrand in the Stieltjes-Perron
inversion formula, but now we have to use that $x=\mu(y)$ with
$|y|<1$. This gives
$$
\lim_{\ep\downarrow 0}\langle G(x+i\ep)u,v\rangle
= 2\sum_{k\leq l} \frac{w_kF_k(y)w_lf_l(y)}
{c(y^{-1})(y-y^{-1})}
\bigl( u_l\bar v_k+u_k\bar v_l\bigr) (1-\hf
\de_{k,l}).
$$
The limit
$\lim_{\ep\downarrow 0}\langle G(x+i\ep)u,v\rangle$ gives the
same result,  so we can only have discrete mass points
for $|x|>1$ in the spectral measure at the zeroes of the
Wronskian, i.e. at the zeroes of $y\mapsto c(y^{-1})$
with $|y|<1$ or at $y=\pm 1$. Let us assume that all zeroes
of the $c$-function are simple, so that the spectral measure
at these points can be easily calculated.

The zeroes of the $c$-function can be read off from the
expressions in \eqref{eq:Keq4191}, and they are
\[
\{ cq^k/d\mid k\in\N\}, \quad \{ dq^k\mid k\in\N\}, \quad \{ q^k/dr\mid k\in\Z\}. 
\]
Assuming that $|c/d|<1$ and
$|d|<1$, we see that the first two sets do not contribute.
(In the more general case we have that the product is
less than $1$, since the product equals $c<1$. We leave this
extra case to the reader.)
The last set, labeled by $\Z$ always contributes to
the spectral measure. Now
for $u,v\in{\mathcal D}(\Z)$ we let $x_p=\mu(y_p)$,
$y_p=q^p/dr$, $p\in\Z$, with
$|q^p/dr|>1$, so that by the Stieltjes-Perron inversion formula
and Cauchy's residue theorem we find
$$
E_{u,v}(\{ x_p\}) = \text{Res}_{y=y_p^{-1}}
\Bigl( \frac{-1}{c(y^{-1})y}\Bigr)
w_k F_k(y_p^{-1}) w_l f_l(x_p)
\bigl( u_l\bar v_k+u_k\bar v_l\bigr) (1-\hf
\de_{k,l})
$$
after substituting $x=\mu(y)$. Now from \eqref{eq:Keq4191}
we find $f_k(x_p)=c(y_p)F_k(y_p^{-1})$, since $c(y_p^{-1})=0$
and we assume here that $c(y_p)\not=0$. Hence, we
can symmetrise the sum again and find
$$
E_{u,v}(\{ x_p\}) = \Bigl(\text{Res}_{y=y_p}
\frac{1}{c(y^{-1})c(y)y}\Bigr)
\bigl({\mathcal F}u\bigr)(x_p) \overline{\bigl({\mathcal F}v\bigr)(x_p)}
$$
switching to the residue at $y_p$.

We can combine the calculations
in the following theorem. Note that most of the regularity
conditions can be removed by continuity after calculating
explicitly all the residues. The case of an extra set of
finite mass points is left to the reader, as stated above.
Of course, there are also other possibilities for choices of the parameters $c$, $d$ and $r$
for which the expression under the square root sign in $a_k$
in \eqref{eq:defLnonpolcase} is positive. See \cite[App.~A]{KoelS-PublRIMS}
for details.

\begin{thm}\label{thm:orthorel2phi1} Assume $r<0$, $0<c\leq q^2$, $d\in\R$
with $|d|<1$ and $|c/d|<1$ such that
the zeroes of $y\mapsto c(y)$ are simple and $c(y)=0$
implies $c(y^{-1})\not= 0$. Then the spectral measure
for the Jacobi operator on $\ell^2(\Z)$ defined by 
is given by, $A\subset\R$ a Borel set,
\begin{equation*}
\begin{split}
&\langle E(A)u,v\rangle =
\int_{\cos\chi\in [-1,1]\cap A}\bigl({\mathcal F}u\bigr)(\cos\chi)
\overline{\bigl({\mathcal F}v\bigr)(\cos\chi)} \frac{d\chi}
{|c(e^{i\chi})|^2} \\ &+
\sum_{p\in\Z, |q^p/dr|>1, \mu(q^p/dr)\in A}
\Bigl(\text{{\rm Res}}_{y=q^p/dr}
\frac{1}{c(y^{-1})c(y)y}\Bigr)
\bigl({\mathcal F}u\bigr)(\mu(q^p/dr))
\overline{\bigl({\mathcal F}v\bigr)(\mu(q^p/dr))}.
\end{split}
\end{equation*}
\end{thm}

\begin{proof} It only remains to prove that $\pm 1$ is not
contained in the point spectrum. These are precisely the points
for which $F(y)$ and $F(y^{-1})$ are not linearly independent
solutions. We have to show that $\varphi_{\pm 1}\not\in\ell^2(\Z)$, and
this can be done by determining its asymptotic behaviour
as $k\to-\infty$, see \cite{Kake}, \cite{KoelS-CA} for more
information.
\end{proof}

Take $A=\R$ and $u=e_k$ and $v=e_l$, then we find the following
orthogonality relations for the ${}_2\vp_1$-series as
in Lemma \ref{lem:lemma4.5.2}. 

\begin{cor}\label{cor:thm:orthorel2phi1}
With the notation and assumptions as in Theorem \ref{thm:orthorel2phi1} we have 
\begin{equation*}
\begin{split}
&\int_0^\pi f_k(\cos\chi) f_l(\cos\chi)
\frac{d\chi}{|c(e^{i\chi})|^2} +
\\ &\sum_{p\in\Z, |q^p/dr|>1}
\Bigl({\text{\rm Res}}_{y=q^p/dr}
\frac{1}{c(y^{-1})c(y)y}\Bigr)
f_k(\mu(\frac{q^p}{dr}))f_l(\mu(\frac{q^p}{dr}))
= \frac{\de_{k,l}}{w_k^2}.
\end{split}
\end{equation*} 
\end{cor}

\begin{remark}\label{rmk:cor:thm:orthorel2phi1} 
Theorem \ref{thm:orthorel2phi1} and Corollary \ref{cor:thm:orthorel2phi1} 
have been obtained under the condition that the operator $L^\ast,\cD^\ast)$ is 
self-adjoint, or that $0<c\leq q^2$. 
In \cite{Koel-Laredo} it is shown that in case $q^2<c<1$, there exists
a self-adjoint extension of $(L,\cD(\Z))$ such that the same decomposition
in Theorem \ref{thm:orthorel2phi1} and Corollary \ref{cor:thm:orthorel2phi1}
remain valis. The case $c=q$ is a bit more intricate and requires a limiting process,
since, see Proposition \ref{prop:solofBqDEat0atinfty}, $u_1$ and $u_2$ are the same,
see \cite[App.~C]{GroeKK} for the details. 
\end{remark}

\begin{remark}\label{rmk:thm:orthorel2phi1}
The result in Corollary \ref{cor:thm:orthorel2phi1} can be viewed as 
$q$-analogue of the integral transform pair of the Jacobi functions, 
see \eqref{eq:defJacobifunctions} for the definition. The Jacobi function
transform is an integral transform pair with a ${}_2F_1$-series,
the Jacobi function, as integral kernel, see \cite{Koor-Jacobi} for details. 

Another possible option is to obtain the result of 
Corollary \ref{cor:thm:orthorel2phi1} as a limiting case of the 
orthogonality relation of the Askey-Wilson polynomials, which is 
comparable to the limit transition of the Jacobi polynomials to 
the Bessel functions, see \cite{KoelS-NATO}. 
$q$-Analogues of the Bessel functions in terms of ${}_2\varphi_1$-series 
have also been studied in \cite{Koel-ITSF93}. 
Taking a similar limit in the little $q$-Jacobi polynomials leads to the 
little $q$-Bessel functions (or ${}_1\varphi_1$-$q$-Bessel function or
as the Hahn-Exton $q$-Bessel function) studied by Koornwinder and Swarttouw
\cite{KoorS-TAMS}. These $q$-Bessel functions have been studied 
intensively, see e.g. \cite{FitoD2}, \cite{FitoD}, \cite{KoelSw} as well
as other references.
\end{remark}

%%%%%%%%%%%%%%%%%%%%%%%%%%%%%%%%%%%%%%%%%%%%%%%%%%%%%%%%%%%%%%%%%%%

\subsection{Exercises}
\begin{enumerate}[1.]
\item\label{ex:lem:maxdomdiJacobi} Prove Lemma \ref{lem:maxdomdiJacobi}. 
\item\label{ex:lem:WronskianLonl2Z} Prove Lemma \ref{lem:WronskianLonl2Z}.
\end{enumerate}

\subsection*{Notes}
The results of this section have been motivated by the paper by Kakehi \cite{Kake} 
and unpublished notes by Koornwinder. The results and techniques have been 
very useful in the study of various problems related to harmonic analysis
on the non-compact quantum group analogue of $SU(1,1)$.
In particular, we have used \cite[App.~A]{KoelS-PublRIMS}, where more general
sets of parameters have been studied, see also \cite{Koel-Laredo}. 
A special case is studied in \cite[App.~C]{GroeKK}. 
In \cite{Koel-IM} another case related to a non-selfadjoint operator is 
studied in detail. 
There is also an approach to doubly infinite Jacobi operators as in \S
\ref{ssec:doublyinfiniteJacobioperators}, due to Krein, and this is 
to relate it to a $2\times 2$-matrix valued three term recurrence on $\N$, see 
e.g. Berezanski\u\i\ \cite[Ch.~VII]{Bere}. 
This leads to the theory of matrix-valued orthogonal polynomials. 

See \cite{Koel-Laredo} for the solution to Exercise \ref{ex:lem:maxdomdiJacobi} 
and Exercise \ref{ex:lem:WronskianLonl2Z}.  
\section{Transmutation properties for the basic $q$-difference equation}\label{sec:transmutation}

In \S \ref{sec:BHS-qdiff} we have discussed the factorisation of the 
basic $q$-difference operator. The Darboux factorisation in \S \ref{sec:BHS-qdiff}
is related to a $q$-shift in both parameters. 
Here we discuss a related shift operator, but we use a relabeling of the 
parameters. Moreover, the shift is more general
and leads to a $q$-analogue of fractional integral operators and other type
of factorisations of the basic $q$-difference operator. 

We rewrite the 
second order hypergeometric $q$-difference operator 
as studied in \S \ref{sec:BHS-qdiff} as 
\begin{equation}\label{eq:KR1.1}
L=L^{(a,b)}= a^2(1+\frac{1}{x})\bigl( T_q-\text{Id}\bigr) +
(1+\frac{aq}{bx})\bigl( T^{-1}_q-\text{Id}\bigl),
\end{equation}
where $T_qf(x)=f(qx)$ for suitable functions $f$ in a suitable Hilbert space.
So we have eigenfunctions to $L$
in terms of basic hypergeometric series, see Proposition \ref{prop:solofBqDEat0atinfty}. 
The little $q$-Jacobi function\index{little $q$-Jacobi function} is defined as
\begin{equation}\label{eq:KR1.2}
\varphi_\la(x;a,b;q) = 
\rphis{2}{1}{a\si,a/\si}{ab}{q,-\frac{bx}{a}}, \quad \la=\hf(\si+\si^{-1})=\mu(\si)
\end{equation}

The little $q$-Jacobi function satisfies
\[
L\varphi_\la(\cdot;a,b;q) =
(-1-a^2+2a\la)\varphi_\la(\cdot;a,b;q). 
\]
We note that the little
$q$-Jacobi functions are eigenfunctions for the eigenvalue
$\la$ of
\begin{equation}\label{eq:KR1.4}
\cL^{(a,b)} = \frac{1}{2a} L^{(a,b)} + \hf(a+a^{-1})
= \frac{a}{2}(1+\frac{1}{x})T_q
- \bigl( \frac{a}{2x} + \frac{q}{2bx}\bigr)\text{Id}
+ \frac{1}{2a}(1+\frac{aq}{bx})T_q^{-1}.
\end{equation}
For simplicity we assume that $a,b>0$, $ab<1$ and $y>0$,
but the results
hold, mutatis mutandis, for the more general range of the
parameters as discussed in \cite[App.~A]{KoelS-PublRIMS}. Then
the operator $L$ is an unbounded symmetric operator on the
Hilbert space $\cH(a,b;y)$ of square integrable
sequences $u=(u_k)_{k\in\Z}$ with respect to the
weights
\begin{equation}\label{eq:KR1.5}
\sum_{k=-\infty}^\infty |u_k|^2 (ab)^k
\frac{(-byq^k/a;q)_\infty}{(-yq^k;q)_\infty},
% % \tag{1.5}
\end{equation}
where the operator $L$ is initially defined on the
sequences with finitely many non-zero entries, see \S \ref{sec:BHS-qdiff-nonpol},
and where $x=yq^k$.

The goal is to give a general factorisation property in 
Theorem \ref{thm:KRTheorem2.3} and Theorem \ref{thm:KRTheorem2.3ii}. As
a motivation we start by giving a Darboux factorisation of
the second order $q$-difference operator
$L^{(a,b)}$ or $\cL^{(a,b)}$, related to the one in \S \ref{sec:BHS-qdiff}.

The backward $q$-derivative\index{backward $q$-derivative operator}
operator is $B_q=M_{1/x}(1-T_q^{-1})$, where $M_g$ is the
operator of multiplication by $g$;\index{M@$M_g$ multiplication operator}
$\bigl( M_gf\bigr)(x)=g(x)f(x)$, and $T_qf(x)=f(qx)$.\index{T@$T_q$} 
Then $B_q$ is closely related to $\tilde{D}_q$ of \S \ref{sec:BHS-qdiff} with 
inverted base $q\leftrightarrow q^{-1}$. 
Now we check that 
\begin{equation}\label{eq:KR5.1}
\bigl( B_q\varphi_\la(\cdot;a,b;q)\bigr)(x)
= \frac{b(1-a\si)(1-a/\si)}{qa(1-ab)}
\, \varphi_\la(x;aq,b;q).
% \tag{5.1}
\end{equation}
Considering $\cH(a,b;y)$ as an $L^2$-space with discrete
weights $(ab)^k (-byq^k/a;q)_\infty/(-yq^k;q)_\infty$ at
the point $yq^k$, $k\in\Z$, we look at $B_q$ as a
(densely defined unbounded) operator
from $\cH(a,b;y)$ to $\cH(aq,b;y)$. Its adjoint,
up to a constant depending only on $y$,
is given by
\begin{equation}\label{eq:KR5.2}
A(a,b) = M_{1+bx/aq} - ab M_{1+x}T_q,
% \tag{5.2}
\end{equation}
and it is a straightforward calculation to show that
\begin{equation}\label{eq:KR5.3}
\bigl( A(a,b) \varphi_\la(\cdot;aq,b;q)\bigr)(x) =
(1-ab)\, \varphi_\la(x;a,b;q)
% \tag{5.3}
\end{equation}
and that
$-b L^{(a,b)} = aq A(a,b) \circ B_q$, with the
notation as in \eqref{eq:KR1.1}. This calculation is 
essentially the same as done in \S \ref{sec:BHS-qdiff}.

Since
$B_q$ and $A(a,b)$ are triangular with respect
to the standard orthogonal basis of Dirac delta's at $yq^k$
of $\cH(a,b;y)$, this means that we have
a Darboux factorisation of $L^{(a,b)}$. Also,
\[
-b(L^{(aq,b)}+(1-q)(1-qa^2))=aq^2 B_q\circ A(a,b), 
\]
from which we deduce 
\[
B_q\circ L^{(a,b)}=L^{(aq,b)}\circ B_q \quad \text{and} \quad 
L^{(a,b)}\circ A(a,b)=A(a,b)\circ L^{(aq,b)}. 
\]
It is the purpose of this section to generalize these
intertwining properties to arbitrary powers of $B_q$.

Introduce the operator $W_\nu$, $\nu\in\C$, acting on
functions defined on $[0,\infty)$ by
\begin{equation}\label{eq:KR5.4}
\bigl( W_\nu f\bigr)(x) = x^\nu \sum_{l=0}^\infty
f(xq^{-l}) q^{-l\nu} \frac{(q^\nu;q)_l}
{(q;q)_l}, \qquad x\in[0,\infty),
% \tag{5.4}
\end{equation}
assuming that the
infinite sum is absolutely convergent if $\nu\not\in -\N$.
So we want
$f$ sufficiently decreasing on a $q$-grid tending to
infinity, e.g. $f(xq^{-l}) =\cO(q^{l(\nu+\ep)})$
for some $\ep>0$. Note that for $\nu\in -\N$ the sum
in \eqref{eq:KR5.4} is finite and $W_0=\text{Id}$ and $W_{-1}=B_q$.

% \end{document}

This operator is a $q$-analogue of the Weyl fractional
integral operator\index{Weyl fractional
integral operator} 
as used in \cite[\S 3]{Koor-ArkMat75},
\cite[\S 5.3]{Koor-Jacobi} for the Abel transform.
With the notation
$$
\int_a^\infty f(t)\, d_qt = a\sum_{k=0}^\infty f(xq^{-k})q^{-k}
$$
for the $q$-integral %, cf. \eqref{eq:KR1.6}, 
we see that for
$n\in\N$ the operator $W_n$ is an iterated $q$-integral;
\begin{equation}\label{eq:KR5.5}
\bigl( W_n f\bigr)(x) = \int_x^\infty \int_{x_1}^\infty \ldots
\int_{x_{n-1}}^\infty f(x_n)\, d_qx_nd_qx_{n-1}\ldots d_qx_1.
% \tag{5.5}
\end{equation}

In the following lemma we collect some results on
$W_\nu$, where we use the function space
\begin{equation}\label{eq:KR5.6}
\cF_\rho = \{ f\colon [0,\infty)\to \C \mid\,
|f(xq^{-l})| = \cO(q^{l\rho})
,\ l\to\infty,\ \forall x\in (q,1]\},\qquad \rho>0.
% \tag{5.6}
\end{equation}
Recall that $\cL^{(a,b)}$ is defined in
\eqref{eq:KR1.4}.

% \end{document}

\begin{lemma}\label{lem:KRLemma5.1} Let
$\nu,\mu\in\C\backslash(-\N)$.
\begin{itemize}
\item[\rm (i)] $W_\nu$ preserves the
space of compactly supported functions,
\item[\rm (ii)] $W_\nu\colon
\cF_\rho\to \cF_{\rho-\Re\nu}$ for $\rho>\Re\nu>0$,
\item[\rm (iii)] $W_\nu\circ W_\mu = W_{\nu+\mu}$
on $\cF_\rho$ for $\rho>\Re(\mu+\nu)>0$,
\item[\rm (iv)] $W_\nu\circ B_q = B_q\circ W_\nu = W_{\nu-1}$
on $\cF_\rho$ for $\rho>\Re\nu-1>0$,
and $B_q^n\circ W_n = \Id$ for $n\in\N$ on $\cF_\rho$
for $\rho>n$,
\item[\rm (v)] $\cL^{(aq^{-\nu},b)}\circ W_\nu =
W_\nu \circ \cL^{(a,b)}$, valid for compactly
supported functions.
\end{itemize}
\end{lemma}

\begin{remark}
It follows from (iii) that
$W_{-n}=B_q^n$, $n\in\N$, and $W_0=\Id$.
\end{remark}

\begin{proof}
The first statement is immediate from \eqref{eq:KR5.4}.
For (ii) we use that for $f\in\cF_\rho$ and $x\in(q,1]$
we have
$$
| W_\nu f(xq^{-k})| \leq M\sum_{l=0}^\infty q^{(k+l)\rho}
q^{-(k+l)\Re\nu} \frac{(q^{\Re\nu};q)_l}{(q;q)_l} =
M q^{k(\rho-\Re\nu)}
\frac{(q^\rho;q)_\infty}{(q^{\rho-\Re\nu};q)_\infty}
$$
by the $q$-binomial theorem for
$\rho>\Re\nu$.
The third statement is a consequence of interchanging
summations, valid for $f\in\cF_\rho$, $\rho>\Re(\mu+\nu)$,
and
$$
\sum_{k+l=p} \frac{(q^\mu;q)_k (q^\nu;q)_l}
{(q;q)_k(q;q)_l} q^{-(l+k)\mu-l\nu} =
q^{-p(\mu+\nu)} \frac{(q^{\mu+\nu};q)_p}
{(q;q)_p},
$$
which is the $q$-Chu-Vandermonde summation formula
\eqref{eq:qChuVDM}.
For (iv) we note that
$B_q\colon \cF_\rho\to \cF_{\rho+1}$, then
the first statement of (iv) is a simple calculation
involving $q$-shifted factorials, which reduces
the second statement of (iv) to verifying the easy case $n=1$.
For (v) recall \eqref{eq:KR1.4},
so that $\cL^{(aq^{-\nu},b)}(W_\nu f)(x)$ and  
$W_\nu (\cL^{(a,b)}f)(x)$ involve the values $f(xq^{-k})$,
$k+1\in \N$. A straightforward calculation
using $q$-shifted factorials shows that the coefficients
of $f(xq^{-k})$ in
$\cL^{(aq^{-\nu},b)}(W_\nu f)(x)$ and  
$W_\nu (\cL^{(a,b)}f)(x)$ are equal.
 \end{proof}

The asymptotically free solution\index{asymptotically free solution} 
$\Phi_\si(yq^k;a,b;q)$ is defined 
by
\begin{equation}\label{eq:KR1.10AFSol}
\Phi_\si(yq^k;a,b;q) = (a\si)^{-k}
\rphis{2}{1}{a\si,q\si/b}{q\si^2}{q, -\frac{q^{1-k}}{y}}.
\end{equation}
so that, see \eqref{eq:4.3.2},
\begin{equation}
\begin{split}
 \varphi_\la(yq^k;a,b;q) &= c(\si;a,b;q) \Phi_\si(yq^k;a,b;q)
+ c(\si^{-1};a,b;q) \Phi_{\si^{-1}}(yq^k;a,b;q), \\
c(\si;a,b,y;q) &= \frac{(b/\si,a/\si;q)_\infty}
{(\si^{-2},ab;q)_\infty}\frac{(-by\si, -q/by\si;q)_\infty}
{(-by/a,-qa/by;q)_\infty},
% \tag{1.10}
\end{split} 
\end{equation}
valid for $\si^2\notin q^\Z$. Then $\Phi_\si$ is the asymptotically
free solution;
\[
L\Phi_\si(\cdot;a,b;q)=(-1-a^2+2a\la)\Phi_\si(\cdot;a,b;q)
\]
on $yq^\Z$
with, as before, $\la=\mu(\si)=\hf(\si+\si^{-1})$.
% Here asymptotically
% free means ???

% The measure $d\nu$
% in \thetag{1.8} can be obtained from \thetag{1.10}, see \cite{6},
% \cite{8, App.~A}. For this we now assume that $a>b$,
% which we can do without loss of generality, cf. Lemma~5.2 
% and \thetag{5.8}. Explicitly, we have

The asymptotically free solution $\Phi_\si(yq^k;a,b;q)
\in\cF_\rho$ for $q^\rho>|a\si|$ as follows from
\eqref{eq:KR1.10AFSol}. A calculation using the $q$-binomial formula gives,
cf. \eqref{eq:qbinomialthm},
\begin{equation}\label{eq:KR5.7}
\bigl(W_\nu \Phi_\si(\cdot;a,b;q)\bigr)(yq^k)=
y^\nu \frac{(a\si;q)_\infty}
{(aq^{-\nu}\si;q)_\infty}
\Phi_\si(yq^k;aq^{-\nu},b;q),
% \tag{5.7}
\end{equation}
for $|a\si|<q^\nu$ in accordance with Lemma 5.1(v).
Note that \eqref{eq:KR5.7} is a $q$-analogue of Bateman's formula,
cf. \cite{Gasp}, \cite{Koor-ArkMat75}.

\begin{lemma}\label{lem:KRLemma5.2} Define the operator
$$
S(a,b) = M_{\frac{(-x;q)_\infty}{(-bx/a;q)_\infty}}
\circ T_{b/a},\qquad T_{b/a}f(x) = f(\frac{b}{a}x),
$$
then $S(a,b)^{-1}\circ \cL^{(a,b)} \circ
S(a,b) = \cL^{(b,a)}$.
In particular, $\tilde W_\nu^{(a,b)} = S(a,bq^{-\nu})\circ
W_\nu\circ S(a,b)^{-1}$ satisfies the intertwining property
$\cL^{(a,bq^{-\nu})}\circ \tilde W_\nu^{(a,b)} =
\tilde W_\nu^{(a,b)} \circ \cL^{(a,b)}$.
\end{lemma}

Note that $S(a,b)^{-1}=S(b,a)$ and that $S(a,b)\colon
\cH(b,a;yb/a)\to \cH(a,b;y)$ is an isometric
isomorphism.
For $f\in\cF_\rho$ we see that
$\bigl(S(a,b)f\bigr)(xq^{-l}) = \cO( |a/b|^lq^{l\rho})$, so that
$S(a,b)f\in \cF_{\rho + \ln(|a/b|)/\ln q}$.

\begin{proof} It follows from \eqref{eq:KR1.4} that
\begin{multline*}
\cL^{(a,b)}\bigl(x\mapsto \frac{(-x;q)_\infty}
{(-bx/a;q)_\infty}f(x)\bigr)(x) = \\ \frac{(-x;q)_\infty}
{(-bx/a;q)_\infty} \Bigl( \frac{b}{2}(1+\frac{a}{bx})f(qx)
+ \frac{1}{2b}(1+\frac{q}{x})f(xq^{-1}) - \hf
(\frac{a}{x}+\frac{q}{bx})f(x)\Bigr)
\end{multline*}
and the term in parentheses can be written as
$T_{b/a}\circ \cL^{(b,a)}\circ T_{a/b}$ applied to $f$.
The second statement then follows from Lemma 5.1(v).
\end{proof}

It follows directly from \eqref{eq:KR1.2}, \eqref{eq:KR1.10AFSol} and 
the last equation of \eqref{eq:1.4.4},
\begin{equation}\label{eq:KR5.8}
\aligned
\bigl( S(a,b)\varphi_\la(\cdot;b,a;q)\bigr)(x)&= \varphi_\la(x;a,b;q), \\
\bigl( S(a,b)\Phi_\si(\cdot;b,a;q)\bigr)(x)&= \Phi_\si(x;a,b;q).
\endaligned
% \tag{5.8}
\end{equation}

\begin{thm} \label{thm:KRTheorem2.3} 
% {\rm (i)} 
Let
$a,b\in\C\backslash\{0\}$, $\nu,\mu\in\C$ with
$|q^{\nu-\mu}b/a|<1$. Define the operator
\begin{multline*}
\bigl(W_{\nu,\mu}(a,b)f\bigr)(x) =
\frac{(-x;q)_\infty}{(-xq^{-\mu};q)_\infty}
q^{-\mu^2}\bigl( \frac{b}{a}\bigr)^\mu x^{\mu+\nu}
\\ \times \sum_{p=0}^\infty
f(xq^{-\mu-p})\, q^{-p\nu}\frac{(q^\nu;q)_p}{(q;q)_p}
\, \rphis{3}{2}
{q^{-p}, q^{-\mu},-q^{1+\mu-\nu}a/bx}{q^{1-p-\nu},-q^{\mu+1}/x}{q,q^{1-\mu}\frac{b}{a}}
\end{multline*}
for any function $f$ with $|f(xq^{-p})|={\cO}(q^{p(\ep+\nu)})$
for some $\ep>0$.
Then 
\[
W_{\nu,\mu}(a,b)\circ
{\cL}^{(a,b)}= {\cL}^{(aq^{-\nu},bq^{-\mu})}
\circ W_{\nu,\mu}(a,b) 
\]
on the space of compactly supported
functions and for $|a\si|<q^\nu$
$$
\bigl(W_{\nu,\mu}(a,b)\Phi_\si(\cdot;a,b;q)\bigr)(yq^k) =
y^{\mu+\nu}
\frac{(a\si,b\si;q)_\infty}
{(aq^{-\nu}\si,bq^{-\mu}\si;q)_\infty}
\Phi_\si(yq^k;aq^{-\nu},bq^{-\mu};q).
$$
% {\rm (ii)}
% Let $a,b>0$, $ab<1$, $\nu>0$ and
% $\mu\in\C\setminus\Z_{\leq 0}$. Define the operator
% \begin{multline*}
% \bigl( A_{\nu,\mu}(a,b)f\bigr)(x) =
% \frac{(-bxq^\mu/a;q)_\infty}{(-bxq^{\mu-\nu}/a;q)_\infty}
% \\ \times \sum_{k=0}^\infty f(xq^{\mu+k})\,
% (ab)^k \frac{(q^\nu,-xq^\mu;q)_k}{(q,-bxq^\mu/;q)_k}
% \, {}_3\vp_2\left( {{q^{-k}, q^\mu,-bxq^{\mu-\nu}/a}
% \atop{q^{1-\nu-k},-xq^\mu}};q,q\right)
% \end{multline*}
% for any bounded function. Then
% ${\cL}^{(aq^\nu,bq^\mu)}\circ A_{\nu,\mu}(a,b) =
% A_{\nu,\mu}(a,b)\circ {\cL}^{(a,b)}$
% on the space of functions compactly supported in
% $(0,\infty)$. Moreover,
% $$
% \bigl(A_{\nu,\mu}(a,b) \varphi_\la(\cdot;a,b;q)\bigr)(x) =
% \frac{(abq^{\nu+\mu};q)_\infty}{(ab;q)_\infty}
% \, \varphi_\la(x;aq^\nu,bq^\mu;q).
% % \tag{2.2}
% $$
\end{thm}

\begin{proof}%[Proof of the first statement of Theorem 2.3??]
It follows from Lemma \ref{lem:KRLemma5.1}(v) and Lemma \ref{lem:KRLemma5.2}
that the operator
\begin{align*}
W_{\nu,\mu}(a,b) &= \tilde W_\mu^{(aq^{-\nu},b)}\circ W_\nu 
= S(aq^{-\nu},bq^{-\mu})\circ W_\mu\circ
S(b,aq^{-\nu})\circ W_\nu
\end{align*}
satisfies the required interwining property. For $f\in {\cF}_\rho$ with $\rho>\Re\nu$  we can interchange summations,
which leads
to the sum with a terminating ${}_3\vp_2$ as kernel. Note that
the ${}_3\vp_2$-series in the kernel of $W_{\nu,\mu}(a,b)$
behaves as
$$
\rphis{2}{1}{q^{-\mu},-q^{1+\mu-\nu}a/bx}{-q^{1+\mu}/x}{q, q^{\nu-\mu}\frac{b}{a}}
$$
as $p\to\infty$.

The statement for the action on $\Phi_\si(\cdot;a,b;q)$ follows
immediately from \eqref{eq:KR5.7} and \eqref{eq:KR5.8}.
\end{proof}

The results in Theorem \ref{thm:KRTheorem2.3} deal with the 
fractional $q$-derivative $W_\nu$ related to the point at $\infty$, and these 
operators act nicely on the eigenfunctions $\Phi_\si$  at $\infty$ of the operator $\cL^{(a,b)}$.
We want to have similar statements on a suitable intertwining operator that 
acts nicely on the eigenfunctions $\vp_\la$ at $0$ of the operator $\cL^{(a,b)}$.
In order to get results in this direction, see Theorem \ref{thm:KRTheorem2.3ii},   we
take appropriate adjoints of the previous construction.
Consider $W_\nu$, $\nu\in\C\setminus (-\N)$, as a
densely defined unbounded operator from
$\cH(aq^\nu,b;y)$ to $\cH(a,b;y)$ and define
$R_\nu^{(a,b)}$ as its adjoint, so
\begin{equation}\label{eq:KR5.9}
\langle R^{(a,b)}_\nu f, g\rangle_{\cH(aq^\nu,b;y)} =
\langle f, W_\nu g\rangle_{\cH(a,b;y)}
% \tag{5.9}
\end{equation}
for all compactly supported functions $g$, cf. Lemma 5.1(i).
Here we use the identification of $\cH(a,b;y)$ as a weighted
$L^2$-space on a discrete set, see \S 1. A $q$-integration by parts
shows
\begin{equation}\label{eq:KR5.10}
\bigl(R_\nu^{(a,b)} f\bigr)(yq^p)= y^\nu
\frac{(-byq^p/a;q)_\infty}{(-byq^{p-\nu}/a;q)_\infty}
\sum_{l=0}^\infty f(yq^{p+l}) (ab)^l
\frac{(q^\nu,-yq^p;q)_l}{(q,-byq^p/a;q)_l}.
% \tag{5.10}
\end{equation}
Now define, for functions $f$, the operator
\begin{equation}\label{eq:KR5.11}
\bigl(A_\nu^{(a,b)} f\bigr)(x)=
\frac{(-bx/a;q)_\infty}{(-bxq^{-\nu}/a;q)_\infty}
\sum_{l=0}^\infty f(xq^l) (ab)^l
\frac{(q^\nu,-x;q)_l}{(q,-bx/a;q)_l},
% \tag{5.11}
\end{equation}
so that $A_\nu^{(a,b)}\Big\vert_{{\cH}(a,b;y)}=y^{-\nu}R_\nu^{(a,b)}$. Note that
$A_\nu^{(a,b)}$ is well-defined for bounded functions
assuming $|ab|<1$.
Recall that the dense
domain of finite linear combinations of the basis
vectors for $\cL^{(a,b)}$ corresponds to the
functions compactly supported in $(0,\infty)$.

\begin{lemma}\label{lem:KRLemma 5.3}
$\cL^{(aq^{\nu},b)}\circ A_\nu^{(a,b)} =
A_\nu^{(a,b)} \circ \cL^{(a,b)}$
on the space of functions
compactly supported in $(0,\infty)$. Moreover,
$$
\bigl(A_\nu^{(a,b)}\varphi_\la(\cdot;a,b;q)\bigr)(x) =
 \frac{(abq^\nu;q)_\infty}{(ab;q)_\infty}
\varphi_\la(x;aq^\nu,b;q).
$$
Defining $\tilde A_\nu^{(a,b)} = S(a,bq^\nu)\circ
A_\nu^{(b,a)}\circ S(b,a)$ we have
$\cL^{(a,bq^\nu)}\circ \tilde A_\nu^{(a,b)} =
\tilde A_\nu^{(a,b)} \circ \cL^{(a,b)}$, and
$$
\bigl(\tilde A_\nu^{(a,b)}\varphi_\la(\cdot;a,b;q)\bigr)(x) =
\frac{(abq^\nu;q)_\infty}{(ab;q)_\infty}
\varphi_\la(x;a,bq^\nu;q).
$$
\end{lemma}

\begin{proof} Note that \eqref{eq:KR5.10} and \eqref{eq:KR5.11} show that
the operators $R_\nu^{(a,b)}$ and $A_\nu^{(a,b)}$ preserve
the space of functions compactly supported in $(0,\infty)$.
The intertwining property for $R_\nu^{(a,b)}$
follows from \eqref{eq:KR5.9} and
Lemma 5.1, and hence for $A_\nu^{(a,b)}$.

To calculate the action of $A_\nu^{(a,b)}$ on the
little $q$-Jacobi function we use 
the last equation of \eqref{eq:1.4.4}
% ??Heine's formulation??
% \cite{4, (1.4.6)} 
to write
\begin{equation}\label{eq:KR5.12}
\varphi_\la(x;a,b;q) = \frac{(-x;q)_\infty}{(-bx/a;q)_\infty}
\, \rphis{2}{1}{b\si,b/\si}{ab}{q,-x}
% \tag{5.12} 
\end{equation}
Using this in \eqref{eq:KR5.11}, interchanging summations,
which is easily justified for $|x|<1$,
and using the $q$-binomial theorem gives
$$
\bigl(A_\nu^{(a,b)}\varphi_\la(\cdot;a,b;q)\bigr)(x) =
\frac{(abq^\nu,-x;q)_\infty}
{(ab,-bxq^{-\nu}/a;q)_\infty}
\, \rphis{2}{1}{b\si,b/\si}{abq^\nu}{q, -x}
% \, {}_2\vp_1\left( {{b\si,b/\si}\atop{abq^\nu}};q, -x\right)
% \tag{5.13}
$$
and using \eqref{eq:KR5.12} again gives the result for $|x|<1$.
The general case follows by analytic continuation in $x$,
see  \eqref{eq:KR1.10AFSol}, 
since the convergence in \eqref{eq:KR5.11} for $f$ the little $q$-Jacobi
function is uniform on compact sets for $x$.

The
statements for $\tilde A_\nu^{(a,b)}$ follow from the
corresponding
statements for $A_\nu^{(a,b)}$ and Lemma 5.2 and
\eqref{eq:KR5.8}.
\end{proof}

\begin{thm}\label{thm:KRTheorem2.3ii} 
% {\rm (i)} Let
% $a,b\in\C\backslash\{0\}$, $\nu,\mu\in\C$ with
% $|q^{\nu-\mu}b/a|<1$. Define the operator
% \begin{multline*}
% \bigl(W_{\nu,\mu}(a,b)f\bigr)(x) =
% \frac{(-x;q)_\infty}{(-xq^{-\mu};q)_\infty}
% q^{-\mu^2}\bigl( \frac{b}{a}\bigr)^\mu x^{\mu+\nu}
% \\ \times \sum_{p=0}^\infty
% f(xq^{-\mu-p})\, q^{-p\nu}\frac{(q^\nu;q)_p}{(q;q)_p}
% \, {}_3\vp_2\left( {{q^{-p}, q^{-\mu},-q^{1+\mu-\nu}a/bx}\atop
% {q^{1-p-\nu}.-q^{\mu+1}/x}};q,q^{1-\mu}\frac{b}{a}\right)
% \end{multline*}
% for any function $f$ with $|f(xq^{-p})|={\cO}(q^{p(\ep+\nu)})$
% for some $\ep>0$.
% Then $W_{\nu,\mu}(a,b)\circ
% {\cL}^{(a,b)}= {\cL}^{(aq^{-\nu},bq^{-\mu})}
% \circ W_{\nu,\mu}(a,b)$ on the space of compactly supported
% functions and for $|a\si|<q^\nu$
% $$
% \bigl(W_{\nu,\mu}(a,b)\varphi_\si(\cdot;a,b;q)\bigr)(yq^k) =
% y^{\mu+\nu}
% \frac{(a\si,b\si;q)_\infty}
% {(aq^{-\nu}\si,bq^{-\mu}\si;q)_\infty}
% \varphi_\si(yq^k;aq^{-\nu},bq^{-\mu};q).
% $$
% {\rm (ii)}
Let $a,b>0$, $ab<1$, $\nu>0$ and
$\mu\in\C\setminus\Z_{\leq 0}$. Define the operator
\begin{multline*}
\bigl( A_{\nu,\mu}(a,b)f\bigr)(x) =
\frac{(-bxq^\mu/a;q)_\infty}{(-bxq^{\mu-\nu}/a;q)_\infty}
\\ \times \sum_{k=0}^\infty f(xq^{\mu+k})\,
(ab)^k \frac{(q^\nu,-xq^\mu;q)_k}{(q,-bxq^\mu/;q)_k}
\, \rphis{3}{2}{q^{-k}, q^\mu,-bxq^{\mu-\nu}/a}{q^{1-\nu-k},-xq^\mu}{q,q}
% \, {}_3\vp_2\left( {{q^{-k}, q^\mu,-bxq^{\mu-\nu}/a}
% \atop{q^{1-\nu-k},-xq^\mu}};q,q\right)
\end{multline*}
for any bounded function. Then
${\cL}^{(aq^\nu,bq^\mu)}\circ A_{\nu,\mu}(a,b) =
A_{\nu,\mu}(a,b)\circ {\cL}^{(a,b)}$
on the space of functions compactly supported in
$(0,\infty)$. Moreover,
$$
\bigl(A_{\nu,\mu}(a,b) \varphi_\la(\cdot;a,b;q)\bigr)(x) =
\frac{(abq^{\nu+\mu};q)_\infty}{(ab;q)_\infty}
\, \varphi_\la(x;aq^\nu,bq^\mu;q).
% \tag{2.2}
$$
\end{thm}

\begin{proof} %[Proof of the second statement of Theorem ??2.3]
Define
\begin{align*}
A_{\nu,\mu}(a,b) &= \tilde A^{(aq^\nu,b)}_\mu\circ A_\nu^{(a,b)}\\
&= S(aq^\nu,bq^\mu)\circ A_\mu^{(b,aq^\nu)}\circ
S(b,aq^\nu)\circ A_\nu^{(a,b)}
\end{align*}
then it follows from Lemma 5.3 that the intertwining
property is valid. The action on a function $f$ can be calculated
and for $f$ compactly suppported in $(0,\infty)$
we find the explicit result
with the ${}_3\vp_2$-series as kernel.
We can extend the result to bounded $f$ if we require $\nu>0$.

The action of $A_{\nu,\mu}(a,b)$ on the little $q$-Jacobi
function follows from Lemma 5.3.
\end{proof}

These results can be used to obtain several identities 
involving the kernels of the transforms $A_\nu^{(a,b)}$ 
and $W_{\nu,\mu}(a,b)$, involving the 
transform of \S \ref{sec:BHS-qdiff-nonpol}. 
We refer to \cite{KoelR-RMJM02} for examples.

%%%%%%%%%%%%%%%%%%%%%%%%%%%%%%%%%%%%%%%%%%%%%%%%%%%%%%%%%%%%%%%%%%%
%%%%%%%%%%%%%%%%%%%%%%%%%%%%%%%%%%%%%%%%%%%%%%%%%%%%%%%%%%%%%%%%%%%

\subsection{Exercises}

\begin{enumerate}[1.]
\item Gasper's $q$-analogue \cite[(1.8)]{Gasp}
of Erd\'elyi's fractional integral is 
\begin{gather*}
\rphis{2}{1}{ar\si,ar/\si}{abrs}{q,
-\frac{byq^l}{ar}} =
\frac{(ab,rs;q)_\infty}{(q,abrs;q)_\infty}\sum_{k=0}^\infty
(ab)^k \frac{(q^{k+1},-byq^{k+l}/a;q)_\infty}
{(rsq^k,-byq^{k+l}/ar;q)_\infty} \\ \times
\rphis{3}{2}{q^{-k},r,ar/b}{rs, -arq^{1-l-k}/by}{q,q}
\, \rphis{2}{1}{a\si,a/\si}{ab}{q,
-\frac{byq^{l+k}}{a}}
\end{gather*}
for $|rs|<1$, $|ab|<1$.
Derive this from Theorem \ref{thm:KRTheorem2.3ii}. 
\end{enumerate}

%%%%%%%%%%%%%%%%%%%%%%%%%%%%%%%%%%%%%%%%%%%%%%%%%%%%%%%%%%%%%%%%%%%
%%%%%%%%%%%%%%%%%%%%%%%%%%%%%%%%%%%%%%%%%%%%%%%%%%%%%%%%%%%%%%%%%%%
\subsection*{Notes}
The result of this section are based on \cite{KoelR-RMJM02}, and 
they focus on the fractional analogues of the $q$-derivative 
for the asymptotically free solution of the second order
$q$-difference equation. Several other results related to this 
factorisation of the $q$-difference equation are presented in 
\cite{KoelR-RMJM02}. 
For the classical situation this is related to factoring the 
Jacobi function transform as a product of the Abel transform 
followed by the (standard) Fourier transform, see \cite{Koor-Jacobi} 
for details. 
% There are various related expansion formulae in \cite{KoelR-RMJM02}. 

% %%%%%%%%%%%%%%%%%%%%%%%%%%%%%%%%%%%%%%%%%%%%%%%%%%%%%%%%%%%%%%%%%%%
% %%%%%NEW SECTION%%%%%%%%%%%%%%%%%%%%%%%%%%%%%%%%%%%%%%%%%%%%%%%%%%%
% %%%%%%%%%%%%%%%%%%%%%%%%%%%%%%%%%%%%%%%%%%%%%%%%%%%%%%%%%%%%%%%%%%%
% 
% \section{Basic hypergeometric $q$-difference equation: matrix-valued}\label{sec:BHS-qdiff-MV}
% 
% Aldenhoven, de los Rios, Koelink??
% 
% Nikki Jaspers??
% 

%%%%%%%%%%%%%%%%%%%%%%%%%%%%%%%%%%%%%%%%%%%%%%%%%%%%%%%%%%%%%%%%%%%
%%%%%NEW SECTION%%%%%%%%%%%%%%%%%%%%%%%%%%%%%%%%%%%%%%%%%%%%%%%%%%%
%%%%%%%%%%%%%%%%%%%%%%%%%%%%%%%%%%%%%%%%%%%%%%%%%%%%%%%%%%%%%%%%%%%

% \subsection*{Notes}

%%%%%%%%%%%%%%%%%%%%%%%%%%%%%%%%%%%%%%%%%%%%%%%%%%%%%%%%%%%%%%%%%%%
%%%%%NEW SECTION%%%%%%%%%%%%%%%%%%%%%%%%%%%%%%%%%%%%%%%%%%%%%%%%%%%
%%%%%%%%%%%%%%%%%%%%%%%%%%%%%%%%%%%%%%%%%%%%%%%%%%%%%%%%%%%%%%%%%%%

% \subsection{Exercises}
% \begin{enumerate}[1.]
% \item 
% \end{enumerate}

%%%%%%%%%%%%%%%%%%%%%%%%%%%%%%%%%%%%%%%%%%%%%%%%%%%%%%%%%%%%%%%%%%%
%%%%%NEW SECTION%%%%%%%%%%%%%%%%%%%%%%%%%%%%%%%%%%%%%%%%%%%%%%%%%%%
%%%%%%%%%%%%%%%%%%%%%%%%%%%%%%%%%%%%%%%%%%%%%%%%%%%%%%%%%%%%%%%%%%%

\section{Askey-Wilson level}\label{sec:AW-level}

At the level of the Askey-Wilson polynomials and Askey-Wilson 
functions one considers the second-order $q$-difference operator

%%%%%%%%%%%%%%%%%%%%%%%%%%%%%%%%%%%%%%%%%%%%%%%%%%%%%%%%%%%%%%%%%%%
%%%%%%%%%%%%%%%%%%%%%%%%%%%%%%%%%%%%%%%%%%%%%%%%%%%%%%%%%%%%%%%%%%%
\subsection{Askey-Wilson polynomials}\label{ssec:AWpols}

In this section we briefly recall the basic properties of the
Askey-Wilson polynomials. We formulate these properties 
using the concept of duality. 

The Askey-Wilson polynomials 
$p_n(x)=p_n(x;a,b,c,d;q)$, $n\in \N$, 
are defined by 
\begin{equation}\label{eq:KSpn}
p_n(x) = p_n(x;a,b,c,d\mid q)= \rphis{4}{3}{q^{-n},  q^{n-1}abcd, ax, ax^{-1}}{ab, ac, ad}{q,q}
\end{equation}
see \cite{AskeW}. Note that $p_n$ is a polynomial in $z=\mu(x) = \frac12(x+x^{-1})$, but 
we consider it as Laurent polynomial in $x$. 
Usually, see \cite{AskeW}, \cite{KoekLS}, \cite{KoekS}, the normalization is chosen differently
in order to make the Askey-Wilson polynomials symmetric in $a$, $b$, $c$ and $d$.
The Askey-Wilson polynomials $\{p_n\}_{n\in\N}$
form a basis of the polynomial algebra $\C[z]=\C[x+x^{-1}]$
consisting of eigenfunctions of the Askey-Wilson second order 
$q$-difference operator
\begin{equation}\label{eq:KSLAW}
% \begin{split}
% L&=\alpha(x)(T_q-1)+\alpha(x^{-1})(T_q^{-1}-1),\\
% \alpha(x)&=\frac{(1-ax)(1-bx)(1-cx)(1-dx)}{(1-x^2)(1-qx^2)},
% \end{split}
L=\alpha(x)(T_q-1)+\alpha(x^{-1})(T_q^{-1}-1),\qquad 
\alpha(x)=\frac{(1-ax)(1-bx)(1-cx)(1-dx)}{(1-x^2)(1-qx^2)},
\end{equation}
where $(T_q^{\pm 1}f)(x)=f(q^{\pm 1}x)$.

The eigenvalue of $L$ corresponding to the Askey-Wilson polynomial
$p_n$ is $\mu(\gamma_n)$, where
$\gamma_n=\tilde{a}q^n$,  $\tilde{a}=\sqrt{q^{-1}abcd}$, and 
\begin{equation}\label{eq:KSeigenvalue}
\bigl(Lp_n\bigr)(x) = \mu(\gamma_n) p_n(x), \qquad 
\mu(\ga)=-1-\tilde{a}^2+\tilde{a}(\ga+\ga^{-1}).
\end{equation}

In order to describe the orthogonality relations concisely we recall the 
dual parameters to $(a,b,c,d)$. We extend the definition of $\tilde{a}$ to  
\begin{equation}\label{eq:KSdual}
\tilde{a}=\sqrt{q^{-1}abcd},\qquad \tilde{b}=ab/\tilde{a}=q\tilde{a}/cd,\qquad 
\tilde{c}=ac/\tilde{a}=q\tilde{a}/bd,\qquad \tilde{d}=ad/\tilde{a}=
q\tilde{a}/bc.
\end{equation}

\begin{lemma}\label{lem:KSdualdomain}
The assignment 
$(a,b,c,d,t)\mapsto (\tilde{a},\tilde{b},\tilde{c},\tilde{d},
\tilde{t})$
defined by \eqref{eq:KSdual} is an involution.
\end{lemma}

Lemma \ref{lem:KSdualdomain} follows by calculation.

The orthogonality relations for the Askey-Wilson polynomials 
hold quite generally, see \cite{AskeW}, \cite{GaspR}, \cite{KoekS}, but we 
assume $0<q<1$ as usual and moreover that $a,b,c$ and $d$ are positive and less than one.
Then Askey and Wilson \cite{AskeW} proved the orthogonality relations
\begin{equation}\label{eq:KSorthoAW}
\frac{1}{2\pi i C_0}\int_{x\in {\mathbb{T}}}p_n(x)p_m(x)\Delta(x)\frac{dx}{x}
=\delta_{m,n}\frac{\underset{x=\gamma_0}{\hbox{Res}}\left(
\frac{\widetilde{\Delta}(x)}{x}\right)}
{\underset{x=\gamma_n}{\hbox{Res}}\left(\frac{\widetilde{\Delta}(x)}{x}\right)}
\end{equation}
where $\de_{m,n}$ is the Kronecker delta and $\T$
is the counterclockwise oriented unit circle in the complex plane,
with the weight function given by
\begin{equation}\label{eq:KSDelta}
\De(x)=\frac{\bigl(x^2,1/x^{2};q\bigr)_{\infty}}
{\bigl(ax,a/x,bx,b/x,cx,c/x,dx,d/x;q\bigr)_{\infty}},
\end{equation}
and with $\widetilde{\De}(x)$ the weight function $\De(x)$
with respect to dual parameters. Here the positive normalization
constant $C_0$ is given by the Askey-Wilson 
integral
\[
C_0=\frac{1}{2\pi i}\int_{x\in {\mathbb{T}}}\Delta(x)\frac{dx}{x}=
\frac{2\bigl(abcd;q\bigr)_{\infty}}
{\bigl(q,ab,ac,ad,bc,bd,cd;q\bigr)_{\infty}}.
\]
Various different proofs of the Askey-Wilson integral exist, 
see e.g. references in \cite{Isma-LN}. The original 
proof follows by an elaborate residue calculus, and now there are many different
approaches to the Askey-Wilson integral as well as to its various extensions,
see \cite{GaspR} for references. 

Having the dual parameters \eqref{eq:KSdual}, 
the explicit expression \eqref{eq:KSpn} for the Askey-Wilson polynomials
show that the duality relation
\begin{equation}\label{eq:KSduality}
p_n(aq^m;a,b,c,d;q)=p_m(\tilde{a}q^n;\tilde{a},\tilde{b},\tilde{c},
\tilde{d};q),\qquad m,n\in \N
\end{equation}
holds. 
The deeper understanding of the duality \eqref{eq:KSdual} 
stems from affine Hecke algebraic 
considerations, see \cite{NoumS}. This duality takes an even nicer form 
in the case of the Askey-Wilson functions, see \ref{ssec:AWfnctiont}. 
The duality \eqref{eq:KSduality} also shows that 
the three-term recurrence relations for the Askey-Wilson polynomials
$p_m(\cdot;\tilde{a},\tilde{b},\tilde{c},\tilde{d};q)$, 
$m\in \N$, follows from the  
eigenvalue equations $Lp_n=\mu(\ga_n)p_n$, 
$n\in\N$, by applying the duality \eqref{eq:KSduality}, 
see e.g. \cite{AskeW} and  \cite{NoumS}.

The orthogonality relations written in the form \eqref{eq:KSorthoAW} 
exhibit the duality \eqref{eq:KSduality} of the Askey-Wilson polynomials
on the level of the orthogonality relations, since it expresses 
the quadratic norms explicitly in terms of the dual 
weight function $\widetilde{\Delta}$.
This description of the quadratic norms was proved in 
\cite{NoumS}.

% ??work out factorisation??

% ??Ismail and Rahman?? 

%%%%%%%%%%%%%%%%%%%%%%%%%%%%%%%%%%%%%%%%%%%%%%%%%%%%%%%%%%%%%%%%%%%
%%%%%%%%%%%%%%%%%%%%%%%%%%%%%%%%%%%%%%%%%%%%%%%%%%%%%%%%%%%%%%%%%%%
\subsection{Askey-Wilson function transform}\label{ssec:AWfnctiont}

We define the Askey-Wilson function transform
and we state the main result concerning the 
Askey-Wilson function transform. 
For this we more generally need to consider general, i.e. non-polynomial,
eigenfunctions to 
\begin{equation}\label{eq:KSeigenvalueequation}
\bigl(Lf\bigr)(x)=\mu(\gamma)f(x),
\end{equation}
which reduces to the Askey-Wilson polynomial for $\gamma=\gamma_n$, $n\in\N$,
and enjoys the same duality properties.
The solutions of \eqref{eq:KSeigenvalueequation} have been studied by
Ismail and Rahman \cite{IsmaR-TAMS}.  

Two linearly independent 
solutions of the eigenvalue equation \eqref{eq:KSeigenvalueequation}
can be derived from Ismail's and Rahman's 
\cite[(1.11)--(1.16)]{IsmaR-TAMS} solutions for the
three term recurrence relation of the associated 
Askey-Wilson polynomials. The solutions are given in terms of very well poised
${}_8\varphi_7$ series, in particular 
\begin{equation}\label{eq:KSphi}
\begin{split}
\phi_{\gamma}(x)=&\phi_\ga(x;a,b,c; d\mid q) =\frac{\bigl(qax\gamma/\tilde{d}, qa\gamma/\tilde{d}x;
q\bigr)_{\infty}}
{\bigl(\tilde{a}\tilde{b}\tilde{c}\gamma, q\gamma/\tilde{d},
q\tilde{a}/\tilde{d}, qx/d, q/dx;q\bigr)_{\infty}}\\
&\qquad\times{}_8W_7\bigl(\tilde{a}\tilde{b}\tilde{c}\gamma/q; 
ax, a/x, \tilde{a}\gamma,
\tilde{b}\gamma, \tilde{c}\gamma;q,q/\tilde{d}\gamma\bigr),\qquad
|q/\tilde{d}\gamma|<1
\end{split}
\end{equation}
is a solution to \eqref{eq:KSeigenvalueequation}. 
This solution is called the \emph{Askey-Wilson function}.\index{Askey-Wilson function}
Lemma \ref{lem:KSreductionAWfunction} shows that the Askey-Wilson function
satisfies the same duality, and moreover extends the Askey-Wilson 
polynomials of \eqref{ssec:AWpols}. 

\begin{lemma}\label{lem:KSreductionAWfunction}
The Askey-Wilson function satisfies the duality and reduction formulas
\begin{gather*}
\phi_{\gamma}(x;a;b,c;d\mid q)=\phi_x(\gamma;\tilde{a};\tilde{b},
\tilde{c};\tilde{d}\mid q) \\
\phi_{\gamma_n}(x)=\frac{1}
{\bigl(bc,qa/d,q/ad;q\bigr)_{\infty}}
p_n(x),\qquad n\in \N.
\end{gather*}
and $\phi_{\gamma^{\pm 1}}(x^{\pm 1})
=\phi_{\gamma}(x)$ for all possible choices. 
\end{lemma}

Note that duality 
\eqref{eq:KSduality} for the Askey-Wilson polynomials is a special case of
the duality  of $\phi_{\gamma}$ in Lemma \ref{lem:KSreductionAWfunction}.

\begin{proof} 
The proof rests on a formula expressing a very-well poised
${}_8\varphi_7$-series as a sum of two balanced ${}_4\varphi_3$-series 
given by Bailey's formula \cite[(III.36)]{GaspR}. 
This gives 
\begin{equation}\label{eq:KS43presentation}
\begin{split}
\phi_{\gamma}(x)=&\frac{1}
{\bigl(bc,qa/d,q/ad;q\bigr)_{\infty}}
{}_4\phi_3\left(\begin{matrix} ax, a/x, \tilde{a}\gamma, \tilde{a}/\gamma\\
 ab, ac, ad \end{matrix}; q,q\right)\\
+&\frac{\bigl(ax, a/x, \tilde{a}\gamma, \tilde{a}/\gamma, qb/d, 
qc/d;q\bigr)_{\infty}}
{\bigl(qx/d, q/dx, q\gamma/\tilde{d}, q/\tilde{d}\gamma, 
ab,ac,bc,qa/d,ad/q;q\bigr)_{\infty}}\\
&\qquad\qquad\times
{}_4\phi_3\left(\begin{matrix} qx/d, q/dx, q\gamma/\tilde{d},
q/\tilde{d}\gamma\\
 qb/d, qc/d, q^2/ad \end{matrix}; q,q\right),
\end{split}
\end{equation}
hence
$\phi_{\gamma}(x)$ extends to a meromorphic function in $x$ and $\gamma$
for generic parameters $a,b,c$ and $d$, with possible poles at
$x^{\pm 1}=q^{1+k}/d$, $k\in\N$, and
$\gamma^{\pm 1}=q^{1+k}/\tilde{d}$, $k\in \N$.
It follows from \eqref{eq:KS43presentation} that
$\phi_{\gamma^{\pm 1}}(x^{\pm 1})
=\phi_{\gamma}(x)$ (all sign combinations possible), and that
$\phi_{\gamma}$ satisfies the duality relation by inspection. 
% \begin{equation}\label{dualityfunction}
% \phi_{\gamma}(x;a;b,c;d;q)=\phi_x(\gamma;\tilde{a};\tilde{b},
% \tilde{c};\tilde{d};q).
% \end{equation}

Finaly, observe that the meromorphic continuation \eqref{eq:KS43presentation}
of $\phi_{\gamma}(x)$ implies that 
\begin{equation}\label{redpol}
\phi_{\gamma_n}(x)=\frac{1}
{\bigl(bc,qa/d,q/ad;q\bigr)_{\infty}}
p_n(x),\qquad n\in\N.
\end{equation}

Indeed, the factor $\bigl(\tilde{a}/\gamma;q\bigr)_{\infty}$
in front of the second ${}_4\varphi_3$ in \eqref{eq:KS43presentation} 
vanishes for $\gamma=\gamma_n=
\tilde{a}q^n$ for $n\in\N$. 
\end{proof}

At this stage we need to specify
a particular parameter domain for the five parameters $(a,b,c,d,t)$
in order to ensure positivity of measures.

%%%%%%%%%%%%%%%%%%%%%%%%%%%%%%%%%%%%%%%%%%%%%%%%%%%%%%%%%%%%%%%%%%%%
\begin{defn}\label{def:AWsetV}
Let $V$ be the set of parameters $(a,b,c,d,t)\in \R^{5}$
satisfying 
\begin{equation*}
t<0,\qquad0<b,c\leq a<d/q,\qquad bd,cd\geq q,\qquad ab,ac<1.
\end{equation*}
\end{defn}

%%%%%%%%%%%%%%%%%%%%%%%%%%%%%%%%%%%%%%%%%%%%%%%%%%%%%%%%%%%%%%%%%%%%%%%%
Observe that $b,c<1$ and $d>q$ for all $(a,b,c,d,t)\in V$.
We extend the duality of \eqref{eq:KSdual} to 
\begin{equation}\label{eq:KSdualt}
\tilde{t}=1/qadt.
\end{equation}

The domain $V$ is self-dual extending Lemma \ref{lem:KSdualdomain}. 

%%%%%%%%%%%%%%%%%%%%%%%%%%%%%%%%%%%%%%%%%%%%%%%%%%%%%%%%%%%%%%%%%%%%%%%%
\begin{lemma}\label{lem:KS2dualdomain}
The assignment 
$(a,b,c,d,t)\mapsto (\tilde{a},\tilde{b},\tilde{c},\tilde{d},
\tilde{t})$
defined by \eqref{eq:KSdual} and \eqref{eq:KSdualt}, is an involution on $V$.
\end{lemma}

%%%%%%%%%%%%%%%%%%%%%%%%%%%%%%%%%%%%%%%%%%%%%%%%%%%%%%%%%%%%%%%%%%%%%%%
\begin{proof}
Again by direct verification.
\end{proof}
%%%%%%%%%%%%%%%%%%%%%%%%%%%%%%%%%%%%%%%%%%%%%%%%%%%%%%%%%%%%%%%%%%%%%%%

From now on we consider $(a,b,c,d,t)\in V$ fixed. 

In order to motivate the measure which we will introduce we look at other 
solutions for the eigenvalue equation for $L$. 
Observe that the eigenvalue equation \eqref{eq:KSeigenvalueequation}
is asymptotically of the form
\begin{equation}\label{asymptequation}
\tilde{a}^2\bigl(f(qx)-f(x)\bigr)+\bigl(f(q^{-1}x)-f(x)\bigr)=\mu(\gamma)f(x)
\end{equation}
when $|x|\rightarrow \infty$.
For generic $\gamma$, 
the asymptotic eigenvalue equation \eqref{asymptequation} has
a basis $\{\Phi_{\gamma}^{\free}, \Phi_{\gamma^{-1}}^{\free}\}$ of solutions
on the $q$-line $I =\{ dtq^k\}_{k\in\Z}$, where
\[
\Phi_{\ga}^{\free}(dtq^k)=\bigl(\tilde{a}\gamma)^{-k},\qquad
k\in\Z.
\]
Furthermore, for generic $\ga$ there exists a unique solution 
$\Phi_{\ga}(x)$
of the eigenvalue equation \eqref{eq:KSeigenvalueequation} on $I$
of the form $\Phi_{\gamma}(x)=\Phi_{\gamma}^{\free}(x)g(x)$,
where $g$ has a convergent power series expansion 
around $\infty$ with constant
coefficient equal to one. The solution $\Phi_{\gamma}$ is the 
\emph{asymptotically
free solution}\index{asymptotically free solution}
of the eigenvalue equation \eqref{eq:KSeigenvalueequation}.

Actually, an explicit expression  for 
$\Phi_{\gamma}$ can be obtained from the study of Ismail and Rahman on
the associated Askey-Wilson polynomials, in which they study solutions 
of the eigenvalue equation  \eqref{eq:KSeigenvalueequation}.
Starting with \cite[(1.13)]{IsmaR-TAMS} and applying the transformation
formula \cite[(III.23)]{GaspR} for very well poised ${}_8\varphi_7$'s we obtain 
\begin{equation}
\begin{split}
\Phi_{\gamma}(x)=
&\frac{\bigl(qa\gamma/\tilde{a}x,qb\gamma/\tilde{a}x,qc\gamma/\tilde{a}x,
q\tilde{a}\gamma/dx,d/x;q\bigr)_{\infty}}
{\bigl(q/ax,q/bx,q/cx,q/dx,q^2\gamma^2/dx;q\bigr)_{\infty}}\\
&\times{}_8W_7\bigl(q\gamma^2/dx;q\gamma/\tilde{a},q\gamma/\tilde{d},
\tilde{b}\gamma,\tilde{c}\gamma,q/dx;q,d/x\bigr)\Phi_{\gamma}^{free}(x)
\end{split}
\end{equation}
for $x\in I$ with $|x|\gg 0$.
We now expand the Askey-Wilson function $\phi_{\ga}(x)$
as a linear combination of the asymptotically free solutions 
$\Phi_{\ga}(x)$ and $\Phi_{\ga^{-1}}(x)$ for $x\in I$ with
$|x|\gg 0$. Since these are all solutions to the same eigenvalue equation,
we can expect a relation with coefficients being constants or $q$-periodic functions.

%%%%%%%%%%%%%%%%%%%%%%%%%%%%%%%%%%%%%%%%%%%%%%%%%%%%%%%%%%%%%%%%%%%
\begin{prop}\label{prop:KScprop}
Let $x\in I$ with $|x|\gg 0$. Then we have the $c$-function expansion
\begin{equation*} %\label{cexpansion}
\phi_{\gamma}(x)=\widetilde{c}(\gamma)\Phi_{\gamma}(x)+
\widetilde{c}(\gamma^{-1})
\Phi_{\gamma^{-1}}(x)
\end{equation*}
for generic $\ga$, where the $c$-function is given by\index{c@$c$-function}
\begin{equation*} % \label{c}
c(\ga)= c(\ga;a;b,c;d;q;t) = \frac{1}{\bigl(ab,ac,bc,qa/d;q\bigr)_{\infty}\theta(qadt)}
\frac{\bigl(a/\gamma,b/\gamma,
c/\gamma;q\bigr)_{\infty}\theta(\gamma/dt)}
{\bigl(q\gamma/d,1/\gamma^2;q\bigr)_{\infty}}.
\end{equation*}
using the notation \eqref{eq:thetafunction} 
and $\widetilde{c}(\ga)=c(\ga;\tilde{a};\tilde{b},\tilde{c}; \tilde{d};q;\tilde{t})$. 
\end{prop}

We call $\widetilde{c}(\gamma)=
c(\gamma;\tilde{a};\tilde{b},\tilde{c};\tilde{d};q;\tilde{t})$
the dual $c$-function, with the dual parameter $\tilde{t}$ defined by
\eqref{eq:KSdualt}.\index{dual $c$-function}

\begin{proof} The proof requires some calculation. The essential ingredients
are as follows. 
First apply Bailey's three term recurrence relation \cite[(III.37)]{GaspR}
with its parameters specialized as 
\[
a\rightarrow q\gamma^2/dx,\quad
b\rightarrow q/dx,\quad 
c\rightarrow q\gamma/\tilde{a},\quad 
d\rightarrow q\gamma/\tilde{d},\quad
e\rightarrow \tilde{b}\gamma,\quad 
f\rightarrow \tilde{c}\gamma. 
\]
This gives an expansion of the required form with explicit
coefficients $\widetilde{c}(\gamma)$ and $\widetilde{c}(\gamma^{-1})$,
which at a first glance still depend on $x\in I$. 
Using the theta function \eqref{eq:thetafunction} and its functional 
equation \eqref{eq:thetafunction} we see that the coefficients 
are independent of $x$.
\end{proof}

For the moment we furthermore 
assume that 
$x\mapsto 1/c(x)c(x^{-1})$ only has simple poles. 
This imposes certain generic conditions on the parameters $(a,b,c,d,t)$, 
which can be removed at a later stage by a continuity argument.

It is convenient to renormalize the
function $1/c(x)c(x^{-1})$ as follows,
\begin{equation}\label{eq:KSWeight}
W(x)=\frac{1}{c(x)c(x^{-1})c_0}
=\frac{\bigl(qx/d,q/dx,x^2,1/x^{2};q\bigr)_{\infty}}
{\bigl(ax,a/x,bx,b/x,cx,c/x;q\bigr)_{\infty}\theta(dtx)\theta(dt/x)},
\end{equation}
where $c_0$ is the positive constant
\begin{equation*} %\label{c0funct}
c_0=\frac{\bigl(ab,ac,bc,qa/d;q\bigr)_{\infty}^2\theta(adt)^2}{a^2}.
\end{equation*}
It follows from \eqref{eq:KSWeight} and \eqref{eq:KSDelta} that 
\begin{equation}\label{eq:KSdifference}
W(x)=\frac{\theta(dx)\theta(d/x)}{\theta(dtx)\theta(dt/x)}\Delta(x).
\end{equation}
By \eqref{eq:thetafunction}, the quotient
of theta functions in \eqref{eq:KSdifference} is a $q$-periodic function. 
In particular,  
the weight function $W(x)$ differs from $\Delta(x)$ 
only by a $q$-periodic function, but this factor introduces additional poles which
arise in the orthogonality (or spectral) measure.

Let $S$ be the discrete subset  
\begin{equation}\label{eq:KSS}
\begin{split}
S&=\{x\in {\mathbb{C}} \,\, | \,\, 
|x|>1, c(x)=0\}=S_+\cup S_-,\\
S_+&=\{aq^k \,\, | \,\, k\in\N,\,\, aq^k>1\},\\
S_-&=\{dtq^k \,\, | \,\, k\in \Z,\,\, dtq^k<-1\}.
\end{split}
\end{equation} 
By $\widetilde{S}$ and $\widetilde{S}_{\pm}$ we denote the subsets $S$ 
and $S_{\pm}$ with respect to dual parameters.
We define a measure $\nu=\nu(\cdot;a;b,c;d;t;q)$ by
\begin{equation}\label{eq:KSnu}
\begin{split}
\int f(x)&d\nu(x)=\frac{K}{4\pi i}\int_{x\in {\mathbb{T}}}f(x)
W(x)\frac{dx}{x}\\
&+\frac{K}{2}\sum_{x\in S}f(x)\underset{y=x}{\hbox{Res}}
\left(\frac{W(y)}{y}\right)
-\frac{K}{2}\sum_{x\in S^{-1}}f(x)
\underset{y=x}{\hbox{Res}}
\left(\frac{W(y)}{y}\right),
\end{split}
\end{equation}
where the positive constant
$K$ is given by 
\begin{equation}\label{eq:KSK}
K=\bigl(ab,ac,bc,qa/d,q;q\bigr)_{\infty}\sqrt{\frac{\theta(qt)\theta(adt)
\theta(bdt)\theta(cdt)}{qabcdt^2}}.
\end{equation}
This particular choice of normalization constant for the measure $\nu$
is justified in Theorem \ref{thm:KSmain}, since the corresponding transform
is made an isometry.

In view of \eqref{eq:KSdifference}, we can relate the 
discrete masses $\nu\bigl(\{x\}\bigr)(=-\nu\bigl(\{x^{-1}\}\bigr))$ 
for $x\in S_+$ to residues of the weight function $\Delta(\cdot)$,
which were written down explicitly in \cite{AskeW}, see also \cite[(7.5.22)]{GaspR}
in order to avoid a small misprint in \cite{AskeW}. Explicitly, we obtain 
for $x=aq^k\in S_+$ with $k\in\N$ the expression 
\begin{equation}\label{eq:KSweightplus}
\nu\bigl(\{aq^k\}\bigr)=\frac{\bigl(qa/d,q/ad,1/a^2;q\bigr)_{\infty}}
{\bigl(q,ab,b/a,ac,c/a;q\bigr)_{\infty}\theta(adt)\theta(dt/a)}
\frac{(1-a^2q^{2k})}{(1-a^2)}\frac{K}{2\tilde{a}^{2k}}
\end{equation}
for the corresponding discrete weight. 
For fixed $k\in\N$, the right hand side of \eqref{eq:KSweightplus}
gives the unique continuous extension of
the discrete weight $\nu\bigl(\{aq^k\}\bigr)$ and 
$-\nu\bigl(\{a^{-1}q^{-k}\}\bigr)$ to all parameters 
$(a,b,c,d,t)\in V$ satisfying $aq^k>1$. Furthermore, 
the (continuously extended) 
discrete weight $\nu\bigl(\{aq^k\}\bigr)$ is strictly positive
for these parameter values. Note that $S_+$  gives a finite number of discrete mass
points in the measure $\nu$. 

A similar argument can be applied for the discrete weights
$\nu\bigl(\{x\}\bigr)(=-\nu\bigl(\{x^{-1}\}\bigr))$ with $x\in S_-$. 
Explicitly we obtain for $x=dtq^k\in S_-$
with $k\in\Z$,
\begin{equation}\label{eq:KSweighttoinfty}
\begin{split}
\nu\bigl(\{dtq^k\}\bigr)=&\frac{\bigl(qt,q/d^2t;q\bigr)_{\infty}}
{\bigl(q,q,a/dt,b/dt,c/dt,adt,bdt,cdt;q\bigr)_{\infty}}
\\
&\times
\frac{\bigl(1/t,a/dt,b/dt,c/dt;q\bigr)_{-k}}
{\bigl(q/adt,q/bdt,q/cdt,q/d^2t;q\bigr)_{-k}}
\left(1-\frac{1}{d^2t^2q^{2k}}\right)\frac{K\tilde{a}^{2k}}{2}.
\end{split}
\end{equation}
As for $\nu\bigl(\{x\}\bigr)$ with $x\in S_+$, we use the 
right hand side of \eqref{eq:KSweighttoinfty} to define 
the strictly positive weight
$\nu\bigl(\{dtq^k\}\bigr)(=-\nu\bigl(\{d^{-1}t^{-1}q^{-k}\}\bigr))$ 
for all $(a,b,c,d,t)\in V$ satisfying  $dtq^k<-1$. 
Note that $S_-$  gives an infinite number of discrete mass
points in the measure $\nu$.

We see that the definition \eqref{eq:KSnu} of the
measure $\nu$  can be extended to arbitrary parameters 
$(a,b,c,d,t)\in V$ using the continuous extensions of its 
discrete weights in \eqref{eq:KSweightplus}, \eqref{eq:KSweighttoinfty}. 
The resulting measure $\nu$ is a 
positive measure for all $(a,b,c,d,t)\in V$.

%%%%%%%%%%%%%%%%%%%%%%%%%%%%%%%%%%%%%%%%%%%%%%%%%%%%%%%%%%%%%%%%%%%%%%
\begin{defn}\label{def:KSdefhilbert}
Let $\cH=\cH(a;b,c;d;t;q)$ be the Hilbert space consisting of $L^2$-functions
$f$ with respect to $\nu$ which
satisfy $f(x)=f(x^{-1})$ $\nu$-almost everywhere.
\end{defn}
%%%%%%%%%%%%%%%%%%%%%%%%%%%%%%%%%%%%%%%%%%%%%%%%%%%%%%%%%%%%%%%%%%%%%%

We write $\widetilde{\nu}$ for the measure $\nu$ with respect to dual 
parameters $(\tilde{a},\tilde{b},\tilde{c},\tilde{d},\tilde{t})$,
and $\widetilde{\cH}$ for the associated Hilbert space 
${\cH}$. 

Let ${\cD}\subset \cH$
be the dense subspace of functions $f$ with compact support, i.e.
\[ 
{\cD}=\{f\in \cH \,\, | \,\, f(dtq^{-k})=0,\ k\gg 0\},
\]
and define
\begin{equation}\label{eq:KSF}
\bigl({\cF}f\bigr)(\ga)=\int f(x){\overline{\phi_{\gamma}(x)}}
d\nu(x),\qquad f\in \cD
\end{equation}
for generic $\ga\in \C\setminus \{0\}$. 

% %%%%%%%%%%%%%%%%%%%%%%%%%%%%%%%%%%%%%%%%%%%%%%%%%%%%%%%%%%%%%%%%%%%%%%
% \begin{remark}
% Observe that the analytic continuation \eqref{43presentation} of 
% $\phi_{\gamma}(x)$ is not defined for parameters $(a,b,c,d,t)\in V$
% satisfying  $\theta(ad)=0$. These apparent poles can be removed
% in view of the original definition \eqref{phi} for 
% $\phi_{\gamma}(x)=\phi_{\gamma^{-1}}(x)$
% in terms of very well poised ${}_8\phi_7$'s
% (observe that $q/\tilde{d}<1$, so that
% either $|q\gamma/\tilde{d}|<1$ or $|q/\tilde{d}\gamma|<1$ for 
% $\gamma\in {\mathbb{C}}\setminus \{0\}$, hence $\phi_{\gamma}$ can be
% expressed in terms of the original definition \eqref{phi}
% for generic $\gamma$). In particular, the transform ${\mathcal{F}}$
% is well defined for parameters satisfying $\theta(ad)=0$.
% \end{remark}
% %%%%%%%%%%%%%%%%%%%%%%%%%%%%%%%%%%%%%%%%%%%%%%%%%%%%%%%%%%%%%%%%%%%%%%%
We write ${\widetilde{\mathcal{D}}}\subset \widetilde{{\mathcal{H}}}$ 
(respectively $\widetilde{\mathcal{F}}$) for the dense subspace ${\mathcal{D}}$
(respectively the function transform $\mathcal{F}$)
with respect to dual parameters 
$(\tilde{a},\tilde{b},\tilde{c},\tilde{d},\tilde{t})$.

%%%%%%%%%%%%%%%%%%%%%%%%%%%%%%%%%%%%%%%%%%%%%%%%%%%%%%%%%%%%%%%%%%%%%%%
\begin{thm}\label{thm:KSmain}
Let $(a,b,c,d,t)\in V$.
The transform ${\cF}$ extends to an isometric isomorphism
${\cF} \colon {\cH}\to \widetilde{\cH}$
by continuity. The inverse of ${\cF}$ is given by 
${\widetilde{\cF}} \colon \widetilde{{\cH}} \to {\cH}$.
\end{thm}

The isometric isomorphism $\cF\colon {\cH}\rightarrow
\widetilde{\cH}$ is called the \emph{Askey-Wilson function transform}.
\index{Askey-Wilson function transform}

We will not prove Theorem \ref{thm:KSmain} in these notes, but we 
refer to the original proof in \cite{KoelS-IMRN2001}.

%%%%%%%%%%%%%%%%%%%%%%%%%%%%%%%%%%%%%%%%%%%%%%%%%%%%%%%%%%%%%%%%%%%
%%%%%%%%%%%%%%%%%%%%%%%%%%%%%%%%%%%%%%%%%%%%%%%%%%%%%%%%%%%%%%%%%%%

\subsection{Exercises}

\begin{enumerate}[1.]
\item\label{ex:AW1} Give a factorisation for the second order $q$-difference
operator \eqref{eq:KSdifference} in terms of a lowering and raising operator. 
% \item 
\end{enumerate}

%%%%%%%%%%%%%%%%%%%%%%%%%%%%%%%%%%%%%%%%%%%%%%%%%%%%%%%%%%%%%%%%%%%
%%%%%%%%%%%%%%%%%%%%%%%%%%%%%%%%%%%%%%%%%%%%%%%%%%%%%%%%%%%%%%%%%%%

\subsection*{Notes}
The Askey-Wilson polynomials have been introduced by Askey and Wilson
in \cite{AskeW}, and these polynomials are on top of the continuous 
part of the $q$-analogue of the Askey scheme, see \cite{AskeW}, 
\cite{KoekLS}, \cite{KoekS}. The discrete counterpart, the $q$-Racah 
polynomials are on top of the discrete part of 
the $q$-analogue of the Askey scheme. The solutions of the 
second order $q$-difference equation 
as considered here were studied by Ismail and Rahman \cite{IsmaR-TAMS},
where they studied the associated Askey-Wilson polynomials.
The corresponding Askey-Wilson function transform as in Theorem 
\ref{thm:KSmain} is due to \cite{KoelS-IMRN2001}. 
It is remarkable that the $q=1$ analogue of this transform, 
the Wilson function transform of Groenevelt \cite{Groe-WT}, is 
only established after the $q$-analogue of the Askey-Wilson function transform.
Moreover, Groenevelt's Wilson transform comes into two versions, which 
also map Wilson polynomials to Wilson polynomials (with dual parameters).
The $q$-analogue of this statement is related to the expansion of the 
Askey-Wilson function in terms of a series involving a product of 
the Askey-Wilson polynomials and the Askey-Wilson polynomials with 
dual parameters, see Stokman \cite{Stok-JAT}. 
Other approaches to the same or closely related functions can be found in e.g.
Haine and Iliev \cite{HainI}, Ruijsenaars \cite{Ruij-AW} and Suslov \cite{Susl}. 
As to multivariable extensions, Stokman \cite{Stok-AoM} established 
the analogue of Proposition \ref{prop:KScprop}. 
For Exercise \ref{ex:AW1} one can e.g. consult \cite{GaspR}, \cite{KoekS}. 
There seems not to be an appropriate analogue of the Frobenius method for the 
Askey-Wilson difference equation, see \cite{IsmaS} for related Taylor  
expansions. 

%%%%%%%%%%%%%%%%%%%%%%%%%%%%%%%%%%%%%%%%%%%%%%%%%%%%%%%%%%%%%ira%%%%%%
%%%%%NEW SECTION%%%%%%%%%%%%%%%%%%%%%%%%%%%%%%%%%%%%%%%%%%%%%%%%%%%
%%%%%%%%%%%%%%%%%%%%%%%%%%%%%%%%%%%%%%%%%%%%%%%%%%%%%%%%%%%%%%%%%%%

\section{Matrix-valued extensions}\label{sec:MVextensions}

The hypergeometric differential equation has a 
matrix-valued analogue, studied by Tirao \cite{Tira}. 
This matrix-valued differential equation plays an 
important role in the study of matrix-valued 
spherical functions on (especially compact) symmetric spaces 
and extensions of these result to more general 
matrix-valued orthogonal polynmials,
see e.g. \cite{GrunPT}
\cite{KoelvPR1}, \cite{KoelvPR2}, \cite{KoeldlRR}.
See also \cite{DamaPS}  
and the lecture notes \cite{Koel-OPSFA2016}, as well as references
given there, 
for more information on matrix-valued orthogonal polynomials. 
This section is of a preliminary nature and based
on the Bachelor thesis of Nikki Jaspers \cite{Jasp} 
and finds its origin in the paper \cite{AldeKR}. 
The results can be considered as $q$-analogues of the solutions of 
the matrix-valued hypergeometric series at $0$ and $\infty$, see 
\cite{Tira} and \cite{RomaS}. 

\subsection{Vector-valued basic hypergeometric $q$-difference equation}

Recall the basic hypergeometric equation \eqref{eq:ex1.13a}, which we now
adapt to 
\begin{equation}\label{eq:vectorvBHqDE}
(q-z)\Id f(q^{-1}z)+\bigl((A+B)z-C-q\Id\bigr)f(z)+(C-ABz)f(qz) = 0
\end{equation}
where $A,B,C\in \End(\C^N)$,\index{End@$\End(\C^N)$} i.e. linear maps from $\C^N$ to $\C^N$, 
and  $f\colon \C\to \C^N$ is the unknown vector-valued function,
which we want to satisfy \eqref{eq:vectorvBHqDE}.
Note that in particular in \eqref{eq:vectorvBHqDE} the identity 
$\Id\in \End(\C^N)$ and $0\in \C^N$. 

Note that $N=1$ brings us back to \eqref{eq:ex1.13a} and 
the results presented in these lecture notes, but also the case
of commuting diagonalizable $A$, $B$ and $C$ bring us back 
to \eqref{eq:ex1.13a}. 

Generically the dimension of the solution space of 
the vector-valued basic $q$-difference equation \eqref{eq:vectorvBHqDE}
is $2N$.

\begin{remark}\label{rmk:eq:vectorvBHqDE} More generally, we 
can consider \eqref{eq:vectorvBHqDE} with $A+B$ and $AB$ replaced 
by more generally $U$ and $V$. In the case $N=1$ this is 
equivalent, but in the vector-valued case this more general.
In the case of the hypergeometric differential operator
this is dicussed by Tirao \cite{Tira}. We will
not discuss this case, see \cite{Jasp} for the
$q$-case. 
\end{remark}

Note that we can generalize \eqref{eq:vectorvBHqDE} to 
\begin{equation}\label{eq:matrixvBHqDE}
(q-z)\Id F(q^{-1}z)+\bigl((A+B)z-C-q\Id\bigr)F(z)+(C-ABz)F(qz) = 0
\end{equation}
where $A,B,C\in \End(\C^N)$
and $F\colon \C\to \End(\C^N)$ is the unknown matrix-valued function,
which we want to satisfy \eqref{eq:matrixvBHqDE}.
Note that in \eqref{eq:matrixvBHqDE} now $0\in\End(\C^N)$ is the 
zero-matrix. 
Moreover, if $F(z)$ is any solution to \eqref{eq:matrixvBHqDE},
then $f(z) = F(z)f_0$ for a fixed vector $f_0\in \C^N$ satisfies 
\eqref{eq:vectorvBHqDE}.

\subsection{Solutions of matrix-valued q-hypergeometric equation}
The Frobenius method can be extended to  case of 
\eqref{eq:vectorvBHqDE}. We first consider expansions around $z=0$. 
For a matrix $C\in \End(\C^N)$ we let $\si(C)$ denote its spectrum, i.e. 
the zeros of the characteristic polynomial of $C$. 
For $A,B,C\in \End(\C^N)$ with $q^{-\N}\cap \si(C)=\emptyset$ we define for 
$n\in \N$ the product
\[
\big(A,B;C;q\big)_n = \displaystyle\prod_{k=0}^{\substack{n-1\\ \gets}}(I-q^kC)^{-1}(I-q^kA)(I-q^kB)
\]
where 
\[
 \displaystyle\prod_{k=0}^{\substack{n-1\\ \gets}} a_k = a_{n-1}\cdots a_0
\]
for  non-commuting elements $a_k$. 
In case $A,B,C$ are $1\times 1$-matrices, this reduces to $\frac{(A;q)_n(B(;q)_n}{(C;q)_n}$.

For $\al\in \C$ we define the matrix-valued basic hypergeometric series\index{matrix-valued basic hypergeometric series} 
\begin{equation}\label{eq:MVbasichypseries}
{}_2\Phi_1^\al(A,B;C;q,z)= \displaystyle\sum_{n=0}^\infty\big(\al A, \al B; \al C;q\big)_n
\frac{z^n}{(\al q;q)_n}
\end{equation}
assuming $q^{-\N}\cap \si(\al C)=\emptyset$.  In case $\al=1$, we drop it from the notation, i.e.
${}_2\Phi_1(A,B;C;q,z)={}_2\Phi_1^1(A,B;C;q,z)$. 

Note that obvious symmetry $A\leftrightarrow B$ of the scalar case no longer holds, since
$AB\not= BA$ in general.

We can now describe the solutions to \eqref{eq:vectorvBHqDE} in the generic case, i.e. when the 
eigenvalues of $C$ are sufficiently generic. 

\begin{thm}\label{thm:MVBHSatzero}
Assume $C$ is diagonalizable, so that $\si(C)=\{c_1,\cdots ,c_N\}$ with $c_i\not=c_j$ for $i\not=j$. 
Assume $c_i\not=0$ for all $i$. 
Let $f_i\not=0$ be the corresponding eigenvectors of $C$; $Cf_i=c_i f_i$, $1\leq i \leq N$. 
Assume furthermore that $\si(C)\cap q^\Z=\emptyset$ and that $c_i/c_j \notin q^\Z$ for  $i\not= j$.
Then the vector-valued functions 
\[
{}_2\Phi_1(A,B;C;q,z) e_i
\]
for any basis $\{e_1,\cdots,e_N\}$ of $\C^N$ and the vector-valued functions
\[
z^{1-\log_q(c_i)}\, {}_2\Phi_1^{q/c_i}(A,B;C;q,z)f_i, \qquad 1\leq i \leq N
\]
span the solution space of \eqref{eq:vectorvBHqDE} (over the $q$-periodic functions).
\end{thm}

\begin{proof} The proof follows the Frobenius method for the 
basic hypergeometric $q$-difference equation. So we assume that we have a 
solution of the form $\sum_{n=0}^\infty f_n z^{n+\mu}$, $f_n\in \C^N$, where 
$f_0\not=0$. Plugging this Ansatz into \eqref{eq:vectorvBHqDE} we get 
\begin{gather*}
0= (q-z)\Id \sum_{n=0}^\infty f_n(q^{-1}z)^{n+\mu} +\bigl((A+B)z-(C+q\Id)\bigr)
\sum_{n=0}^\infty f_n z^{n+\mu} + (C-ABz)\sum_{n=0}^\infty f_n (qz)^{n+\mu}  \\ 
= \bigl( f_0q^{1-\mu} - \bigl(C+q)  f_0  + q^\mu C f_0  \bigr) z^{\mu} 
+ \\
\sum_{n=1}^\infty z^{\mu+n} 
\Bigl( \bigl( Cq^{\mu+n} -(C+q) + q^{1-\mu-n}\bigr)f_n + \bigl( -ABq^{\mu+n-1} 
+ A+B - q^{1-\mu-n}\bigr)f_{n-1} \Bigr)
\end{gather*}
So in particular, the vector-valued coefficient needs to vanish, and this gives
the indicial equation\index{indicial equation}
\[
\bigl( (q^{1-\mu} -q) - (1-q^\mu) C\bigr)  f_0 =(1-q^\mu)(q^{1-\mu} -C)f_0 = 0
\]
This gives $2N$ for $q^\mu$, namely $q^\mu=1$ and $f_0\in \C^N$ arbitrary 
and $Cf_0=q^{1-\mu} f_0$, i.e. $f_0=f_i$ and $q^{1-\mu}=c_i$ for $1\leq i \leq N$.

With each of these solutions we then need to solve recursively 
\begin{gather*}
\bigl( Cq^{\mu+n} -(C+q) + q^{1-\mu-n}\bigr)f_n = \bigl( ABq^{\mu+n-1} 
-(A+B) + q^{1-\mu-n}\bigr)f_{n-1} \quad \Longrightarrow \\
-(1-q^{\mu+n})(C-q^{1-\mu-n})f_n = (Aq^{\mu+n-1}-1)(B-q^{1-n-\mu})f_{n-1}  \quad \Longrightarrow \\
(1- q^{\mu+n-1}C)f_n = \frac{1}{(1-q^{\mu+n})}(1-q^{\mu+n-1}A)(1-q^{n+\mu-1}B)f_{n-1}
\end{gather*}
since $q^{\mu+n}\not= 1$ under the assumptions on $C$ and $q^{1-\mu}\in \si(C)$. 

In case $q^\mu=1$, we find 
\begin{gather*}
f_n = \frac{1}{(1-q^{n})}(1- q^{n-1}C)^{-1}(1-q^{n-1}A)(1-q^{n-1}B)f_{n-1} 
= \frac{1}{(q;q)_n} (A,B;C;q)_n f_0
\end{gather*}
without condition on $f_0$. This gives the first set of solutions by taking $\mu=0$. 
All other solutions of $q^\mu=1$ lead to the same solution up to $q$-periodic functions, cf. the 
proof of Proposition \ref{prop:solofBqDEat0atinfty}. 

In the other case, we have 
$f_0=f_i$ and $q^{1-\mu}=c_i$ for some  $1\leq i \leq N$. Take $\mu=1-\log_q(c_i)$,
so $q^\mu=q/c_i$. Then the recurrence is 
\begin{gather*}
f_n = \frac{1}{(1-q^{1+n}/c_i)}(1- q^{n}c_i^{-1}C)^{-1}(1-q^{n}c_i^{-1}A)(1-q^{n}c_i^{-1}B)f_{n-1} \\ 
= \frac{1}{(q/c_i;q)_n} (qc_i^{-1}A,qc_i^{-1}B;qc_i^{-1}C;q)_n f_i
\end{gather*}
and this gives the other set of solutions. Again, choosing a different solution of $q^\mu=q/c_i$
leads to the same solution up to a $q$-periodic function. 

Since the space of solutions is $2N$ dimensional, and the set of solutions are linearly independent
we have obtained all solutions. 
The linear independence follows since the first set is linearly independent as analytic solutions
with linearly independent values at $z=0$, and the other solutions all have different 
behaviour as $z\to 0$.
\end{proof}

In order to describe the solutions at $\infty$ we introduce the notation
\begin{equation}\label{eq:MVsolastinfty}
\begin{split}
[A,B;C;\al;q]_n & = \displaystyle\prod_{k=0}^{\substack{n-1\\ \gets}} 
(A-\al q^k)^{-1}(B-\al q^k)^{-1} (C-\al q^k) \\
\Theta^\al(A,B;C;q,z) &= \sum_{n=0}^\infty (\al/q;q)_n [A,B;C;\al;q]_n z^n
\end{split}
\end{equation}
where we assume that all inverses of the matrices involved exist. 

\begin{thm}\label{thm:MVBHSatinfty}
Assume that $A$ and $B$ are diagonalizable with non-zero eigenvalues and such that 
the following genericity conditions on the spectra $\si(A)$, $\si(B)$ hold;
\begin{gather*}
\si(A)\cap \si(B)=\emptyset, \quad \si(A)\cap \si(B)q^{1+\N}=\emptyset, 
\quad \si(B)\cap \si(B)q^{1+\N}=\emptyset,
\end{gather*}
Moreover, let $\si(A)=\{a_1,\cdots, a_N\}$, $\si(B)=\{b_1,\cdots, b_N\}$,
with $Af_i^A=a_i f_i^A$, $Bf_i^A=b_i f_i^B$
then the solutions 
\begin{gather*}
z^{-\log_q(b_i)} \Theta^{qb_i}(B,A;C;q,qz^{-1}) f_i^B, \qquad 1\leq i \leq N, \\
z^{-\log_q(a_i)} \Theta^{qa_i}(B,A;C;q,qz^{-1}) (a_i-B)^{-1} f_i^A, \qquad 1\leq i \leq N, \\
\end{gather*}
are linearly independent solutions of \eqref{eq:vectorvBHqDE}. 
\end{thm}

Note that in case $N=1$, this leads to the solutions $u_3$ and $u_4$ of 
Proposition \ref{prop:solofBqDEat0atinfty}. 

\begin{proof}
Now assume that a solution at $\infty$ has the expansion $\sum_{n=0}^\infty f_n z^{-n-\mu}$ for 
$f_n\in \C^N$.
Plugging this expression in in \eqref{eq:vectorvBHqDE} and rearranging terms we find 
\begin{gather*}
0 = z^{-1-\mu} \bigl( -q^\mu + (A+B) -q^{-\mu} AB\bigr) f_0 
+ \\
\sum_{n=0}^\infty 
z^{-n-\mu} \Bigl( \bigl( q^{n+\mu+1} -(C+q) + q^{-n-\mu}C\bigr) f_n 
+ \bigl( -q^{n+\mu+1} + (A+B) - q^{-1-n-\mu} AB\bigr) f_{n+1} \Bigr).
\end{gather*}
In order to have a solution, the coefficients of the powers of $z$ have to be zero.
The first equation is the indicial equation (for $z=\infty$)\index{indicial equation}
\[
-q^{-\mu}(q^\mu-A)(q^{\mu}-B) f_0 = \bigl( -q^\mu + (A+B) -q^{-\mu} AB\bigr) f_0 = 0
\]
Since $A$ and $B$ are diagonalizable with $N$ different non-zero eigenvalues  
 and $\si(A)\cap \si(B) =\emptyset$,  we
have $2N$ solutions for the indicial equation.
In the first case, $f_0$ is an eigenvector for $B$ with eigenvalue $b_i=q^\mu$, say
$f_0=f_i^B$. So 
\[
q^\mu = b_i, \qquad f_0 = f_i^B, \qquad 1\leq i \leq N. 
\]
In the second case, we find that $(q^\mu-B)f_0$ is an eigenvector of 
$A$ for the eigenvalue $a_i=q^\mu$, i.e. 
\[
q^\mu = a_i, \qquad f_0 = (a_i-B)^{-1} f_i^A, \qquad 1\leq i \leq N,
\]
where $A f_i^A= a_i f_i^A$. 

In the first case we find the recursion 
\begin{gather*}
\bigl( b_iq^{n+1} - (A+B) + b_i^{-1}q^{-1-n} AB\bigr) f_{n+1} = 
\bigl( b_iq^{n+1} -(C+q) + q^{-n}b_i^{-1}C\bigr) f_n \quad \Longrightarrow\\
(b_iq^{n+1}-A)(1-b_i^{-1}q^{-1-n}B) f_{n+1} = (b_iq^{n+1}-C)(1-b_i^{-1}q^{-n}) 
f_n \quad \Longrightarrow\\
(A- b_iq^{n+1}) (B- b_iq^{n+1}) \, f_{n+1} = 
q\, (1-q^nb_i) (C-b_iq^{n+1}) f_n \quad \Longrightarrow\\
f_{n+1} = q\, (1-q^nb_i) 
(B- b_iq^{n+1})^{-1}   (A- b_iq^{n+1})^{-1}  (C-b_iq^{n+1}) f_n \\
% = (-q) (1-q^nb_i) (B- b_iq^{n+1})^{-1}  (A- b_iq^{n+1})^{-1}  (C-b_iq^{n+1}) f_n \\
= q^{n+1}\,  (b_i;q)_{n+1} [B,A;C; qb_i;q]_{n+1} f_0 = 
 q^{n+1}\,  (b_i;q)_{n+1} [B,A;C; qb_i;q]_{n+1}  f_i^B,
\end{gather*}
and this gives the first set of solutions. Other choices of $\mu$ lead to 
the same solution up to a $q$-constant function. 

In the second case we find the recursion 
\begin{gather*}
\bigl( a_iq^{n+1} -(C+q) + a_i^{-1}q^{-n}C\bigr) f_n 
= \bigl( a_iq^{n+1} - (A+B) + a_i^{-1}q^{-1-n} AB\bigr) f_{n+1}   \quad \Longrightarrow\\ 
f_{n+1} = q (1-a_iq^n) (B-a_iq^{n+1})^{-1}(A-a_iq^{n+1})^{-1} (C-a_iq^{n+1}) f_n = \\
q^{n+1}\,  (a_i;q)_{n+1} [B,A;C; qa_i;q]_{n+1}  (a_i-B)^{-1} f_i^A
\end{gather*}
which gives the second solution. 

Since the singularities of the solutions are all different by the genericity assumptions
on the eigenvalues, linear independence follows. 
\end{proof}

It is now a natural question to ask if one can develop an analoguous theory for 
matrix-valued little $q$-Jacobi polynomials using the solutions developed in 
this section. A first attempt is in \cite{AldeKR}. 
For this one needs to study when a matrix-valued 
basic hypergeometric series terminates, and we can directly see that 
\[
{}_2\Phi_1^1(A,B;C;q,z) f
\]
is a polynomial of degree $l$ if $f\not\in \Ker \bigl( (1-q^kA)(1-q^kB)\bigr)$, 
$1\leq k<l$ and $f\in \Ker \bigl( (1-q^lA)(1-q^lB)\bigr)$.

For an analogous theory of the little $q$-Jacobi function, we need to connect 
the solutions of Theorem \ref{thm:MVBHSatzero} with the solutions of 
Theorem \ref{thm:MVBHSatinfty}, i.e. we need the matrix-valued analogue of 
Watson's formulas \eqref{eq:Watsonint2phi1}, \eqref{eq:4.3.2}. 
First results on this approach can be found in \cite{Jasp}.

%%%%%%%%%%%%%%%%%%%%%%%%%%%%%%%%%%%%%%%%%%%%%%%%%%%%%%%%%%%%%%%%%%%
%%%%%NEW SECTION%%%%%%%%%%%%%%%%%%%%%%%%%%%%%%%%%%%%%%%%%%%%%%%%%%%
%%%%%%%%%%%%%%%%%%%%%%%%%%%%%%%%%%%%%%%%%%%%%%%%%%%%%%%%%%%%%%%%%%%

\subsection{Exercises}

\begin{enumerate}[1.]
\item Prove that the series in \eqref{eq:MVbasichypseries} converges in 
$\End(\C^N)$ (equipped with the operator norm) for $|z|<1$, so that 
it defines an analytic function. 
Similarly for the series in \eqref{eq:MVsolastinfty}.
\item Show that the series in \eqref{eq:MVbasichypseries} and 
\eqref{eq:MVsolastinfty} can be written in terms of standard basic hypergeometric
series if we assume that the matrices $A$, $B$ and $C$ pairwise commute.
\item Determine more generally solutions in power series at $0$ and $\infty$ 
for the equation 
\begin{equation*} %\label{eq:vectorvBHqDE}
(q-z)\Id f(q^{-1}z)+\bigl(Uz-C-q\Id\bigr)f(z)+(C-Vz)f(qz) = 0
\end{equation*}
which reduces to \eqref{eq:vectorvBHqDE} in case $V=AB$ and $U=A+B$. 

\end{enumerate}

%%%%%%%%%%%%%%%%%%%%%%%%%%%%%%%%%%%%%%%%%%%%%%%%%%%%%%%%%%%%%%%%%%%
%%%%%NEW SECTION%%%%%%%%%%%%%%%%%%%%%%%%%%%%%%%%%%%%%%%%%%%%%%%%%%%
%%%%%%%%%%%%%%%%%%%%%%%%%%%%%%%%%%%%%%%%%%%%%%%%%%%%%%%%%%%%%%%%%%%

\subsection*{Notes}
A slightly more general situation is considered in \cite[\S 4]{AldeKR},
but then only the analytic solutions are considered. 
Conflitti and Schlosser \cite{ConfS} consider also matrix-valued 
basic hypergeometric $q$-difference equations and hypergeometric
differential equation analogues of Tirao \cite{Tira}, but 
the approach in \cite{ConfS} is different to ours. 
In \cite{AldeKR} a family of $2\times 2$-matrix valued 
little $q$-Jacobi polynomials is considered. 
In \cite{AldeKR} matrix-valued analogues of the 
Askey-Wilson polynomials (of the subclass of the Chebyshev polynomials
of the 2nd kind) are constructed using representation of 
quantum symmetric pairs as a $q$-analogue of \cite{KoelvPR1}, \cite{KoelvPR2}. 
Using non-symmetric Askey-Wilson polynomials
a $2\times 2$-matrix-valued orthogonality is constructed by 
Koornwinder and Mazzoco \cite{KoorM}. It is not clear if these
two approaches can be combined to study the matrix-valued Askey-Wilson
polynomials more generally.
A suitable spectral analysis of the matrix-valued $q$-difference operator
has not been developed. 

%%%%%%%%%%%%%%%%%%%%%%%%%%%%%%%%%%%%%%%%%%%%%%%%%%%%%%%%%%%%%%%%%%%
%%% INDEX %%%%
%%%%%%%%%%%%%%%%%%%%%%%%%%%%%%%%%%%%%%%%%%%%%%%%%%%%%%%%%%%%%%%%%%%

\printindex

%%%%%%%%%%%%%%%%%%%%%%%%%%%%%%%%%%%%%%%%%%%%%%%%%%%%%%%%%%%%%%%%%%%
%%%%%BIBLIOGRAPHY%%%%%%%%%%%%%%%%%%%%%%%%%%%%%%%%%%%%%%%%%%%%%%%%%%
%%%%%%%%%%%%%%%%%%%%%%%%%%%%%%%%%%%%%%%%%%%%%%%%%%%%%%%%%%%%%%%%%%%


\begin{thebibliography}{99}



% \bibitem{Adam-etal-KreinCentenary}
% V.~Adamyan, Y.~Berezansky, I.~Gohberg, M.~Gorbachuk, V.~Gorbachuk, A.~Kochubei, H.~Langer, 
% G.~Popov,
% \emph{Modern Analysis and Applications. The Mark Krein Centenary Conference}, 
% Vol. 1: Operator theory and related topics, Vol. 2: Differential operators and mechanics,
% Operator Theory: Advances and Appl. 190, 191, Birkh\"auser, 2009.

\bibitem{Akhi}
N.I.~Akhiezer,
\emph{The Classical Moment Problem and Some Related
Questions in Analysis}, Hafner, 1965.

\bibitem{AldeKR} 
N.~Aldenhoven,  E.~Koelink,  A.M.~de los R\'\i os,
\emph{Matrix-valued little q-Jacobi polynomials}, 
J. Approx. Theory \textbf{193} (2015), 164--183.

% \bibitem{AldeKR} N.~Aldenhoven, 
% E.~Koelink, P.~Rom\'an, 
% \emph{Matrix-valued orthogonal polynomials related to the quantum analogue of 
% $(\SU(2) \times \SU(2),diag)$}, Ramanujan J. Math., to appear, 
% \texttt{arXiv:1507.03426}. 
% 
% \bibitem{AlSaC} 
% W.A.~Al-Salam, T.S.~Chihara, 
% \emph{Another characterization of the classical orthogonal polynomials}, 
% SIAM J. Math. Anal.  \textbf{3}  (1972) 65--70.

\bibitem{AndrA}
G.E.~Andrews, R.~Askey, 
\emph{Enumeration of partitions: the role of Eulerian series and $q$-orthogonal polynomials},
pp. 3--26 in ``Higher combinatorics'', 
NATO Adv. Study Inst. Ser., Ser. C: Math. Phys. Sci. 31, Reidel, 1977. 

\bibitem{AndrAR}  
G.E.~Andrews, R.A.~Askey, R.~Roy,
\emph{Special Functions}, Cambridge Univ. Press, 1999.

% \bibitem{ApteN} 
% A.I.~Aptekarev, E.M.~Nikishin,
% \emph{The scattering problem for a discrete Sturm-Liouville operator}.
% Mat. USSR Sbornik \textbf{49} (1984), 325--355. 
% 
% \bibitem{Aske} 
% R.~Askey,
% \emph{Continuous $q$-Hermite polynomials when $q>1$}, p.~151--158 in
% ``$q$-Series and Partitions'' (ed. D.~Stanton), IMA Vol. Math. Appl. \textbf{18},
% Springer, 1989.

\bibitem{AskeI} 
R.~Askey, M.E.H.~Ismail, 
\emph{Recurrence relations, continued fractions, and orthogonal polynomials}, 
Mem. Amer. Math. Soc. \textbf{49} (1984), no. 300.

% \bibitem{AskeRS}
% R.A.~Askey, M.~Rahman, S.K.~Suslov, 
% \emph{On a general $q$-Fourier transformation with nonsymmetric kernels}
% J. Comput. Appl. Math. \textbf{68} (1996), 25--55.

\bibitem{AskeW} 
R.~Askey, J.~Wilson, 
\emph{Some basic hypergeometric orthogonal polynomials that generalize Jacobi polynomials}, 
Mem. Amer. Math. Soc. \textbf{54} (1985), no. 319.

\bibitem{Bail}
W.N.~Bailey,
\emph{Generalized Hypergeometric Series}, 
Cambridge Univ. Press, 1935, reprinted Hafner, 1964.

% Little and Big q−Jacobi Polynomials and the Askey-Wilson algebra
% Pascal Baseilhac, Xavier Martin, Luc Vinet, Alexei Zhedanov
% arXiv:1806.02656

\bibitem{Bere}
J.M.~Berezanski\u\i, 
\emph{Expansions in Eigenfunctions of Selfadjoint Operators},
Transl. Math. Monographs 17, Amer. Math. Soc., 1968.

% % \bibitem{Bere-Kreinvol2}
% % Yu.M.~Berezansky, 
% % \emph{Spectral theory of the infinite block Jacobi type normal matrices, orthogonal polynomials on a complex domain, and the complex moment problem}, 
% % pp.~37--50 in Vol.~2 of \cite{Adam-etal-KreinCentenary}.
% 
% \bibitem{Berg-Markov}
% C.~Berg, 
% \emph{Markov's theorem revisited},
% J. Approx. Theory \textbf{78} (1994), 260--275. 
% 
% \bibitem{Berg}
% C.~Berg,
% \emph{The matrix moment problem},
% p. 1--57 in ``Coimbra Lecture Notes on Orthogonal Polynomials'' (eds.
% A.J.P.L Branquinho, A.P. Foulqui\'e Moreno), Nova Science, 2008.
% 
% Braaksma, B. L. J.; Meulenbeld, B.
% Integral transforms with generalized Legendre functions as kernels.
% Compositio Math. 18 1967 235–287 (1967). 

% \bibitem{Broa-LNM} J.T.~Broad, 
% \emph{Extraction of continuum properties from $L^2$ basis set matrix representations of the 
% Schr\"odinger equation: the Sturm sequence polynomials and Gauss quadrature},
% pp. 53--70 in 
% ``Numerical Integration of Differential Equations and Large Linear Systems'' (ed. J.~Hinze),
% LNM 968, Springer, 1982.
% 

\bibitem{BuchC} 
H.~Buchwalter, G.~Cassier, 
\emph{La param\'etrisation de Nevanlinna dans le probl\`eme des moments de Hamburger}, 
Exposition. Math. \textbf{2} (1984), 155--178.

% \bibitem{Bult} 
% F.J.~van de Bult, 
% \emph{Ruijsenaars' hypergeometric function and the modular double of  $U_q(\mathfrak{sl}_2(\mathbf{C}))$},  
% Adv. Math.  \textbf{204}  (2006), 539--571.
% 
% \bibitem{vdB:Rai:Sto} F.J.~van de Bult, E.M.~Rains,  J.V.~Stokman, \emph{Properties of generalized univariate hypergeometric functions},  Comm. Math. Phys.  \textbf{275}  (2007),  37--95.

\bibitem{Chih}
T.S.~Chihara,
\emph{An Introduction to Orthogonal Polynomials},
Math. and its Appl. 13, Gordon and Breach, 1978.

% \bibitem{Chri-PhD}
% J.S.~Christiansen, 
% \emph{Indeterminate moment problems within the Askey-scheme},
% PhD, University of Copenhagen, 2004. 
% 
% \bibitem{ChriK-JAT}
% J.S.~Christiansen, E.~Koelink, 
% \emph{Self-adjoint difference operators and classical solutions to the Stieltjes-Wigert moment problem}, 
% J. Approx. Theory \textbf{140} (2006), 1--26.
% 
% \bibitem{ChriK}
% J.S.~Christiansen, E.~Koelink, 
% \emph{Self-adjoint difference operators and symmetric Al-Salam--Chihara polynomials}, 
% Constr. Approx. \textbf{28} (2008), 199--218.
% 
% \bibitem{CiccKK}
% N.~Ciccoli, E.~Koelink, T.H.~Koornwinder,
% \emph{$q$-Laguerre polynomials and big $q$-Bessel functions
% and their orthogonality relations},  Meth. Appl. Anal.
% {\bf 6} (1999), 109--127.

\bibitem{ConfS}
A.~Conflitti, M.~Schlosser,
\emph{Noncommutative hypergeometric and basic hypergeometric equations},
J. Nonlinear Math. Phys. \textbf{17} (2010), 429--443.


\bibitem{DamaPS} D.~Damanik, A.~Pushnitski, B.~Simon,  
\emph{The analytic theory of matrix orthogonal polynomials},
Surveys in Approx. Th. \textbf{4} (2008), 1--85.

\bibitem{Deif}
P.~Deift,
\emph{Orthogonal Polynomials and Random Matrices: a
Riemann-Hilbert Approach},
Courant Lect. Notes Math. 3, Courant Inst. Math. Sciences,
NY University, 1999.

% 
% % \bibitem{DelvH}
% % S.~Delvaux, H.~Dette, 
% % \emph{Zeros and ratio asymptotics for matrix orthogonal polynomials}, 
% % J. Anal. Math. \textbf{118} (2012), 657--690. 
% % 
% % \bibitem{DiesU}
% % J.~Diestel, J.J.~Uhl, Jr, 
% % \emph{Vector Measures},
% % Math. Surveys 15, AMS, 1977.
% 
% \bibitem{Dies} D.J.~Diestler, 
% \emph{The discretization of continuous infinite sets of coupled ordinary linear differential equations: application 
% to the collision-induced dissociation of a diatomic molecule by an atom}, 
% pp.  40--52 in 
% ``Numerical Integration of Differential Equations and Large Linear Systems'' (ed. J.~Hinze),
% LNM 968, Springer, 1982.
% 
% \bibitem{Dieu}
% J.~Dieudonn\'e, 
% \emph{History of Functional Analysis}, 
% North-Holland Math. Stud. \textbf{49}, North-Holland, 1981. 
% 
% % \bibitem{Domb}
% % J.~Dombrowski, \emph{Orthogonal polynomials and functional
% % analysis}, ``Orthogonal Polynomials: Theory and Practice''
% % (P.~Nevai, ed.),
% % NATO Adv. Sci. Inst. Ser. C Math. Phys. Sci. 294,
% % Kluwer, 1990, pp.~147--161.

\bibitem{DunfS}
N.~Dunford, J.T.~Schwartz,
\emph{Linear Operators II: Spectral Theory},
Interscience, 1963.

% \bibitem{Dura-LN} A.J.~Dur\'an, 
% \emph{Exceptional orthogonal polynomials}, 
% lecture notes in this summerschool. 
% 
% \bibitem{Dura-JAT1999}
% A.J.~Dur\'an, 
% \emph{Ratio asymptotics for orthogonal matrix polynomials}, 
% J. Approx. Theory \textbf{100} (1999), 304--344. 
% 
% % \bibitem{DuraLR-JAT1997}
% % A.J.~Dur\'an, P.~Lopez-Rodriguez,
% % ??JAT 1997
% 
% \bibitem{DuraLR-Laredo}
% A.J.~Dur\'an, P.~L\'opez-Rodr\i\i guez, 
% \emph{Orthogonal matrix polynomials}, 
% pp. 13--44 in ``Laredo Lectures on Orthogonal Polynomials and Special Functions'' 
% (eds. R. \'Alvarez-Nodarse, F. Marcell\'an, W. Van Assche), 
% Nova Science Publishers, 2004.
% 
% \bibitem{DuraVA} A.J.~Dur\'an, W.~Van Assche,
% \emph{Orthogonal matrix polynomials and higher-order recurrence relations},
% Linear Algebra Appl. \textbf{219} (1995), 261--280. 
% 
% % \bibitem{HTFeen}
% % A.~Erd\'elyi, W.~Magnus, F.~Oberhettinger, F.G.~Tricomi,
% % \emph{Higher Transcendental Functions}, Vol.~1, McGraw-Hill,
% % 1953.
% 
% \bibitem{Fara} 
% J.~Faraut, 
% \emph{Un th\'eor\`eme de Paley-Wiener pour la transformation 
% de Fourier sur un espace riemannien sym\'etrique de rang un}, 
% J. Funct. Anal. \textbf{49} (1982), 230--268. 

\bibitem{FitoD2}
A.~Fitouhi, L.~Dhaouadi, 
\emph{On a $q$-Paley-Wiener theorem}
J. Math. Anal. Appl. \textbf{294} (2004), 17--23. 

\bibitem{FitoD}
A.~Fitouhi, L.~Dhaouadi, 
\emph{Positivity of the generalized translation associated with the $q$-Hankel transform}, 
Constr. Approx. \textbf{34} (2011), 453--472.

\bibitem{Gasp} 
G.~Gasper,
\emph{$q$-Extensions of Erd\'elyi's fractional integral
representations for hypergeometric functions and some summation
formulas for double $q$-Kamp\'e de F\'eriet series}, 
Contemp. Math. \textbf{254} (2000), 187--198. 


\bibitem{GaspR}
G.~Gasper, M.~Rahman, \emph{Basic Hypergeometric Series}, 2nd ed.,
Cambridge Univ. Press, 2004.

% 
% \bibitem{GeneIVZ-PAMS2016}
% V.X.~Genest, M.E.H.~Ismail, L.~Vinet, A.~Zhedanov,
% \emph{Tridiagonalization of the hypergeometric operator and the Racah-Wilson algebra}, 
% Proc. Amer. Math. Soc., to appear \texttt{1506.07803}
% 
% \bibitem{Gero}
% J.S.~Geronimo, 
% \emph{Scattering theory and matrix orthogonal polynomials on the real line}, 
% Circuits Systems Signal Process.
% \textbf{1} (1982), 471--495.

\bibitem{Groe-WT}
W.~Groenevelt, 
\emph{The Wilson function transform}
Int. Math. Res. Not. \textbf{2003} (2003), 2779--2817. 

% \bibitem{Groe-Laguerre}
% W.~Groenevelt, 
% \emph{Laguerre functions and representations of
% ${\mathfrak{su}}(1,1)$}, Indagationes Math. N.S. {\bf 14} (2003), 
% 329--352. 
% 
% \bibitem{Groe-PhD}
% W.~Groenevelt, 
% \emph{Tensor product representations and special functions},
% PhD-thesis, Delft University of Technology, 2004. 

\bibitem{Groe-Ramanujan}
W.~Groenevelt, 
\emph{Bilinear summation formulas from quantum algebra representations}, 
Ramanujan J. \textbf{8} (2004), 383--416.

\bibitem{Groe-CA2009}
W.~Groenevelt, 
\emph{The vector-valued big q-Jacobi transform},
Constr. Approx. \textbf{29} (2009), 85--127. 

% \bibitem{Groe-JMP2014}
% W.~Groenevelt,  
% \emph{Coupling coefficients for tensor product representations of quantum $SU(2)$},  
% J. Math. Phys. \textbf{55} (2014),  101702, 35 pp.
% 
% \bibitem{GroeIK}
% W.~Groenevelt, M.E.H.~Ismail, E.~Koelink, 
% \emph{Spectral decomposition and matrix-valued orthogonal polynomials},
% Adv. Math. \textbf{244} (2013), 91--105. 
% 
% \bibitem{GroeK-JPA2002}
% W.~Groenevelt, E.~Koelink, 
% \emph{Meixner functions and polynomials related to Lie algebra
% representations}, J. Phys. A: Math. Gen. {\bf 35} (2002), 65--85.
% 
% \bibitem{GroeK-Meixner}
% W.~Groenevelt, E.~Koelink, 
% \emph{The indeterminate moment problem for the $q$-Meixner polynomials}, 
% J. Approx. Theory \textbf{163} (2011), 838--863.
% 
% \bibitem{GroeK}
% W.~Groenevelt, E.~Koelink, 
% \emph{A hypergeometric function transform and matrix-valued orthogonal polynomials},
% Constr. Approx. \textbf{38} (2013), 277--309.

\bibitem{GroeKK}
W.~Groenevelt, E.~Koelink, J.~Kustermans,  
\emph{The dual quantum group for the quantum group analog of the normalizer of $SU(1,1)$ in $SL(2,\C)$}, 
Int. Math. Res. Not. IMRN \textbf{2010} (2010), no. 7, 1167--1314.

% \bibitem{GroeKR}
% W. Groenevelt, E. Koelink, H. Rosengren, 
% \emph{Continuous Hahn functions as Clebsch-Gordan coefficients}, 
% pp. 221--284 in ``Theory and Applications of Special Functions. A Volume 
% Dedicated to Mizan Rahman'' (eds. M. E. H. Ismail, E. Koelink), 
% Developments in Mathematics, Vol. 13, Kluwer, 2005. 

\bibitem{GrunPT} F.A.~Gr{\"u}nbaum, I.~Pacharoni, J.~Tirao,
\emph{Matrix valued spherical functions associated to the complex projective plane}, 
J. Funct. Anal. \textbf{188} (2002), 350--441.

% \bibitem{GrunT} 
% F.A.~Gr\"unbaum, J.~Tirao,
% \emph{The algebra of differential operators associated to a weight matrix}, 
% Integral Eq. Operator Theory \textbf{58} (2007), 449--475.
% 
% \bibitem{GuptaIsmailMasson1992} 
% D.P.~Gupta, M.E.H.~Ismail, D.R.~Masson, 
% \emph{Contiguous relations, basic hypergeometric functions, and orthogonal polynomials. II. Associated big $q$-Jacobi polynomials},  J. Math. Anal. Appl.  \textbf{171}  (1992), 477--497.
% 
% \bibitem{Gup:Ism:Mas96} 
% D.P.~Gupta, M.E.H.~Ismail, D.R.~Masson, 
% \emph{Contiguous relations, basic hypergeometric functions, and orthogonal polynomials. III. Associated continuous dual $q$-Hahn polynomials},  J. Comput. Appl. Math.  \textbf{68}  (1996), 115--149. 
% 

\bibitem{HainI}
L.~Haine, P.~Iliev, 
\emph{Askey-Wilson type functions with bound states},
Ramanujan J. \textbf{11} (2006), 285--329. 

% \bibitem{HornJ}
% R.A.~Horn, C.R.~Johnson, 
% \emph{Matrix Analysis}, Cambridge Univ. Press, 1985.
% 
% \bibitem{Ilie}
% P.~Iliev, 
% \emph{Bispectral extensions of the Askey-Wilson polynomials}, 
% J. Funct. Anal. \textbf{266} (2014), 2294--2318. 

\bibitem{Isma} M.E.H.~Ismail, 
\emph{Classical and Quantum Orthogonal Polynomials in One Variable}, 
Cambridge Univ. Press,  2009. 

\bibitem{Isma-LN} M.E.H.~Ismail, 
\emph{A brief review of $q$-series}, 
lecture notes 2016 OPSFA summerschool. 

% 
% \bibitem{IsmaK} 
% M.E.H.~Ismail, E.~Koelink,
% \emph{The $J$-matrix method}, 
% Adv. in Appl. Math. \textbf{46} (2011), 379--395.
% 
% \bibitem{IsmaK-AA} 
% M.E.H.~Ismail, E.~Koelink,
% \emph{Spectral properties of operators using tridiagonalization}, 
% Anal. Appl. (Singap.) \textbf{10} (2012),  327--343. 
% 
% \bibitem{IsmaK-SIGMA} 
% M.E.H.~Ismail, E.~Koelink,
% \emph{Spectral analysis of certain Schr\"odinger operators}, 
% SIGMA Symmetry Integrability Geom. Methods Appl. \textbf{8} (2012), Paper 061, 19 pp.
% 
% \bibitem{IsmaM} M.E.H.~Ismail, D.R.~Masson,
% \emph{$q$-Hermite polynomials, biorthogonal rational functions, and $q$-beta integrals},
% Trans. Amer. Math. Soc. \textbf{346} (1994), 63--116.

\bibitem{IsmaR-TAMS} 
M.E.H.~Ismail, M.~Rahman, 
\emph{The associated Askey-Wilson polynomials}, 
Trans. Amer. Math. Soc. \textbf{328} (1991), 201--237. 

% \bibitem{IsmaR-PMJ} 
% M.E.H.~Ismail, M.~Rahman, 
% \emph{Some basic bilateral sums and integrals},  
% Pacific J. Math.  \textbf{170}  (1995), 497--515.

\bibitem{IsmaS} 
M.E.H.~Ismail, D.~Stanton, 
\emph{$q$-Taylor theorems, polynomial expansions, and interpolation of entire functions},
J. Approx. Theory \textbf{123} (2003), 125--146. 

\bibitem{Jasp}
N.~Jaspers, 
\emph{Connection formulas for hypergeometric series}, 
BSc-thesis, Radboud Universiteit,  2017.

% 
% % \bibitem{Kac}
% % I.~Kac, ??

\bibitem{Kake}
T.~Kakehi,
\emph{Eigenfunction expansion associated with the Casimir
operator on the quantum group $SU_q(1,1)$},
Duke Math. J. {\bf 80} (1995), 535--573.

% \bibitem{Kjel}
% T.H.~Kjeldsen, \emph{The early history of the moment problem},
% Historia Math. {\bf 20} (1993), 19--44.
% 

\bibitem{KoekLS} R.~Koekoek, P.A.~Lesky, R.F.~Swarttouw, 
\emph{Hypergeometric Orthogonal Polynomials and their $q$-Analogues},
Springer, 2010.

\bibitem{KoekS} R.~Koekoek, R.F.~Swarttouw, 
\emph{The Askey-scheme
of hypergeometric orthogonal polynomials and its $q$-analogue}, online at
\texttt{http://aw.twi.tudelft.nl/\~{}koekoek/askey.html}, Report
98-17, Technical University Delft, 1998. 

\bibitem{Koel-ITSF93} 
H.T.~Koelink, 
\emph{A basic analogue of Graf's addition formula and related formulas}, 
Integral Transform. Spec. Funct. \textbf{1} (1993), 165--182. 

\bibitem{Koel-IM}
E.~Koelink, 
\emph{One-parameter orthogonality relations for basic hypergeometric series},
Indag. Math. (N.S.)  \textbf{14} (2003), 423--443.

\bibitem{Koel-Laredo}
E.~Koelink, 
\emph{Spectral theory and special functions}, 
pp. 45--84 in ``Laredo Lectures on Orthogonal Polynomials and Special Functions'' 
(eds. R. \'Alvarez-Nodarse, F. Marcell\'an, W. Van Assche), 
Nova Science Publishers, 2004.

\bibitem{Koel-OPSFA2016}
E.~Koelink, 
\emph{Applications of spectral theory to special functions}, 
to appear in Lecture Notes of the London Math. Soc. (ed. H. Cohl, M.E.H. Ismail), Cambridge U. Press, 
lecture notes for OPSFA Summerschool 6, July 2016, \texttt{arXiv:1612.07035}. 

% \bibitem{KoelK}
% E.~Koelink, J.~Kustermans,  
% \emph{A locally compact quantum group analogue of the normalizer of $SU(1,1)$ in $SL(2,\C)$}, 
% Comm. Math. Phys. \textbf{233} (2003), 231--296.

\bibitem{KoelvPR1} E.~Koelink, M.~van Pruijssen, P.~Rom\'an, 
\emph{Matrix valued orthogonal polynomials related to $(\SU(2)\times\SU(2),\text{diag})$}, 
Int. Math. Res. Not. \textbf{2012} (2012),  5673--5730.

\bibitem{KoelvPR2} E.~Koelink, M.~van Pruijssen, P.~Rom\'an, 
\emph{Matrix valued orthogonal polynomials related to $(\SU(2)\times\SU(2),\text{diag})$, II}, 
Publ. RIMS Kyoto \textbf{49} (2013), 271--312. 

\bibitem{KoeldlRR}
E.~Koelink, A.M.~de los R\'\i os, P.~Rom\'an, 
\emph{Matrix-valued Gegenbauer-type polynomials}, 
Constr. Approx.  \textbf{46} (2017), 459--487.

% \bibitem{KoelP}
% E.~Koelink, P.~Rom\'an, 
% \emph{Orthogonal vs. non-orthogonal reducibility of matrix-valued measures}, 
% SIGMA Symmetry Integrability Geom. Methods Appl. \textbf{12} (2016), Paper 008, 9 pp.

\bibitem{KoelR-RMJM02} E.~Koelink, H.~Rosengren, 
\emph{Transmutation kernels for the little $q$-Jacobi function transform}, 
Rocky Mountain J. Math. \textbf{32} (2002), 703--738. 

\bibitem{KoelS-CA}
E.~Koelink, J.V.~Stokman,
\emph{The big $q$-Jacobi function transform},
Constr. Approx. {\bf 19} (2003), 191--235.

\bibitem{KoelS-PublRIMS}
E.~Koelink, J.V.~Stokman, with an appendix by M.~Rahman,
\emph{Fourier transforms on the quantum $SU(1,1)$ group},
Publ. Res. Inst. Math. Sci., Kyoto Univ. {\bf 37} (2001), 621-715.

\bibitem{KoelS-NATO}
E.~Koelink, J.V.~Stokman,
\emph{The Askey-Wilson function transform scheme},
pp. 221-241 
in ``Special Functions 2000: 
Current Perspective and Future Directions''  
(eds. J. Bustoz, M.E.H. Ismail, S.K. Suslov), 
NATO Science Series II, Vol. 30, Kluwer, 2001.

\bibitem{KoelS-IMRN2001}
E.~Koelink, J.V.~Stokman,
\emph{The Askey-Wilson function transform},
Intern. Math. Res. Notices {\bf 2001}, 22, 1203-1227.

\bibitem{KoelSw}
H.T.~Koelink, R.F.~Swarttouw, 
\emph{On the zeros of the Hahn-Exton $q$-Bessel function and associated $q$-Lommel polynomials}
J. Math. Anal. Appl. \textbf{186} (1994), 690--710. 

% 
% \bibitem{KoelVdJ}
% H.T.~Koelink, J.~Van Der Jeugt, 
% \emph{Convolutions for orthogonal polynomials from Lie and quantum algebra representations},  
% SIAM J. Math. Anal. \textbf{29} (1998), 794--822. 
% 
% \bibitem{KoelVdJ-CA}
% H.T.~Koelink, J.~Van Der Jeugt, 
% \emph{Bilinear generating functions for orthogonal polynomials},
% Constr. Approx. \textbf{15} (1999), 481--497.

\bibitem{Kolb}
S.~Kolb, {Quantum symmetric Kac-Moody pairs}, 
Adv. Math. \textbf{267} (2014), 395--469.

\bibitem{Koor-ArkMat75}
T.~Koornwinder, 
\emph{A new proof of a Paley-Wiener type theorem for the Jacobi transform}, 
Ark. Mat. \textbf{13} (1975), 145--159.

\bibitem{Koor-Jacobi}
T.H.~Koornwinder, 
\emph{Jacobi functions and analysis on noncompact semisimple Lie groups},
pp.~1--85 in ``Special Functions: Group Theoretical Aspects and Applications''
(eds. R.A.~Askey, T.H.~Koornwinder, W.~Schempp)
Math. Appl., Reidel, 1984. 

% \bibitem{Koor-LNM} T.H.~Koornwinder, 
% \emph{Special orthogonal polynomial systems mapped onto each other by the Fourier-Jacobi transform}, pp. 174--183 in ``Orthogonal Polynomials and Applications''
% (eds. C.~Brezinski, A.~Draux, A.P.~Magnus, P.~Maroni, A.~Ronveaux), LNM 1171, Springer,  1985.
% 
% \bibitem{Koor-SIAM}
% T.H.~Koornwinder, 
% \emph{Matrix elements of irreducible representations of $\SU(2)\times\SU(2)$
% and vector-valued orthogonal polynomials}, 
% SIAM J. Math. Anal.
% \textbf{16} (1985), 602--613.

\bibitem{Koor-CQGSF} 
T.H.~Koornwinder, 
\emph{Compact quantum groups and $q$-special functions}, 
p. 46--128 in 
``Representations of Lie groups and quantum groups'' (eds. V.~Baldoni, M.A.~Picardello),
Pitman Res. Notes Math. Ser., 311, Longman Sci. Tech., 1994.


\bibitem{KoorM}
T.H.~Koornwinder, M.~Mazzocco,
\emph{Dualities in the $q$-Askey scheme and degenerated DAHA},
\texttt{arXiv:1803.02775}. 

\bibitem{KoorS-TAMS}
T.H.~Koornwinder, R.F.~Swarttouw, 
\emph{On $q$-analogues of the Fourier and Hankel transforms}, 
Trans. Amer. Math. Soc. \textbf{333} (1992), 445--461.

% 
% \bibitem{Krei-mm} 
% M.G.~Kre\u\i n, 
% \emph{The fundamental propositions of the theory of representations of Hermitian operators with deficiency index $(m,m)$},
% Ukrain. Mat. \v{Z}urnal \textbf{1} (1949),  3--66. English translation in 
% {AMS Translations ser. 2} 97 (1970), 75--143.
% 
% \bibitem{Krei} 
% M.~Kre\u\i n, 
% \emph{Infinite J-matrices and a matrix-moment problem},
% Doklady Akad. Nauk SSSR (N.S.) \textbf{69} (1949), 125--128. 
% (English translation by W. Van Assche at \texttt{1606.07754}.)
% 
% \bibitem{Labe}
% J.~Labelle, 
% \emph{Tableau d'Askey}, 
% pp. xxxvi--xxxvii in ``Orthogonal Polynomials and Applications (Bar-le-Duc, 1984)'' 
% (eds. C.~Brezinski, A.~Draux, A.P.~Magnus, P.~Maroni, A.~Ronveaux), 
% LNM 1171, Springer, 1985. 
% 
% \bibitem{Lanc}
% E.C.~Lance, 
% \emph{Hilbert $\text{C}^\ast$-modules. A toolkit for operator algebraists}, 
% London Math. Soc. Lecture Note Series 210, Cambridge University Press, 1995. 
% 
% \bibitem{Land}
% H.J.~Landau, 
% \emph{The classical moment problem: Hilbertian proofs}, 
% J. Funct. Anal. \textbf{38} (1980), 255--272. 
% 
% % \bibitem{Lasa}
% % A.~Lasarow, JAT 2011??
% 
% \bibitem{Lax}
% P.D.~Lax, 
% \emph{Functional Analysis}, 
% Wiley-Interscience, 2002.
% 
% 
% \bibitem{Lubi} 
% D.S.~Lubinsky, 
% \emph{The size of $(q;q)\sb n$ for $q$ on the unit circle},  
% J. Number Theory  \textbf{76}  (1999),  217--247.
% 
% % \bibitem{Mass}
% % D.R.~Masson,
% % \emph{Difference equations, continued fractions,
% % Jacobi matrices and orthogonal polynomials},
% % ``Nonlinear Numerical Methods and Rational Approximation''
% % (A.~Cuyt, ed.), Reidel, 1988, pp.~239--257.

\bibitem{MassR}
D.R.~Masson, J.~Repka,
\emph{Spectral theory of Jacobi matrices in $\ell^2(\Z)$
and the ${\mathfrak{su}}(1,1)$ Lie algebra},
SIAM J. Math. Anal. {\bf 22} (1991), 1131--1146.

% \bibitem{Monn}
% A.F.~Monna, 
% \emph{Functional Analysis in Historical Perspective},
% Wiley, 1973. 
% 
% \bibitem{Nere}
% Yu.A.~Neretin, 
% \emph{Some continuous analogues of the expansion in Jacobi polynomials, 
% and vector-valued orthogonal bases},
% Funct. Anal. Appl. \textbf{39} (2005), 106--119.
% 

\bibitem{NoumS}
M.~Noumi, J.V.~Stokman,
\emph{Askey-Wilson polynomials: an affine Hecke algebra approach},
pp. 111--144 in ``Laredo Lectures on Orthogonal Polynomials and Special Functions'' 
(eds. R. \'Alvarez-Nodarse, F. Marcell\'an, W. Van Assche), 
Nova Science Publishers, 2004.

\bibitem{Olve}
F.W.J.~Olver, 
\emph{Asymptotics and Special Functions},
AKP Classics, A.K. Peters, 1997.

% \bibitem{PachZ}
% I.~Pacharoni, I.~Zurri\'an, 
% \emph{Matrix Gegenbauer polynomials: the $2\times 2$ fundamental cases} 
% Constr. Approx. \textbf{43} (2016), 253--271.
% 
% % \bibitem{RaebW}
% % I.~Raeburn, D.P.~Williams, 
% % \emph{Morita Equivalence and Continuous-trace $\text{C}^\ast$-algebras},
% % Math. Surveys 60, AMS, 1998.
% 

\bibitem{Rain} 
E.D.~Rainville, \emph{Special Functions}, Macmillan, 1960. 

% \bibitem{ReedS}
% M.~Reed, B.~Simon,  
% \emph{Methods of Modern Mathematical Physics. I. Functional Analysis},  
% Academic Press, 1972.
% 
% % \bibitem{Rodm}
% % L.~Rodman,
% % \emph{Orthogonal matrix polynomials}, 
% % p.  345--362 in ``Orthogonal polynomials (Columbus, OH, 1989)'' (eds. P.~Nevai),
% % NATO Adv. Sci. Inst. Ser. C Math. Phys. Sci. \textbf{294}, Kluwer 1990. 
% 
% \bibitem{Rose}
% M.~Rosenberg, 
% \emph{The square-integrability of matrix-valued functions with respect to a non-negative Hermitian measure}, 
% Duke Math. J. \textbf{31} (1964), 291--298. 
% 

\bibitem{RomaS}
P.~Rom\'an, S.~Simondi, 
\emph{Solutions at infinity of the generalized matrix-valued hypergeometric equation}, 
Appl. Math. Lett. \textbf{23} (2010), 39--43.

\bibitem{Rose}
H.~Rosengren, 
\emph{A new quantum algebraic interpretation of the Askey-Wilson polynomials},
Contemp. Math. \textbf{254} (2000), 371--394.

% % \bibitem{Rota}
% % G.C.~Rota, 
% % \emph{Indiscrete Thoughts}, ??
% 
% \bibitem{Roqu} J.~Roques, 
% \emph{Galois groups of the basic hypergeometric equations},  
% Pacific J. Math.  \textbf{235}  (2008), 303--322.

% \bibitem{Rudi-RCA}
% W.~Rudin, 
% \emph{Real and Complex Analysis}, 
% McGraw-Hill, 1966.

\bibitem{Rudi}
W.~Rudin,
\emph{Functional Analysis},
McGraw-Hill, 1973.

% \bibitem{Ruij-JMathPhys1997} 
% S.N.M.~Ruijsenaars, 
% \emph{First order analytic difference equations and integrable quantum systems}, 
% J. Math. Phys.  \textbf{38}  (1997), 1069--1146.
% 

\bibitem{Ruij-AW} S.N.M.~Ruijsenaars, 
\emph{A generalized hypergeometric function satisfying four analytic difference equations of Askey-Wilson type}.
Comm. Math. Phys. \textbf{206} (1999), 639--690; 
\emph{A generalized hypergeometric function. II. Asymptotics and $D\sb 4$ symmetry}, 
\emph{III. Associated Hilbert space transform}, Comm. Math. Phys.\textbf{243} (2003),  389--412, 413--448. 

% % 
% \bibitem{Sche} 
% M.~Schechter, 
% \emph{Operator Methods in Quantum Mechanics}, North-Holland, New York, 1981. 
% 
% \bibitem{Schm}
% K.~Schm\"udgen,
% \emph{Unbounded Self-adjoint Operators on Hilbert Space},
% GTM 265, Springer, 2012.
% %Ch. 16
% 
% \bibitem{ShohT}
% J.A.~Shohat, J.D.~Tamarkin,
% \emph{The Problem of Moments},
% Math. Surveys 2, AMS, 1943.
% 

\bibitem{Simo}
B.~Simon,
\emph{The classical moment problem as a self-adjoint finite
difference operator},
Adv. Math. {\bf 137} (1998), 82--203.

% \bibitem{Simo-book}
% B.~Simon, 
% \emph{Szeg\H o's Theorem and its Descendants. Spectral Theory for $L^2$ perturbations of orthogonal polynomials},
% Princeton Univ. Press, 2011. 
% 

\bibitem{Slat}
L.J.~Slater, 
\emph{Generalized Hypergeometric Functions}, Cambridge Univ. Press, 1966. 

% \bibitem{Stie}
% T.J.~Stieltjes,
% \emph{Recherches sur les fractions continues},
% Annales de la Facult\'e des Sciences de Toulouse
% {\bf 8} (1894), J.1--122, {\bf 9} (1895), A.1--47,
% reprinted in ``\OE uvres Compl\`etes-Collected Papers'',
% vol.~II (ed. G.~van Dijk), Springer Verlag, 1993, pp.~406--570.

\bibitem{Stok-JAT} 
J.V.~Stokman,
\emph{An expansion formula for the Askey-Wilson function},
J. Approx. Theory \textbf{114} (2002), 308--342. 

% \bibitem{Stok-AdvM2005} 
% J.V.~Stokman,
% \emph{Hyperbolic beta integrals}, 
% Adv. Math. \textbf{190} (2005), 119--160. 
%  

\bibitem{Stok-AoM} 
J.V.~Stokman,
\emph{The $c$-function expansion of a basic hypergeometric function associated to root systems}
Ann. of Math. (2) \textbf{179} (2014), 253--299. 

% 
% \bibitem{Ston}
% M.H.~Stone,
% \emph{Linear Transformations in Hilbert Space},
% AMS Colloq. Publ. 15, AMS, 1932.
% 

\bibitem{Susl}
S.K.~Suslov, 
\emph{Some orthogonal very-well-poised ${}_8\phi_7$-functions that generalize Askey-Wilson polynomials}, 
Ramanujan J. \textbf{5} (2001), 183--218. 

% \bibitem{Szeg} 
% G.~Szeg\H o, 
% \emph{Orthogonal Polynomials}, 4th ed., AMS Colloquium Publ. \textbf{23}, AMS, 1975.

\bibitem{Temm}
N.M.~Temme, 
\emph{Special Functions}, 
Wiley, 1996.

% 
% % \bibitem{Tesc}
% % G.~Teschl,
% % \emph{Jacobi Operators and Completely Integrable Nonlinear
% % Lattices}, Math. Surveys Monographs 72, AMS, 2000.
% 

\bibitem{Tira}
J.~Tirao, 
\emph{The matrix-valued hypergeometric equation}, 
Proc. Natl. Acad. Sci. USA \textbf{100} (2003), 8138--8141.

% \bibitem{TiraZ}
% J.~Tirao, I. Zurri\'an.
% \emph{Reducibility of matrix weights}, 
% \texttt{arXiv:1501.04059v4}.
% 

\bibitem{Vaks}
L.L.~Vaksman, 
\emph{Quantum Bounded Symmetric Domains},
Transl. Math. Monographs 238, AMS, 2010. 

% % \bibitem{VAssc}
% % W.~Van Assche, \emph{ Orthogonal polynomials in the
% % complex plane and on the real line},
% % ``Special Functions, $q$-Series and Related Topics''
% % (M.E.H.~Ismail, D.R.~Masson and M.~Rahman, eds.),
% % Fields Inst. Commun. 14, AMS, 1997, pp.~211--245.
% 
% \bibitem{Wern}
% D.~Werner,
% \emph{Funktionalanalyse}, 4th ed., Springer, 2002. 
% 
% \bibitem{Wien}
% N.~Wiener, 
% \emph{The Fourier Integral and Certain of its Applications},
% Cambridge Univ. Press, 1933. 

\end{thebibliography}
\end{document}